\documentclass[reqno,11pt]{amsart}

\usepackage{amsmath}
\usepackage{amsthm}
\usepackage{amsfonts}
\usepackage{amssymb}
\usepackage{enumerate}
\usepackage{graphicx}
\usepackage{xcolor}
\usepackage{verbatim}
\usepackage{fullpage}
\usepackage[nocompress]{cite}
\usepackage{bbm}

\newcommand\details[1]{}  

\newcommand\abs[1]{\left|#1\right|}
\newcommand\norm[1]{\left\|#1\right\|}

\numberwithin{equation}{section}

\DeclareMathOperator{\tr}{tr}

\DeclareMathOperator{\rank}{rank}
\DeclareMathOperator{\supp}{supp}

\newcommand{\Prob}{\mathbb{P}}
\newcommand{\E}{\mathbb{E}}

\renewcommand\Re{\operatorname{Re}}
\renewcommand\Im{\operatorname{Im}}
\newcommand{\eps}{\varepsilon}

\renewcommand{\d}{\, d }

\newcommand\transpose[1]{#1^{\mathrm T}}

\newcommand{\pfrac}[2]{\left(\frac{#1}{#2}\right)}

\theoremstyle{plain}
  \newtheorem{theorem}{Theorem}[section]
  \newtheorem{conjecture}[theorem]{Conjecture}
  \newtheorem{lemma}[theorem]{Lemma}
  \newtheorem{corollary}[theorem]{Corollary}
  
  \newtheorem{proposition}[theorem]{Proposition}
  
\theoremstyle{definition}
  \newtheorem{definition}[theorem]{Definition}
  \newtheorem{example}[theorem]{Example}

  \newtheorem{question}[theorem]{Question}
  
\theoremstyle{remark}
  \newtheorem{remark}[theorem]{Remark}

\newcommand\sixthree{C_{\ref{lemma:intbnd}}}
\newcommand\sixthreed{D_{\ref{lemma:intbnd}}}
\newcommand\Csixone{C_{\ref{thm:non}}}

\newcommand\Jmat{\mathbb{J}}
\newcommand\pone{R}
\newcommand\ptwo{P}

\begin{document}
\title{Quantitative results for banded Toeplitz matrices subject to random and deterministic perturbations}

\author[S. O'Rourke]{Sean O'Rourke}
\address{Department of Mathematics, University of Colorado at Boulder, Boulder, CO 80309  }
\email{sean.d.orourke@colorado.edu}

\author[P. Wood]{Philip Matchett Wood}
\thanks{The first author has been partially supported in part by National Science Foundation grant number DMS-1810500.  The second author was partially supported by National Security Agency (NSA) Young Investigator Grant number H98230-14-1-0149.} 
\address{Department of Mathematics, Harvard University, Science Center Room 325
1 Oxford Street, Cambridge, MA 02138
}
\email{pmwood@math.harvard.edu}

\begin{abstract}
We consider the eigenvalues of a fixed, non-normal matrix subject to a small additive perturbation.  In particular, we consider the case when the fixed matrix is a banded Toeplitz matrix, where the bandwidth is allowed to grow slowly with the dimension, and the perturbation matrix is drawn from one of several different random matrix ensembles.  We establish a number of non-asymptotic results for the eigenvalues of this model, including a local law and a rate of convergence in Wasserstein distance of the empirical spectral measure to its limiting distribution.  In addition, we define the \emph{classical locations} of the eigenvalues and prove a rigidity result showing that, on average, the eigenvalues concentrate closely around their classical locations.  While proving these results we also establish a number of auxiliary results that may be of independent interest, including a quantitative version of the Tao--Vu replacement principle, a general least singular value bound that applies to adversarial models, and a description of the limiting empirical spectral measure for random multiplicative perturbations.  
\end{abstract}

\maketitle


\section{Introduction}

Due to the spectral theorem, the eigenvalues of Hermitian matrices are stable under small perturbations.  For example, when $A$ and $B$ are $n \times n$ Hermitian matrices, Weyl's perturbation theorem (see \cite[Corollary III.2.6]{Bhatia}) guarantees that
\begin{equation} \label{eq:weyleig}
	\max_{ 1 \leq j \leq n} |\lambda_j(A) - \lambda_j(B)| \leq \|A-B\|,
\end{equation} 
where $\lambda_1(M) \geq \cdots \geq \lambda_n(M)$ are the ordered eigenvalues of the $n \times n$ Hermitian matrix $M$ and $\|M\|$ is its spectral norm.  In contrast, the spectrum of a non-normal matrix can be extremely sensitive to small perturbations if there is pseudospectrum present \cite{MR2155029}.  Consider the case of the $n \times n$ matrix
\begin{equation} \label{eq:defT}
	\Jmat := \begin{bmatrix} 0 & 1 & 0 & 0 & \cdots & 0 \\ 0 & 0 & 1 & 0 & \cdots & 0 \\ \vdots &\vdots & \ddots & \ddots & \ddots & \vdots \\ \vdots & \vdots &  & \ddots & \ddots & 0 \\ 0 & 0 & \cdots & \cdots & 0 & 1 \\ 0 & 0 & \cdots & \cdots & \cdots & 0 \end{bmatrix} 
\end{equation} 
with ones on the super-diagonal and zeros everywhere else.  If  $\eps > 0$ and if $\ptwo$ is the $n \times n$ matrix with 1 in the $(n, 1)$-entry and zeros everywhere else, then the eigenvalues of $\Jmat +  \eps \ptwo$ lie on the circle $\{z \in \mathbb{C} :  |z| = \eps^{1/n} \}$ in the complex plane, while all eigenvalues of $\Jmat$ are zero.  In fact, by writing out the characteristic polynomial, it is not difficult to see that the eigenvalues of $\Jmat + \eps P$ are the $n$-th roots of $\eps$ in the complex plane;  see the red squares ($\color{red}\square\color{black}$) in Figure~\ref{fig:intro}.  In particular, even if $\eps$ is polynomially small in $n$ (say, $\eps = n^{-C}$ for a constant $C >0$), the eigenvalues of $\Jmat + \eps P$ are still $O(\log n/n)$ distance\footnote{See Section \ref{sec:notation} for a complete description of the asymptotic notation used here and in the sequel.} away from the $n$-th roots of unity.  

\begin{figure}
\includegraphics[scale=0.5]{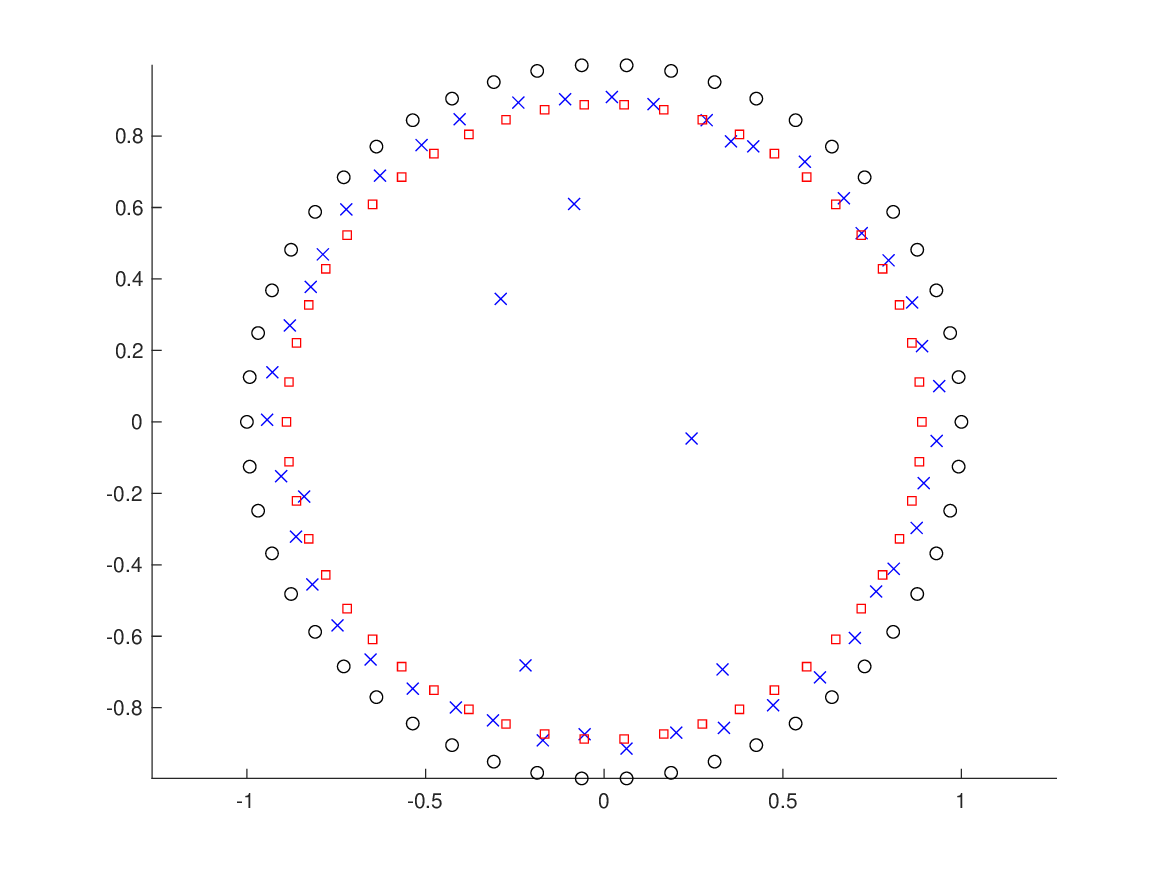}
\caption{For $n=50$ and $\gamma=1.5$, this figure shows plots of the $n$-th roots of unity (black circles {\Large $\color{black}\circ\color{black}$}), the eigenvalues of $\Jmat+n^{-\gamma}P$ (red squares $\color{red}\square\color{black}$), and the eigenvalues of $\Jmat+n^{-\gamma}E$ (blue {\Large $\color{blue}\times\color{black}$} symbols), where $P$ is an $n$ by $n$ matrix with $(n,1)$ entry equal to 1 and all other entries zero, and $E$ is an $n$ by $n$ matrix with iid standard complex Gaussian entries.  This figure uses small values of $n$ and $\gamma$ to make the general phenomenon in Theorem~\ref{thm:sample} visually apparent; see Section~\ref{sec:ESM} for figures with larger values of $n$ and $\gamma$.}
\label{fig:intro}
\end{figure}

A similar phenomenon can be observed when the perturbation matrix is random.  Indeed, suppose the entries of $E$ are independent and identically distributed (iid) standard complex Gaussian random variables.  For $\gamma > 1/2$, it follows from the work of Guionnet, Zeitouni, and the second author \cite{MR3134007} that the empirical spectral measure\footnote{See \eqref{eq:defESM} for the definition of the empirical spectral measure.} of $\Jmat + n^{-\gamma} E$ converges weakly in probability to the uniform probability measure on the unit circle as $n$ tends to infinity.  More generally, this result holds when the random matrix $E$ is replaced by other ensembles of random matrices \cite{2001.09024}.  

Figure~\ref{fig:intro} shows the eigenvalues of $\Jmat + n^{-\gamma} E$ (blue {\Large $\color{blue}\times\color{black}$} symbols) as well as the $n$-th roots of unity (black circles {\Large$\color{black}\circ\color{black}$}).  The figure clearly indicates a connection between the eigenvalues and the roots of unity.  For example, one can observe: 
\begin{enumerate}[(i)]
\item Even for small values of $n$, it appears there is a nearly one-to-one pairing between the eigenvalues and the roots of unity.  
\item While some eigenvalues of $\Jmat + n^{-\gamma} E$ lie closer to the origin, the vast majority of eigenvalues lie close to the unit circle.  
\item There is at least one eigenvalue of $\Jmat + n^{-\gamma} E$ near each root of unity, i.e., there are no ``gaps'' in the spectrum near the unit circle, even when the dimension is small.    
\end{enumerate}

The goal of the paper is to explore these properties and the pairing phenomenon between the eigenvalues and roots.  Our results go beyond the matrix $\Jmat$ and focus on classes of Toeplitz and non-normal matrices.  In addition, we do not require the random perturbation $E$ to have Gaussian entries, or even require it to have independent entries in our most general results.   Unlike the result cited above from \cite{MR3134007}, our results are non-asymptotic and focus on quantitatively describing the connection between the eigenvalues and roots observed in Figure~\ref{fig:intro}.  

\subsection{An illustrating example} \label{sec:illustrating_example}
Since we explore a variety of different spectral properties, we have organized the paper into several sections, and each section is devoted to a single property.  While this organization simplifies the presentation, it also makes it difficult to understand how the results connect together to provide a more complete picture of the spectral properties.  For the reader's convenience, Theorem \ref{thm:sample} below aggregates several results together in one specific example.  We emphasize that Theorem \ref{thm:sample} only provides one specific example of our results; our main results are significantly more general than what is stated in Theorem \ref{thm:sample}.  For an $n \times n$ matrix $A$ (not necessarily Hermitian), we let $\lambda_1(A), \ldots, \lambda_n(A) \in \mathbb{C}$ denote its eigenvalues, counted with algebraic multiplicity. For concreteness, we order the eigenvalues in lexicographic order:  we first sort the values by real part in decreasing order, and, in the event of a tie, we sort in decreasing order by the imaginary part.  For simplicity, we do not always state how the implicit constants in our asymptotic notation depend on the other parameters in Theorem \ref{thm:sample}.  

\begin{theorem}[Example main results for the matrix $\Jmat$] \label{thm:sample}
Let $E$ be an $n \times n$ matrix with iid Rademacher entries (taking the values $+1$ and $-1$ each with probability $1/2$) and let $\lambda_1', \ldots, \lambda_n'$ be the $n$-th roots of unity.  Then the following holds:
\begin{enumerate}[(i)]
\item (local law) \label{item:locallaw} Let $\varphi: \mathbb{C} \to \mathbb{C}$ be a smooth function with compact support, and fix $z_0 \in \mathbb{C}$.  For any $a \geq 0$, let $\varphi_{z_0}(z) = \varphi(n^a (z - z_0))$ be the rescaling of $\varphi$ to size order $n^{-a}$ around $z_0$.  If $\gamma \geq 2 +a$, then, with probability $1 - o(1)$, 
\[ \left| \sum_{i=1}^n \varphi_{z_0}( \lambda_i(\Jmat + n^{-\gamma} E)) - \sum_{i=1}^n \varphi_{z_0}(\lambda_i') \right| = O_{\varphi}( \log n). \]
\item (pairing) \label{item:pairing} For $\gamma > 1/2$ and any $p \geq 1$, there exists $\eps > 0$, so that, with probability $1 - o(1)$, 
\begin{equation} \label{eq:ellpbnd}
	\min_{\sigma} \left( \frac{1}{n} \sum_{i=1}^n \left| \lambda_i(\Jmat + n^{-\gamma} E) - \lambda_{\sigma(i)}' \right|^p \right)^{1/p} = O(n^{-\eps}), 
\end{equation} 
where the minimum is over all permutations $\sigma: \{1, \ldots, n\} \to \{1, \ldots, n\}$.  
\item (distance) \label{item:dist} For $\gamma \geq 3$ and any $\eps > 0$, with probability $1 - o(1)$, 
\[ \max_{1 \leq k \leq n} \min_{1 \leq i \leq n} \left| \lambda_k' - \lambda_i(\Jmat + n^{-\gamma} E) \right| = O \left( \frac{ \log^{1 + \eps} n}{n} \right). \]
\item (inliers) \label{item:inliers} If $\gamma \geq 2$, then for any $r \in (0,1)$, with probability $1 - o(1)$, there are at most $O_r(\log n)$ eigenvalues of $\Jmat + n^{-\gamma} E$ in the disk of radius $r$ centered at the origin.  Moreover, if $\gamma \geq 5$, then for any $\eps > 0$, with probability $1$, there are at most $O_{\eps}(1)$ eigenvalues in the disk of radius $1/4 - \eps$.  
\end{enumerate}
\end{theorem}

Theorem \ref{thm:sample} describes the spectral properties observed in Figure~\ref{fig:intro}.  For example, conclusion \eqref{item:pairing} quantitatively describes the one-to-one pairing phenomenon observed in the figure by measuring the average $\ell^p$-distance between the eigenvalues and the roots of unity; the minimum over all permutations $\sigma$ is due to the fact that there is no natural ordering of the complex numbers.  
Later we will show how the bound in \eqref{eq:ellpbnd} can be used to establish a rate of convergence for the empirical spectral measure of $\Jmat + n^{-\gamma} E$ to its limiting distribution in Wasserstein distance.  

As can be seen in Figure~\ref{fig:intro}, some of the eigenvalues of $\Jmat + n^{-\gamma}E$ lie away from the unit circle and some of them can even be close to the origin.  However, conclusion \eqref{item:inliers} shows that there cannot be too many of these so-called ``inliers.''  The second bound given in \eqref{item:inliers} is  similar to a bound given by Davies and Hager in \cite{MR2490477} but saves a factor of $\log n$.  
Despite the inlier eigenvalues, the bound in \eqref{item:dist} shows that every $n$-th root of unity is $O((\log n)^{1 + \eps}/n)$ distance away from an eigenvalue of $\Jmat + n^{-\gamma} E$.  

The local law stated in \eqref{item:locallaw} is one of the main tools we will use to establish some of the other local spectral properties we observed in Figure~\ref{fig:intro}.  The choice of $a \geq 0$ allows one to control how many eigenvalues contribute to the sum, with a larger value for $a$ corresponding to fewer contributing eigenvalues.  The case $a = 0$ describes the macroscopic (i.e., global) behavior of the eigenvalues, while $a > 0$ corresponds to the local, mesoscopic behavior of the eigenvalues.  Roughly speaking, the local law shows that, up to a small logarithmic error, the local eigenvalue behavior of $\Jmat + n^{-\gamma}E$ is similar to the behavior of the roots of unity.  We have stated the local law here to match other versions of the local law for non-Hermitian matrices in the random matrix theory literature, including \cite{MR3683369,MR3770875,MR3622892,MR3230004,MR3278919}.  Interestingly, the error term in Theorem \ref{thm:sample} is quite different than the error term for the local circular law \cite{MR3230002,MR3683369,MR3770875,MR3230004,MR3278919}, which hints that the behavior of the eigenvalues of $\Jmat + n^{-\gamma} E$ may be significantly different from the behavior for other random matrix ensembles, even at the local, mesoscopic level.  We explore this idea more in Section \ref{sec:locallaw}.

Theorem \ref{thm:sample} is stated only in the case when the random matrix $E$ contains iid Rademacher entries.  This is merely for convenience, and most of our results hold for much more general ensembles of random matrices.  
Similarly, while Theorem \ref{thm:sample} focuses on the matrix $\Jmat$, our main results apply to a much larger class of deterministic matrices.  One such class, which generalizes $\Jmat$, is the collection of banded Toeplitz matrices.  

\begin{definition}[Banded Toeplitz matrix] \label{def:toep}
Let $\{a_j \}_{j \in \mathbb{Z}}$ be a sequence of complex numbers, indexed by the integers, and let $k \geq 0$ be an integer.  We say that $A = (A_{ij})_{i,j=1}^n$ is an $n \times n$ \emph{Toeplitz matrix with symbol $\{a_j \}_{j \in \mathbb{Z}}$ truncated at $k$} if 
\[ A_{ij} = \left\{
     \begin{array}{ll}
       a_{i-j} & \text{ if } |i - j | \leq k \\ 
       0 & \text{ if }  | i - j | > k.
     \end{array}
   \right. \]
That is, the matrix $A$ has the form 
\begin{equation} \label{eq:ToeplitzA}
 A =
\begin{bmatrix}
  a_0 & a_{-1}   & a_{-2} & \cdots & a_{-k} & 0 & \cdots & 0  \\
  a_1 & a_0      & a_{-1} & \ddots &  &\ddots &    \ddots   & \vdots \\
  a_2 & a_1      & \ddots & \ddots & \ddots & &\ddots & 0 \\ 
 \vdots & \ddots & \ddots & \ddots &\ddots & \ddots &  & a_{-k}\\
 a_k  & & \ddots & \ddots & \ddots & \ddots & \ddots & \vdots \\
 0 & \ddots &  & \ddots & \ddots & \ddots & \ddots & a_{-2}  \\
 \vdots &   \ddots     & \ddots &     & \ddots    & \ddots & \ddots & a_{-1} \\
0 & \cdots & 0 & a_k & \cdots    & a_2    & a_1 & a_0
\end{bmatrix}. 
\end{equation} 
\end{definition}

We refer the reader to \cite{MR2179973} and references therein for further details about the spectral properties of banded Toeplitz matrices.  

Given a sequence $\{a_j \}_{j \in \mathbb{Z}}$ of complex numbers and an integer $k \geq 0$, we use the following convention for summations:
\begin{equation} \label{eq:sum_convention}
	\sum_{|j| \leq k} a_{j} = \sum_{ \substack{j \in \mathbb{Z} \\ |j| \leq k} } a_{j}. 
\end{equation} 

\begin{figure}
\begin{center}
\includegraphics[scale=0.4]{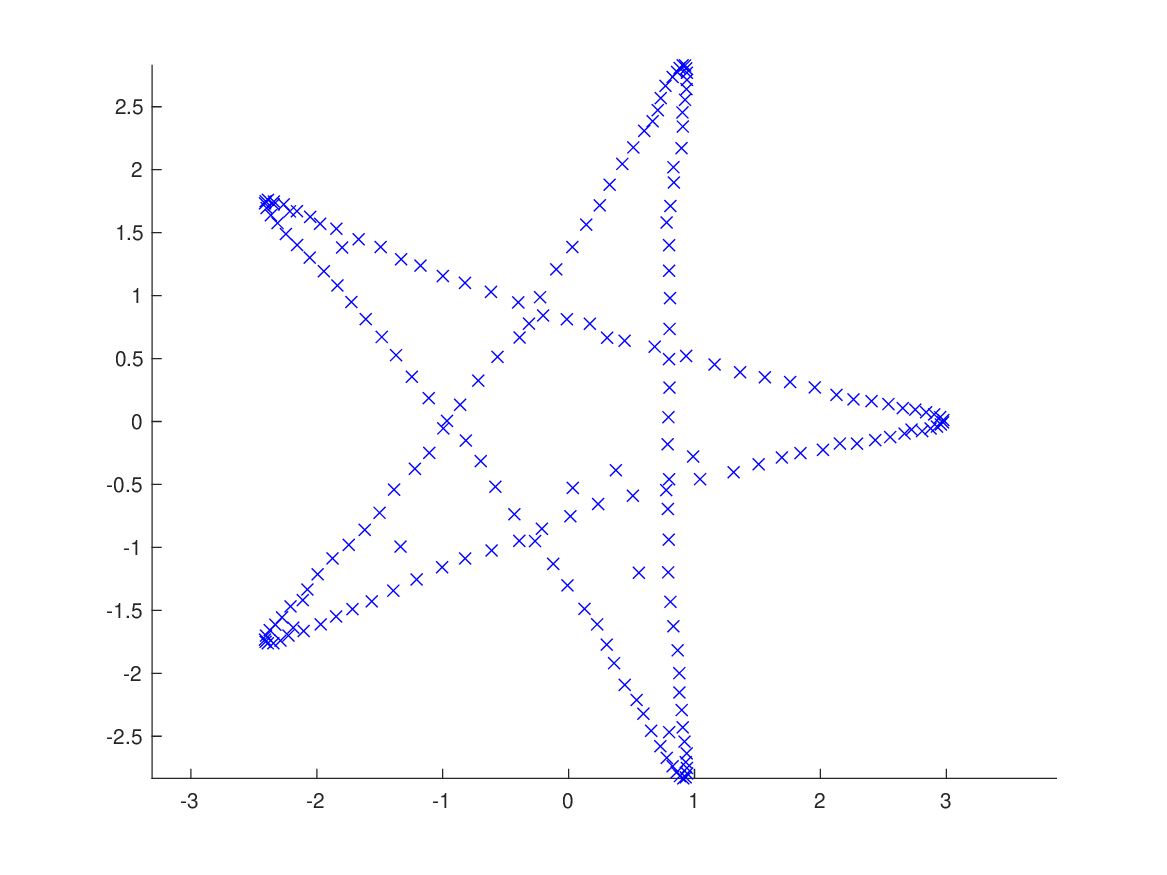}
\includegraphics[scale=0.4]{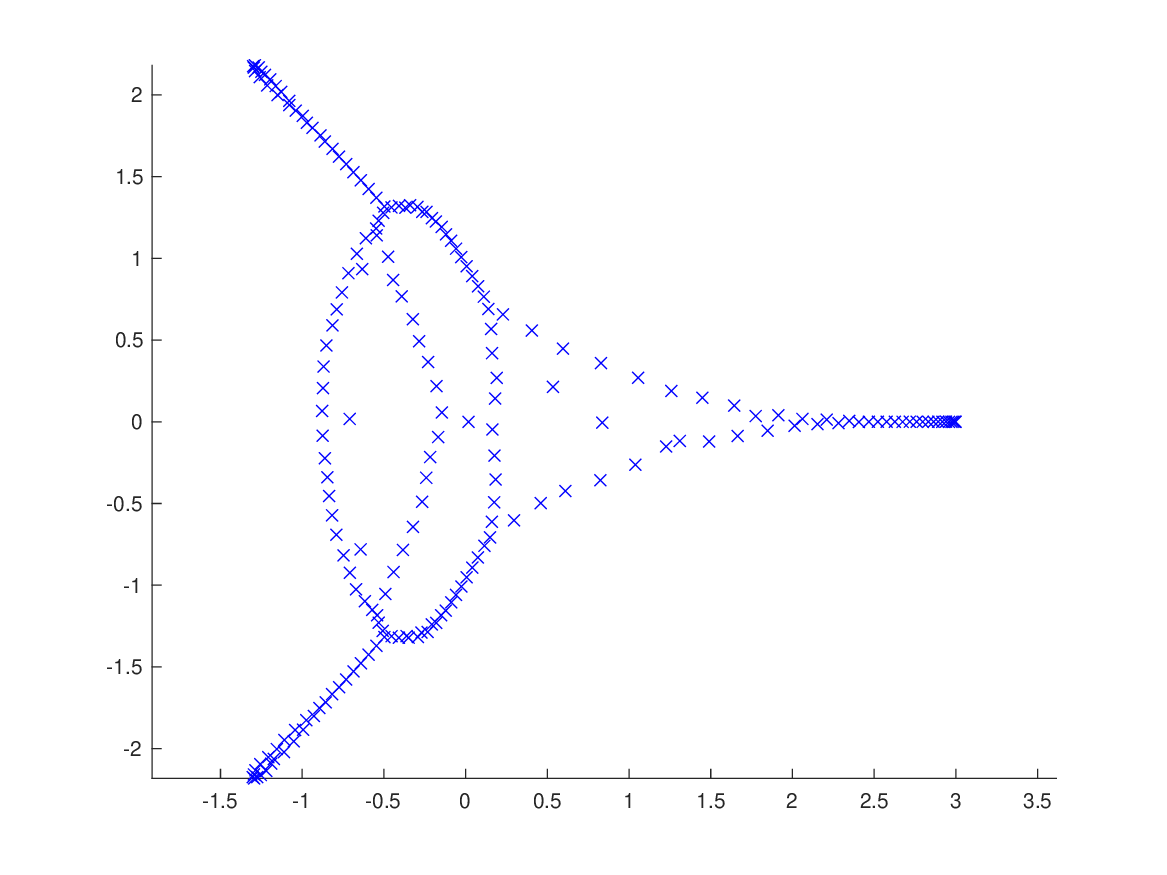}
\end{center}
\caption{Above are plots of two perturbed banded Toeplitz $n$ by $n$ matrices, on the left with symbol defined by $a_{-2}=2$, 
 $a_3 =-1$ and all other $a_j$ equal to zero; and on the right with symbol defined by $a_{-2}=1$, $a_{-1}=1$, 
 $a_3 =-1$ and all other $a_j$ equal to zero.  Each banded Toeplitz matrix was perturbed by adding a random matrix $n^{-\gamma} E$, where $E$ has iid standard complex Gaussian entries, where $n=200$ and $\gamma = 1.5$.}
\label{fig:justrand}
\end{figure}

Let $A$ be a Toeplitz matrix with symbol $\{a_j\}_{j \in \mathbb{Z}}$ truncated at $k$. 
As can be seen in Figure~\ref{fig:justrand}, we cannot expect the eigenvalues of $A + n^{-\gamma} E$ to always be near the unit circle.  In other words, it does not generally make sense to compare the eigenvalues of $A + n^{-\gamma} E$ to the roots of unity, as we did in Theorem \ref{thm:sample}.  This raises the question: 
\begin{question} \label{question:det}
	What deterministic values should we compare the eigenvalues to?
\end{question}

Let $S^1 := \{z \in \mathbb{C} : |z| = 1\}$ be the unit circle in the complex plane, and define the function $f: S^1 \to \mathbb{C}$ by
\[ f(\omega) := \sum_{|j| \leq k} a_j \omega^j, \]
where we use the summation convention introduced in \eqref{eq:sum_convention}.  
The function $f$ is often called the \emph{symbol} of the banded Toeplitz matrix $A$.  
We will show (see Theorem \ref{thm:pert-toep}) that, under some appropriate assumptions, the empirical spectral measure of $A + n^{-\gamma} E$ converges to the same distribution as $f(U)$, where $U$ is a random variable uniformly distributed on the unit circle $S^1$.  In other words, the limiting empirical spectral measure is the push-forward of the uniform probability measure on $S^1$ by the symbol $f$.  

This leads us to the following answer to Question \ref{question:det}: 
If $\lambda_1', \ldots, \lambda_n'$ are again the $n$-th roots of unity, we shall compare the eigenvalues of $A + n^{-\gamma} E$ to the deterministic points $f(\lambda_1'), \ldots, f(\lambda_n')$.   We often refer to $f(\lambda_1'), \ldots, f(\lambda_n')$ as the \emph{classical locations} of the eigenvalues, and these new deterministic points approximate the eigenvalues of $A + n^{-\gamma} E$ just as the roots of unity approximated the eigenvalues of $\Jmat + n^{-\gamma} E$; see Figure~\ref{fig:justrand+classical}. 

\begin{figure}
\includegraphics[scale=0.4]{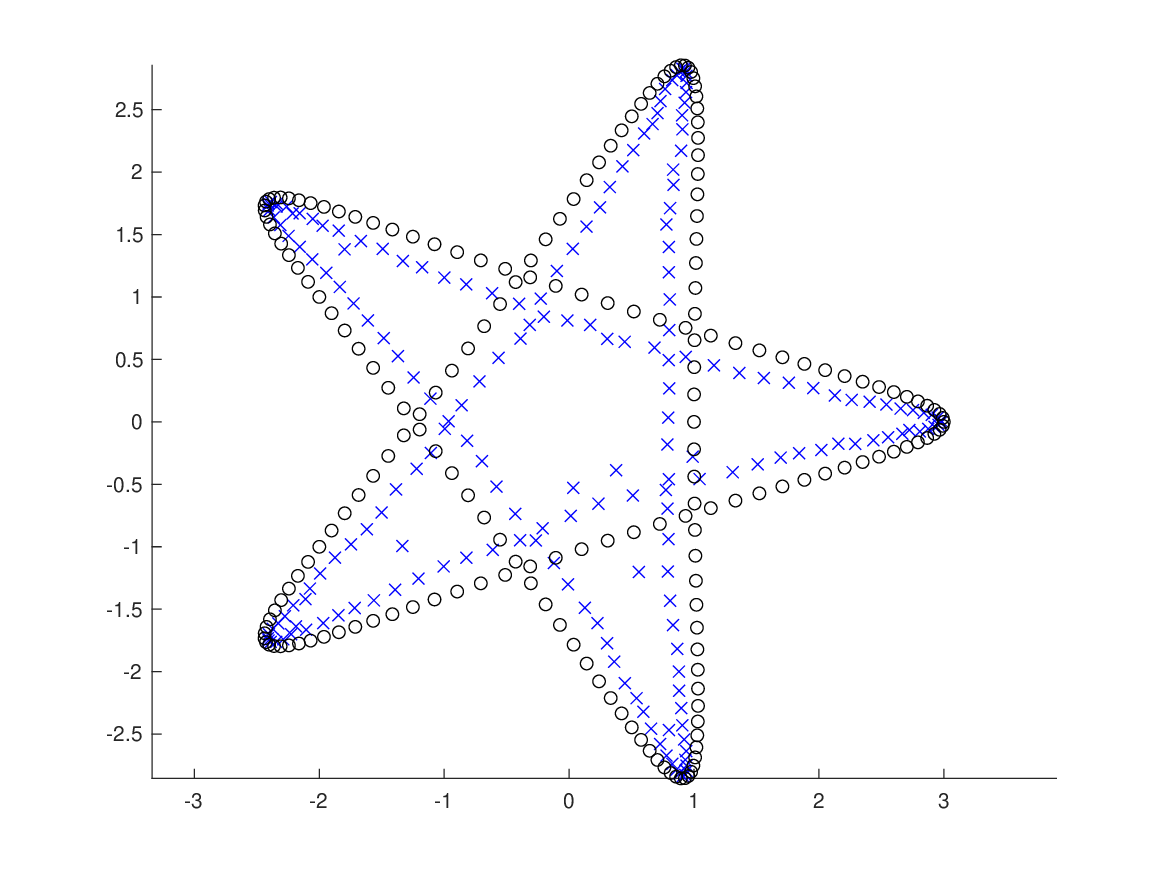}
\includegraphics[scale=0.4]{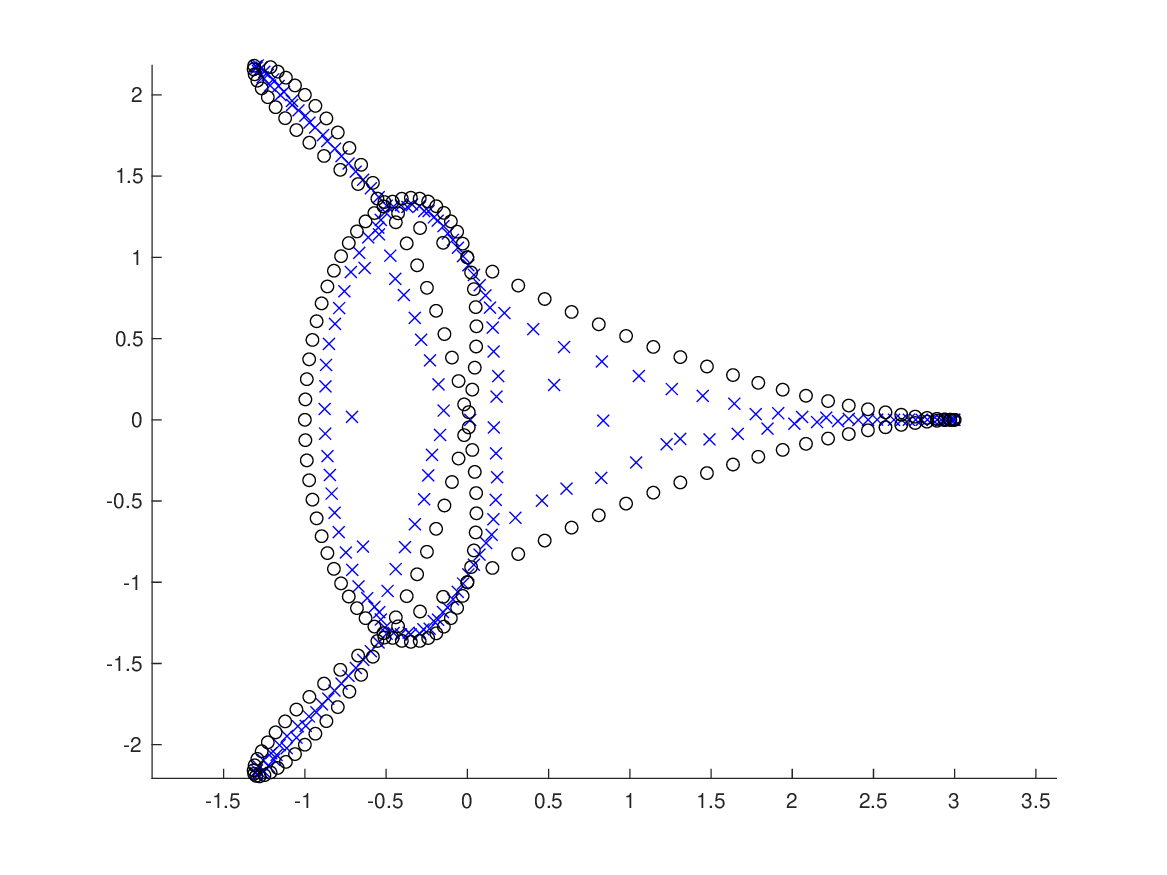}
\caption{The blue $\color{blue}\times\color{black}$ symbols plot the same eigenvalues of perturbed banded Toeplitz matrices as in Figure~\ref{fig:justrand}. Each example has been overlaid with
black circles~{\Large$\circ$} marking the deterministic classical locations of the eigenvalues, as described after Question~\ref{question:det}.  Note the reasonably close agreement in the random eigenvalues and the corresponding deterministic classical locations.}
\label{fig:justrand+classical}
\end{figure}

Roughly speaking, our main results show that Theorem \ref{thm:sample} holds when $\Jmat$ is replaced by a banded Toeplitz matrix $A$, provided the roots of unity $\lambda_1', \ldots, \lambda_n'$ are replaced with the classical locations $f(\lambda_1'), \ldots, f(\lambda_n')$.  More importantly, our results also aim to show why $f(\lambda_1'), \ldots, f(\lambda_n')$ are the correct choices for the classical locations of the eigenvalues.  

\subsection{Contribution of this paper and comparison to existing results}
One of the main goals of this paper is to obtain non-asymptotic results for matrix models of the form $A + n^{-\gamma} E$, where $A$ is a banded Toeplitz matrix and $E$ is a random matrix.  In particular, we are interested in results that hold even when the dimension $n$ is fairly small.  As such, we have attempted to use methods where we can explicitly state all constants.  For example, in the results from Theorem \ref{thm:sample}, we try to explicitly state or provide ways to compute all of the implicit constants from our asymptotic notation.  We have also attempted to state our results as generally as possible, applying not only to banded Toeplitz matrices with growing bandwidth but also to very general classes of random matrices $E$.  

While many results in the field require the matrix $E$ to be random, in some cases our results can be shown to hold also for general classes of deterministic matrices.  An example of this is given in the second statement of part \eqref{item:inliers} from Theorem \ref{thm:sample}: the  $O_{\eps}(1)$ bound on the number of eigenvalues of $\Jmat + n^{-\gamma}E$ in the disk of radius $1/4 - \eps$ holds with probability $1$, meaning it holds for every possible realization of the signs for the entries of $E$.  In other words, out of the $2^{n^2}$ possible choices of signs for the entries of $E$, there is not a single choice that violates this property.  Here, our methods are limited to the disk of radius $1/4$ due to technical issues in the proof, but we suspect this result is part of a larger phenomenon.  

We also note that in many asymptotic results in the field, the value of $\gamma > 1/2$ is irrelevant (see, for instance, \cite{FPZeitouni_regularization_2015,basak_regularization_2019,basak_spectrum_2020} and references therein).  This can be seen as a type of universality since the limit of the empirical spectral measure is independent of $\gamma$. In Theorem \ref{thm:sample} there are several minimum bounds required on $\gamma$ (some of these lower bounds will be relaxed in the forthcoming sections).  For non-asymptotic results (and in particular when $n$ is small) it seems likely (based on some numerical results, see Figure~\ref{fig:non-asymptotic}) that the behavior of the eigenvalues does depend on $\gamma$.  One early result quantifying how the sizes of $n$ and $\gamma$ affects the behavior of the eigenvalues is due to Davies and Hager in \cite{MR2490477} in the case of the matrix $\Jmat$ (defined in \eqref{eq:defT}).  While our work attempts to explore this relationship in more generality, there are still many open questions.  

\begin{figure}
\includegraphics[scale=0.27]{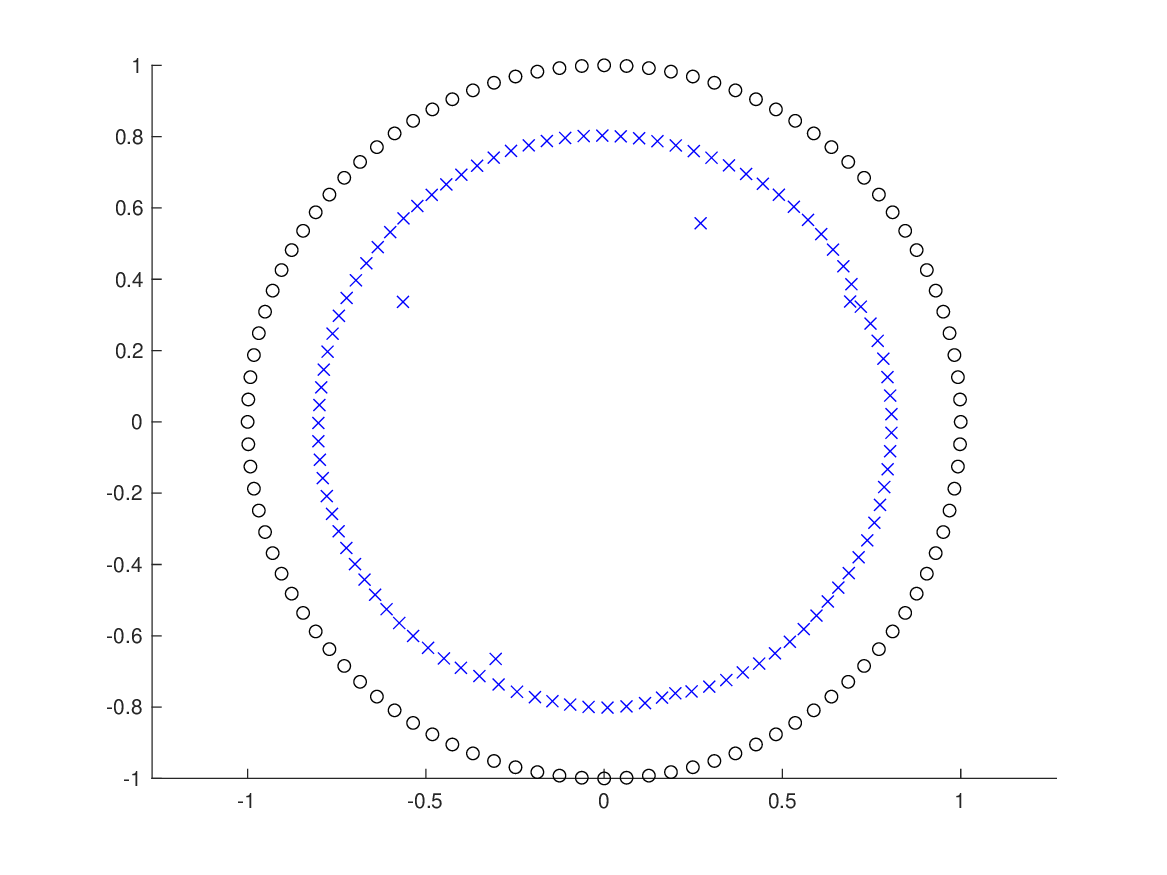}
\includegraphics[scale=0.27]{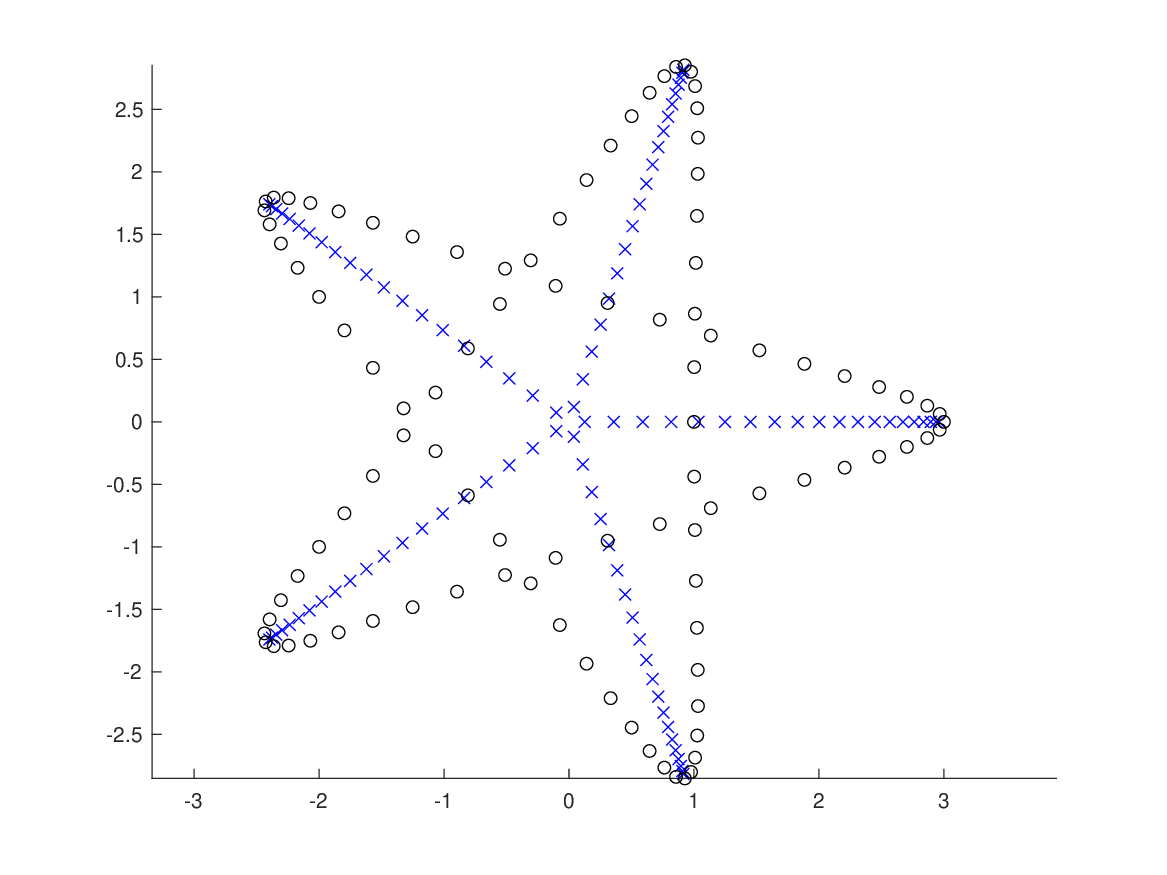}
\includegraphics[scale=0.27]{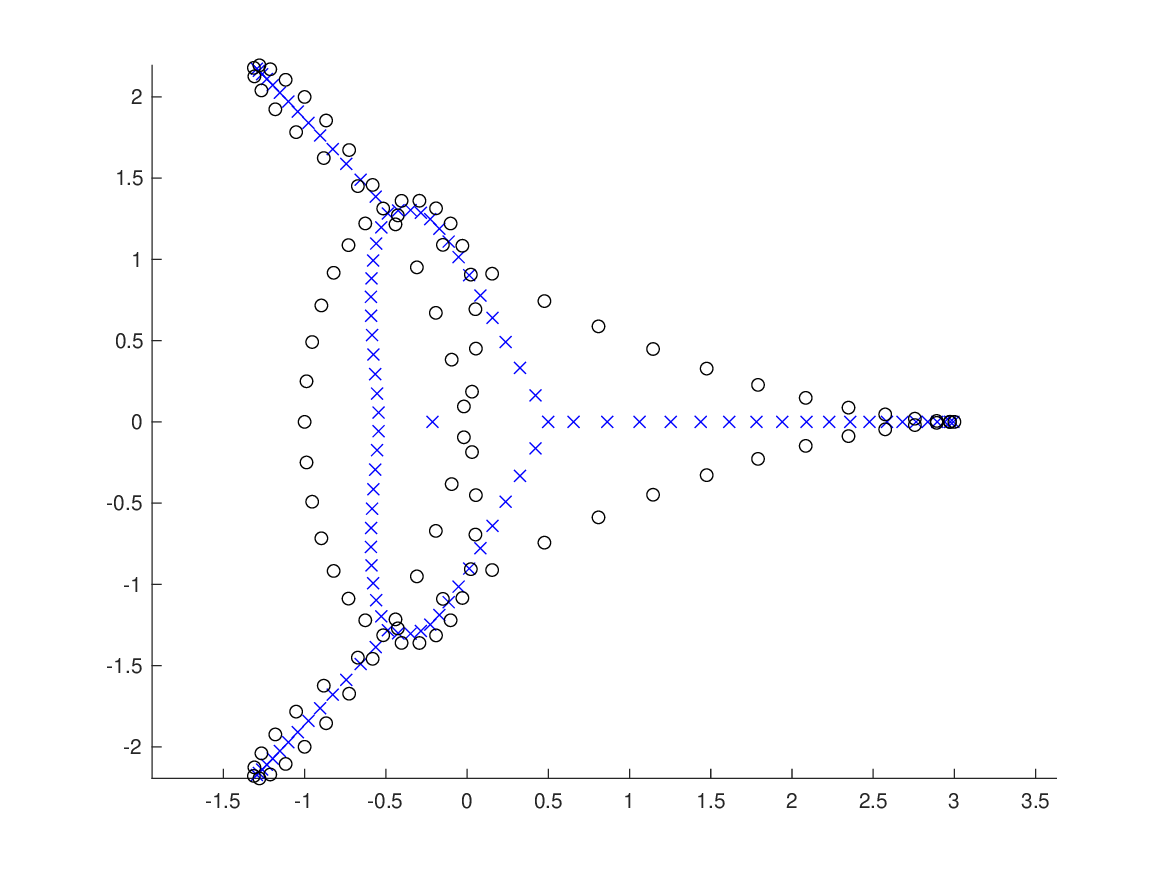}
\caption{Above are three plots showing the eigenvalues for an $n$ by $n$ perturbed banded Toeplitz matrix (blue $\color{blue}\times\color{black}$ symbols) compared to the corresponding classical locations of the eigenvalues (black circles~{\Large$\circ$}), where $n=100$, $\gamma=5$, and the perturbation is $n^{-\gamma}E$ where $E$ has iid standard complex Gaussian entries.  While results in Sections~\ref{sec:ratec} and \ref{sec:locallaw} show that the classical locations are a good approximation for the random eigenvalues when $n$ is large, it is evident from the plots above that for relatively small $n=100$ and relatively large $\gamma =5$, the approximation is not yet close in these cases.}
\label{fig:non-asymptotic}
\end{figure}

Lastly, we wish to point out that while some of our results are similar to existing results in the literature, our methods of proof are substantially different.  We utilize a comparison method (discussed in more detail in Section \ref{sec:outline}).  Comparison and replacement methods have been used extensively in the random matrix theory literature as they often allow one to compare the eigenvalues of $A + n^{-\gamma}E$ to the eigenvalues of $A + n^{-\gamma} \tilde E$ for two different random matrices $E$ and $\tilde E$.  Here, we take a different approach by comparing the eigenvalues of $A + n^{-\gamma} E$ to the eigenvalues of $\tilde A + n^{-\gamma} E$.  By replacing the deterministic matrix only, our method is robust and can be applied to many different random matrix models; in some of our most general results, we will not even need the entries of $E$ to be independent.    Moreover, this method naturally explains the choice for the classical locations of the eigenvalues discussed above.  

We will discuss how our results compare with other works in the literature after we have introduced our main results.  For now, we simply focus on Theorem \ref{thm:sample}.  Perhaps the closest results to part \eqref{item:locallaw} of Theorem \ref{thm:sample} come from the work of Sj\"{o}strand and Vogel \cite{MR4200678}.  They showed precise asymptotic bounds for the number of eigenvalues in smooth domains of the matrix $A + n^{-\gamma} E$, where $A$ is a banded Toeplitz matrix and $E$ is a random matrix with iid standard complex Gaussian entries.  Our results differ from those in \cite{MR4200678} in that we consider linear statistics of the eigenvalues, rather than the counting function of the eigenvalues.  As we will see in Sections \ref{sec:ratec} and \ref{sec:locallaw}, our results also apply to much more general choices for the matrix model $E$.  We also show how results such as part \eqref{item:locallaw} of Theorem \ref{thm:sample} can be used to establish the pairing result in part \eqref{item:pairing} as well as a rate of convergence in Wasserstein distance; these results are quite technical to prove and do not seem to follow immediately from the local law.  

Lastly, we mention that part \eqref{item:inliers} of Theorem \ref{thm:sample} is similar to the work of Davies and Hager \cite{MR2490477}.  Some of the results in \cite{MR2490477} are stronger than ours (especially for small values of $n$).  However, our results go beyond the matrix $\Jmat$, while the results in \cite{MR2490477} are limited to Jordan matrices; we also demonstrate cases in which the $O(\log n)$ error can in fact be replaced by a $O(1)$ error.

\subsection{Organization of the paper}
The paper is organized as follows.  In Section \ref{sec:tools}, we introduce the preliminary material we will need to state and prove our main results.  In particular, we list all of the notation used in the paper as well as the main definitions and tools used in the proofs.  Section \ref{sec:tools} also contains a number of additional examples involving matrices in Jordan normal form.  In Section \ref{sec:ESM}, we compute the limiting empirical spectral measure for random perturbations of banded Toeplitz matrices with growing bandwidth.  The results in this section appear fairly standard, although we were not able to find them in the literature.  We include Section~\ref{sec:ESM} for completeness and because the proofs highlight our main comparison techniques.  We establish a generalization of part \eqref{item:pairing} from Theorem \ref{thm:sample} in Section \ref{sec:ratec} as well as a rate of convergence in Wasserstein distance for the empirical spectral measure to its limiting distribution.  In Section \ref{sec:app}, we establish several deterministic results, including the bound from part \eqref{item:inliers} of Theorem \ref{thm:sample} discussed above.  A general version of the local law is presented and proven in Section \ref{sec:locallaw}.  We also include a number of examples and numerical simulations throughout the paper.  

The appendix contains a number of auxiliary results and arguments worked out in greater detail.  For example, in Section \ref{sec:tools}, we introduce a number of tools that will be used in the forthcoming proofs; although these tools are fairly standard in the literature, we could not find them in precisely the forms required here, and so the proofs have been provided (in full detail) in the appendix.  

\subsection*{Acknowledgments}
Part of this research was conducted by the second author at the International Centre for Theoretical Sciences (ICTS) during a visit for the program ``Universality in random structures: Interfaces, Matrices, Sandpiles'' (Code: ICTS/urs2019/01).  The first author thanks Elliot Paquette and Anirban Basak for useful conversations and suggestions.  The first author also thanks Mark Embree for helpful comments and references.

\section{Notation, definitions, and tools} \label{sec:tools}

This section introduces the notation, tools, and main definitions we will use throughout the paper.  In Section~\ref{sec:notation}, we highlight some notation and conventions, and in Section~\ref{sec:theory}, we develop and discuss the main tools utilized in our proofs, Theorems \ref{thm:non}, \ref{thm:replacement}, and \ref{thm:replrank}. Section \ref{sec:outline} provides a high-level overview of our methods, and in Section~\ref{sec:works}, we discuss related work in the literature.  The proofs for Theorems \ref{thm:non}, \ref{thm:replacement}, \ref{thm:replrank}, and Proposition \ref{prop:small-cancel} 
can be found in Appendix~\ref{sec:prooftheory}. 

\subsection{Notation} \label{sec:notation}

For a vector $x \in \mathbb{C}^n$, $\|x\|$ denotes its Euclidean norm.  
For a matrix $A$, $A^\mathrm{T}$ is its transpose and $A^\ast$ denotes the conjugate transpose of $A$.  $\|A\|$ is the spectral norm of $A$ and $\|A\|_2$ is its Frobenius norm, defined as 
\begin{equation} \label{eq:deffrob}
	\|A\|_2 := \sqrt{ \tr (A A^\ast) } = \sqrt{ \tr (A^\ast A)}.
\end{equation} 
All matrices under consideration in this article are assumed to have complex entries, unless otherwise indicated.  We will use $I_n$ to denote the $n \times n$ identity matrix; often we will write $I$ when its size can be deduced from context.  

We let $\lambda_1(A), \ldots, \lambda_n(A) \in \mathbb{C}$ be the eigenvalues of the $n \times n$ matrix $A$, counted with algebraic multiplicity.  Recall that, unless otherwise notes, we order the eigenvalues in lexicographic order:  we first sort the values by real part in decreasing order, and, in the event of a tie, we sort in decreasing order by the imaginary part.  The empirical spectral measure $\mu_A$ of $A$ is defined as the probability measure
\begin{equation} \label{eq:defESM}
	\mu_A := \frac{1}{n} \sum_{j=1}^n \delta_{\lambda_j(A)}, 
\end{equation} 
where $\delta_z$ denotes the point mass at $z \in \mathbb{C}$. 


The singular values of the $n \times n$ matrix $A$ are the eigenvalues of $\sqrt{A A^\ast}$.  We let $\sigma_1(A) \geq \cdots \geq \sigma_n(A)$ denote the ordered singular values of $A$ and $\nu_A$ be the empirical measure constructed from the singular values of $A$:
\begin{equation} \label{eq:defsingEM}
	\nu_A := \frac{1}{n} \sum_{j=1}^n \delta_{\sigma_j(A)}. 
\end{equation} 
We often use $\sigma_{\min}(A) := \sigma_n(A)$ to denote the smallest singular value of $A$, which can be computed by the variational characterization
\begin{equation} \label{eq:varminsing}
	\sigma_{\min}(A) = \min_{x \in \mathbb{C}^n : \|x\| = 1} \| A x\|. 
\end{equation}

For two probability measures $\mu$ and $\nu$ on $\mathbb{C}$, the $L^1$-Wasserstein distance between $\mu$ and $\nu$ is denoted as $W_1(\mu, \nu)$ and defined by
\begin{equation} \label{eq:defW1}
	W_1(\mu, \nu) := \inf_{\pi} \int |x - y| d \pi(x,y), 
\end{equation} 
where the infimum is over all probability measures $\pi$ on $\mathbb{C} \times \mathbb{C}$ with marginals $\mu$ and $\nu$.

We use $\log (x)$ to denote the natural logarithm of $x$, and $[n]$ to denote the discrete interval $\{1, \ldots, n\}$.  For a finite set $S$, we use $|S|$ to denote the cardinality of $S$.
Also, we use $\sqrt{-1}$ to denote the imaginary unit, and we reserve $i$ as an index.   

Let $C_c^\infty(\mathbb{C})$ be the set of smooth, compactly supported functions $\varphi: \mathbb{C} \to \mathbb{C}$.  We will use $\supp(\varphi)$ to denote the support of $\varphi$ and $\|\varphi\|_{\infty}$ to denote its $L^\infty$-norm.  

Let
\[ B(z,r) := \{ w \in \mathbb{C} : |z - w| < r \} \]
be the open ball of radius $r > 0$ centered at $z \in \mathbb{C}$ in the complex plane, and let
$S^1$ be the unit circle $\{z \in \mathbb{C} : |z| = 1\}$ in the complex plane centered at the origin.  The quantifiers ``almost everywhere'' and ``almost all'' will be with respect to the Lebesgue measure on $\mathbb{C}$.

Asymptotic notation is used under the assumption that $n$ tends to infinity.  We use $X = O(Y)$, $Y=\Omega(X)$, $X \ll Y$, or $Y \gg X$ to denote the estimate $|X| \leq C Y$ for some constant $C > 0$, independent of $n$, and all $n \geq C$.  If $C$ depends on other parameters, e.g. $C = C_{k_1, k_2, \ldots, k_p}$, we indicate this with subscripts, e.g. $X = O_{k_1, k_2, \ldots, k_p}(Y)$.  The notation $X = o(Y)$ denotes the estimate $|X| \leq c_n Y$ for some sequence $(c_n)$ that converges to zero as $n \to \infty$, and, following a similar convention, $X = \omega(Y)$ means $|X| \geq c_n Y$ for some sequence $(c_n)$ that converges to infinity as $n \to \infty$.  Finally, we write $X = \Theta(Y)$ if $X \ll Y \ll X$.


\subsection{Tools} \label{sec:theory}

For a probability measure $\mu$ on $\mathbb{C}$ that integrates $\log|\cdot|$ in a neighborhood of infinity, its logarithmic potential $\mathcal{L}_\mu$ is the function $\mathcal{L}_\mu: \mathbb{C} \to [-\infty, +\infty)$ given by
\[ \mathcal{L}_\mu(z) := \int_{\mathbb{C}} \log |w -z| \ d \mu(w). \]
It follows from Fubini's theorem that the logarithmic potential is finite almost everywhere.  

The logarithmic potential is one of the key tools used to study the eigenvalues of non-Hermitian random matrices \cite{MR2908617}.  The logarithmic potential of $\mu_A$  is given by
\begin{equation} \label{eq:logpotESM}
	\mathcal{L}_{A}(z) \equiv \mathcal{L}_{\mu_A}(z) := \int_{\mathbb{C}} \log|w-z| \ d \mu_A(w) = \frac{1}{n} \sum_{j=1}^n \log|\lambda_j(A) - z| = \frac{1}{n} \log |\det (A - zI)|, 
\end{equation} 
where $\det(A - zI)$ is the determinant of $A - zI$, $z \in \mathbb{C}$, and $I$ is the identity matrix.  
Many of tools we introduce in Section~\ref{sec:theory} focus on understanding and estimating the logarithmic potential of the empirical spectral measure.
 
In many cases, convergence of the logarithmic potential for almost every $z \in \mathbb{C}$ is enough to guarantee the convergence of the empirical spectral measure; see, for instance, \cite[Theorem 2.8.3]{TaoBook} or \cite{tao_random_2010-1,LogPot, GTcirc}.   The following replacement principle due to Tao and Vu \cite{tao_random_2010-1} compares the empirical spectral measures of two random matrices.  

\begin{theorem}[Replacement principle \cite{tao_random_2010-1}]\label{thm:TVrepl}
 Suppose for each $n$ that $M_1$ and $M_2$ are $n \times n$ ensembles of random matrices. Assume that
 \begin{enumerate}
  \item[(i)] the expression $\frac 1n \norm{M_1}^2_2 +\frac 1n \norm{M_2}^2_2$ is bounded in probability (resp, almost surely); and
  
  \item[(ii)] for almost all complex numbers $z$,
  \[ \mathcal{L}_{M_1}(z) - \mathcal{L}_{M_2}(z) \]
  converges in probability (resp., almost surely) to zero as $n \to \infty$ and, in particular, for each fixed $z$, these logarithmic potentials are finite with probability tending to $1$ as $n$ tends to infinity (resp., almost surely nonzero for all but finitely many $n$).
\end{enumerate}
Then, $\mu_{M_1}-\mu_{M_2}$ converges in probability (resp., almost surely) to zero as $n \to \infty$.
\end{theorem}

Since we are interested in non-asymptotic results, our first main tool is a non-asymptotic version of Theorem \ref{thm:TVrepl}, which quantitatively measures how close $\mu_{M_1}$ is to $\mu_{M_2}$ when the dimension $n$ is finite.  
Recall that $C_c^\infty(\mathbb{C})$ is the set of smooth, compactly supported functions $\varphi: \mathbb{C} \to \mathbb{C}$, and recall that $\supp(\varphi)$ denotes the support of $\varphi$.  

\newcommand{\diam}{D}
\newcommand{\sdist}{\mathfrak{D}}
\begin{theorem}[Non-asymptotic replacement principle] \label{thm:non}
Let $M_1$ and $M_2$ be two $n \times n$ random matrices (not necessarily independent).  Let $\eps, \eta > 0$, $T > 2$ (all possibly depending on $n$), and take $\varphi \in C_c^\infty(\mathbb{C})$.  Assume the following: 
\begin{enumerate}
\item (Norm bound) $\Prob(  \| M_1 \| + \|M_2 \| \geq T ) \leq \eps$.
\item (Concentration of log determinants) For $Z$ uniformly distributed on $\supp(\Delta \varphi)$, independent of $M_1$ and $M_2$, we have
\begin{equation} \label{eq:concentration}
	\Prob \left( \left| \mathcal{L}_{M_1}(Z) - \mathcal{L}_{M_2}(Z) \right| \geq \eta \right) \leq \eps. 
\end{equation} 
\end{enumerate}
Then there exists a constant $\Csixone > 0$ depending only on $\varphi$ such that for every integer $m \geq 1$
\[ \left| \int_{\mathbb{C}} \varphi \ d \mu_{M_1} - \int_{\mathbb{C}} \varphi \ d \mu_{M_2} \right| \leq \Csixone \left( \eta + \frac{\log T}{m \sqrt{\eps}} \right) \]
with probability at least $1 - 2(m+1) \eps$.   The constant $\Csixone$ is described explicitly in \eqref{e:ThmConst} and is usually relatively simple; for example, if $\supp \varphi$ is contained in the unit disk, then one may take $\Csixone= 14\sqrt \pi \| \Delta \varphi\|_\infty$, where  $\| \cdot\|_\infty$ denotes the $L^\infty$-norm.
\end{theorem}

\begin{remark} \label{rem:mlim}
We have formulated Theorem \ref{thm:non} in probabilistic terms.  However, if the matrices $M_1$ and $M_2$ are deterministic and we can take $\eps$ arbitrarily small, the result can be reformulated.  Indeed, taking $\eps = m^{-3/2}$ and letting $m$ approach infinity gives  
\[ \left| \int_{\mathbb{C}} \varphi \ d \mu_{M_1} - \int_{\mathbb{C}} \varphi \ d \mu_{M_2} \right| \leq \Csixone \eta \]
provided 
\[ \left| \mathcal{L}_{M_1}(z) - \mathcal{L}_{M_2}(z) \right| < \eta \]
for almost every $z \in \supp(\varphi)$. 
\end{remark}

Since 
\[ \int_{\mathbb{C}} \varphi \ d \mu_{M_i}= \frac{1}{n} \sum_{j=1}^n \varphi(\lambda_j(M_i)) \]
for $i=1,2$, it is often useful to let $\varphi$ approximate an indicator function.  In particular, by allowing $\varphi$ to depend on $n$, we will use Theorem \ref{thm:non} (along with the explicit formula for $\Csixone$ given in \eqref{e:ThmConst}) to establish a local law and rate of convergence for the empirical spectral measure of a Toeplitz matrix subject to a random perturbation.  Such local laws  describe the mesoscopic behavior of the eigenvalues.  Similar local laws have been established in the random matrix theory literature for a variety of ensembles, see \cite{MR3306005,MR3230002,MR3683369,MR3770875,MR3230004,MR3278919,MR3000559,MR3697772,MR3704770,MR3183577,MR4021251,MR3800840,MR4091114,MR4061437,MR3622895,MR3622892,MR3602820,MR3791390,MR2481753,MR3688032,MR4125959,MR3757532,MR3915283,MR3098073,MR3406427} and references therein for a partial list of such results.  

In order to use Theorem \ref{thm:non}, we need to be able to control the difference $\mathcal{L}_{M_1}(z) - \mathcal{L}_{M_2}(z)$ for a dense enough collection of $z \in \mathbb{C}$.  Our next two theorems provide such bounds.


Our next result can be compared to Weyl's perturbation theorem (see \eqref{eq:weyleig}), which implies that the eigenvalues of Hermitian matrices are stable under small perturbations.    
Theorem \ref{thm:replacement} below shows that the logarithmic determinant (and hence the empirical spectral measure by Theorem~\ref{thm:non}) of an arbitrary matrix is also stable under small perturbations, provided the smallest singular values of the matrix are not too extreme.  

\begin{theorem}[Norm comparison principle] \label{thm:replacement}
Let $M_1$ and $M_2$ be $n \times n$ matrices.  Take $z \in \mathbb{C}$ and $\eps \in (0,1/2)$, and assume that
\begin{equation} 
	\sigma_{\min} := \min\{ \sigma_{\min}(M_i - z I) : i=1,2 \} > 0
\end{equation}
and $\| M_1 - M_2 \| < \eps/2$.  Then
\begin{equation} \label{eq:conclusion}
	\left| \mathcal{L}_{M_1}(z) - \mathcal{L}_{M_2}(z) \right| \leq 
6\left(| \log(\eps/2)| + |\log \sigma_{\min} | \right) \nu_{M_2 - zI}([0, \eps]) +  \frac{2}{\eps}\| M_1 - M_2 \|. 
	\end{equation} 
\end{theorem}

Theorem~\ref{thm:replacement} has two useful quantitative features.
First, the right-hand side of \eqref{eq:conclusion} depends only on $\nu_{M_2 - zI}$ and not on $\nu_{M_1 - zI}$; this means one only needs to control the number of small singular values for one of the matrices.  
Second, Theorem~\ref{thm:replacement} makes no assumptions on the randomness of the matrices $M_1$ and $M_2$.  In particular, one can take the matrix $M_2$ to be entirely deterministic, and in some cases, the eigenvalues can be explicitly computed.  We give examples using these features in Sections \ref{sec:examples} and \ref{sec:app} below. 

In cases where the smallest singular values are comparable and the second smallest singular values are bounded from below, it is possible to prove a more precise result along the same lines as Theorem~\ref{thm:replacement}.  We will use Proposition~\ref{prop:small-cancel} below in Section~\ref{sec:deterministic} to derive a non-asymptotic result.

\begin{proposition}\label{prop:small-cancel}
Let $M_1$ and $M_2$ be $n$ by $n$ matrices, let $z \in \mathbb{C}$, and let $a,b,$ and $\epsilon$ be positive real numbers (which may depend on $n$).  If the second-smallest singular values satisfy $\sigma_{n-1}(M_i-zI)\ge \epsilon$ for $i=1,2$ and the smallest singular values satisfy $a \le \sigma_n(M_i - zI) \le b$ for $i=1,2$, then
$$\abs{\mathcal L_{M_1}(z) - \mathcal L_{M_2}(z)} \le \frac{\log(b/a)}{n} + \frac{1}{\epsilon} \|M_1 - M_2\|.$$
\end{proposition}

We now turn to our final result of the subsection.  Recall the following interlacing result for the eigenvalues of Hermitian matrices (see \cite[Exercise III.2.4]{Bhatia}). 
If $A$ is an $n \times n$ Hermitian matrix with eigenvalues $\lambda_1(A) \geq \cdots \geq \lambda_n(A)$ and $E$ is a positive semi-definite Hermitian matrix of rank one, then the eigenvalues $\lambda_1(A + E) \geq \cdots \geq \lambda_n(A+E)$ of $A+E$ interlace with the eigenvalues of $A$:
\begin{equation} \label{eq:interlace}
	\lambda_i(A+E) \geq \lambda_i(A) \geq \lambda_{i+1}(A+E) 
\end{equation} 
for $1 \leq i \leq n-1$.  In many cases, \eqref{eq:interlace} implies that low rank perturbations of Hermitian matrices do not change the spectrum significantly.  Our next main result captures a similar behavior for the logarithmic potential (and hence the empirical spectral measure by Theorem \ref{thm:non}) of arbitrary matrices whose smallest and largest singular values are not too extreme.

\begin{theorem}[Rank comparison principle] \label{thm:replrank}
Let $M_1$ and $M_2$ be $n \times n$ matrices.  Take $z \in \mathbb{C}$, and assume that
\[ \sigma_{\min} := \min\{ \sigma_{\min}(M_i - z I) : i =1,2 \} > 0. \]
Similarly, define
\[ \sigma_{\max} := \max \{ \| M_i - z I \| : i = 1, 2 \}. \]
Then
\[ \left| \mathcal{L}_{M_1}(z) - \mathcal{L}_{M_2}(z) \right| \leq 2 \left( | \log \sigma_{\min} | + |\log \sigma_{\max} | \right) \frac{\rank (M _1 - M_2)}{n}. \]
\end{theorem}

Theorems \ref{thm:replacement} and \ref{thm:replrank} are closely related to a number of techniques used in the random matrix theory literature to study non-Hermitian matrices, including those found in \cite{MR1428519,basak_regularization_2019,MR2409368,MR2490477}.  The authors are not aware of any works where these results are stated in the deterministic forms given above.

\subsection{Overview and outline of the methods} \label{sec:outline}

In this section, we outline the methods used in the proofs of our main results.  We begin with some details concerning the proofs of Theorems \ref{thm:replacement} and \ref{thm:replrank}.  Recall from \eqref{eq:logpotESM} that the logarithmic potential $\mathcal{L}_M$ of the empirical spectral measure $\mu_M$ of an $n \times n$ matrix $M$ is given by
\[ \mathcal{L}_M(z) = \frac{1}{n} \log |\det(M - zI)| = \frac{1}{n} \sum_{j=1}^n \log \sigma_j(M - zI), \]
where we used the fact that the absolute value of the determinant is given as the product of singular values.  This technique, which allows us to focus on the singular values rather than the eigenvalues, is at the heart of Girko's Hermitization technique in random matrix theory; see, for instance, \cite{MR3808330,MR773436,MR1310560,MR2130247,MR2046403,MR2908617,MR1428519,MR2409368,tao_random_2010-1} and references therein.  Our proofs of Theorems \ref{thm:replacement} and \ref{thm:replrank} utilize the fact that the singular values of $M-zI$ are stable under small perturbations (as well as low rank perturbations).  Heuristically, if none of the singular values of $M-zI$ are too extreme (to avoid the poles of  $\log |\cdot|$ at zero and infinity), the logarithmic potential $\mathcal{L}_M$ would also be stable.  
Our proof uses Weyl's perturbation theorem for singular values (cf. \eqref{eq:weyleig}), see Theorem 1.3 in \cite{chafai_lectnotes_2009} or Problem III.6.13 in \cite{Bhatia}. 
 
\begin{theorem}[Weyl's perturbation theorem for singular values] \label{thm:weyl}
If $M_1$ and $M_2$ are two $n \times n$ matrices, then 
\[ \max_{1 \leq j \leq n} | \sigma_j(M_1) - \sigma_j(M_2)| \leq \|M_1 - M_2\|. \]
\end{theorem} 

The proof of Theorem \ref{thm:non} is based on a Monte Carlo sampling method developed by Tao and Vu to study random polynomials \cite{TV-univ_zeros}.  

We will use Theorems \ref{thm:non}, \ref{thm:replacement}, and \ref{thm:replrank} to prove our other main results.  Let us focus on how we would prove parts \eqref{item:locallaw} and \eqref{item:pairing} from Theorem \ref{thm:sample}.  Recall that $\Jmat$ is defined in \eqref{eq:defT}, and let $\ptwo$ be the $n \times n$ deterministic matrix with $1$ in the $(n,1)$-entry and zeros everywhere else.  It is easy to check that the eigenvalues of $\Jmat + \ptwo$ are the $n$-th roots of unity.  In order to establish Theorem \ref{thm:sample}, we will compare the eigenvalues of $\Jmat + n^{-\gamma} E$ to the eigenvalues of $\Jmat + \ptwo$.  More generally, for any banded Toeplitz matrix $A$, we show the existence of a low-rank, deterministic matrix $A'$ so that $A + A'$ is a circulant matrix.  We then compare the eigenvalues of $A + n^{-\gamma}E$ to the eigenvalues of $A + A'$.   Since the eigenvalues of circulant matrices are well-known and easy to compute (see Lemma \ref{lemma:circulant}), we can give an explicit description of the eigenvalues of $A + A'$ in terms of the roots of unity and the symbol of $A$.  

We will use Theorem \ref{thm:non} to compare the eigenvalues of $A + n^{-\gamma}E$ to the eigenvalues of $A + A'$.  However, Theorem \ref{thm:non} requires estimating the difference between the logarithmic potentials of the two matrices.  For this we use Theorems \ref{thm:replacement} and \ref{thm:replrank}.  
Our method can be summarized by the following sequence of comparisons:
\[ \mathcal{L}_{A + n^{-\gamma}E} \approx \mathcal{L}_{A + A' + n^{-\gamma}E} \approx \mathcal{L}_{A + A'}, \]
where the first approximation utilizes Theorem \ref{thm:replrank} (since $A'$ has low rank) and the second uses Theorem \ref{thm:replacement} (since $n^{-\gamma}E$ has small norm).  For both of these approximations we will need to control both the largest and smallest singular values in order to apply Theorems \ref{thm:replacement} and \ref{thm:replrank}.  As is typical in non-Hermitian random matrix theory, the largest singular values are fairly easy to bound.  The smallest singular values of $A + n^{-\gamma}E$ and $A + A' + n^{-\gamma}E$ are controlled by using the randomness of $E$ (since we can view these matrices as deterministic perturbations of a random matrix); for these bounds we use several off-the-shelf results.  Since $A + A'$ is deterministic, we can no longer employ standard random matrix theory bounds.  Instead, we use the fact that Theorem \ref{thm:non} introduced an additional source of randomness.  In particular, since the value $Z$ from Theorem \ref{thm:non} is random, we need to control the smallest singular value of $A + A' - ZI$.  Here, we can use the fact that $Z$ is random (and continuously distributed) to show that, with high probability, it avoids the values of $z \in \mathbb{C}$ for which $A + A' - zI$ has a very small singular value.  

The proof of part \eqref{item:pairing} from Theorem \ref{thm:sample} is based on the local law from part \eqref{item:locallaw}.  Indeed, if we take $\varphi$ in part \eqref{item:locallaw} to be an approximate indicator function, we can approximate the number of eigenvalues of $\Jmat + n^{-\gamma} E$ (or more generally $A + n^{-\gamma} E$) by the number of roots of unity (alternatively, eigenvalues of $A + A'$) in the same region.  From here, we use a divide-and-conquer approach to divide the plane into small rectangular regions and apply the local law to an approximate indicator function on each rectangle.  While the method is fairly simple, the technical details require careful control of the error terms as well as a delicate balance between the size of each rectangle and the number of roots of unity that fall into them.  



\subsection{Related works} \label{sec:works}

The use of the logarithmic potential to study the eigenvalues of non-Hermitian random matrices has a long history in the field, including in the investigation of the famous circular law; we refer the reader to the works \cite{MR3808330,MR773436,MR1310560,MR2130247,MR2046403,MR1428519,MR2409368,tao_random_2010-1,GTcirc,TaoBook} and references therein as well as the survey \cite{MR2908617} for further historical details.   Our main methods for studying the logarithmic potential of the empirical spectral measure are similar to several works in the random matrix theory literature, including \cite{MR1428519,basak_regularization_2019,MR2409368,tao_random_2010-1}.  Many of the methods introduced in these works are comparison methods, which show how the spectrum of one matrix can be compared to another in order to compute the limiting eigenvalue distribution.  For instance, the results in \cite{MR2409368,tao_random_2010-1} show conditions under which small perturbations (or low rank perturbations) of random matrices do not change the limiting distributions of the eigenvalues.  The replacement principle (Theorem \ref{thm:TVrepl}), which was a major inspiration for the current article, is another example appearing in \cite{tao_random_2010-1}.  The non-asymptotic replacement principle introduced above was motivated by similar methods used by Tao and Vu to study roots of random polynomials \cite{TV-univ_zeros}.  

In recent years, a number of results have exploited the logarithmic potential to understand the local behavior of the eigenvalues of non-Hermitian random matrices, including local laws and related rates of convergence for the empirical spectral measure.  These results are too numerous to list in entirety but include \cite{1912.08856,MR3325952,MR3820329,MR3230002,MR3683369,MR3770875,MR3622892,MR3791390,MR3230004,MR3278919} and references therein.  

Perturbations of the matrix $\Jmat$ (defined in \eqref{eq:defT}) were investigated by Davies and Hager in \cite{MR2490477}.  Similar to our results, they investigated both random perturbations having small norm as well as low rank perturbations using a relevant Grushin problem.  
The pseudospectrum of Toeplitz matrices was investigated in \cite{MR1148398}.  
Random and deterministic perturbations of Toeplitz matrices have also been considered in \cite{MR1987715} and \cite{MR1972739}.  
The limiting distribution of the eigenvalues for Gaussian perturbations of non-normal matrices was investigated by \'{S}niady \cite{MR1929504} and later generalized by Guionnet, Zeitouni, and the second author \cite{MR3134007} by analyzing the logarithmic potential and using tools from free probability theory.  

Extensions of these results have appeared for the cases of Gaussian perturbations \cite{FPZeitouni_regularization_2015,MR4201590,MR3584192,MR4200678,basak_regularization_2019} as well as more general perturbations \cite{2001.09024,basak_spectrum_2020}.  In particular, the methods in \cite{MR4201590} are strong enough to handle general Toeplitz matrices.  We conjecture that our results should also hold for such a general class of Toeplitz matrices, but our present methods require the restriction to banded Toeplitz matrices.  
We also emphasize the connection with the work of Vogel and  Zeitouni in \cite{2001.09024}.  Similar to our results, Vogel and  Zeitouni provide a (nearly) deterministic comparison result for deterministic matrices subject to small random perturbations.  In specific cases, our main results can recover their bounds.  Our non-asymptotic results (such as Theorem \ref{thm:rate2}) extend their results in certain cases by providing a rate of convergence in Wasserstein distance.  
Very recently, the eigenvectors of random perturbations of Toeplitz matrices were investigated and shown to be localized by Basak, Vogel, and Zeitouni \cite{2103.17148}.  

Localized random perturbations of infinite banded Laurent matrices were studied in \cite{MR1870436}, where it is shown that the spectrum can be approximated by perturbed circulant matrices.  Our methods (discussed in Section \ref{sec:outline}) similarly use a connection between the eigenvalues of perturbed Toeplitz matrices and those of circulant matrices.  

We will discuss further details about how our results compare to these existing results in the literature after we introduce our main results in the sections below.

\subsection{A non-asymptotic example} \label{sec:examples}

\newcommand\eevent{\mathcal E_r}
\newcommand\smin{\sigma_{\mathrm{min}}}
\renewcommand\Pr\Prob

We will use Theorem \ref{thm:replacement} extensively in the coming sections.  Before doing so, we present a simpler case where Theorem \ref{thm:replacement} can be applied to perturbations of matrices in Jordan canonical form to derive a non-asymptotic result. 
Asymptotic results that cover similar types of matrices with general types of random perturbations have appeared in \cite{Wood2016,basak_spectrum_2020,2001.09024} and references therein.  

\begin{example}[$A$ has blocks of size $o(\log n)$] \label{eg:smallblocks}
The following is a demonstration of a non-asymptotic result that holds for values of $n$ as small as $11$. 

Let $\gamma\ge 4$ and $n \ge 11$, and suppose that $A$ is an $n$ by $n$ matrix in Jordan canonical form with all eigenvalues equal to zero for simplicity, and where each block has size at most $m:= \log(n)/t_n\ge 1$ for all $n \ge 1$, where $t_n \to \infty$ as $n \to \infty$ and $1\le t_n \le \log n$ for all $n$.  Let $M_1$ be an $n$ by $n$ matrix with iid~random entries each having absolute value at most $n$, and let $M_2$ be the $n$ by $n$ zero matrix.  By applying Theorem~\ref{thm:replacement} and Theorem~\ref{thm:non} using a smooth test function $\varphi$ with compact support, one can show that

$$\left| \int_{\mathbb C} \varphi \, d \mu_{A} - \int_{\mathbb C} \varphi \, d \mu_{A+n^{-\gamma} M_1} \right| \le C \left( n^{-\gamma/2} + 
e^{-\gamma t_n/20}\right),
$$
with probability at least $1-c'(e^{-\gamma t_n/10}+ e^{-3\gamma t_n/10})$, 
where $C$ and $c'$ are constants depending only on $\varphi$. This shows that even for small $n$, the eigenvalues of the perturbed matrix are all close to zero (the value of the unperturbed eigenvalues). We give a more detailed sketch of this example in Section~\ref{a:s2egs} of Appendix \ref{sec:prooftheory}. 
\end{example}

\section{Limiting empirical spectral measure} \label{sec:ESM}

This section studies the limiting empirical spectral distribution of the model $A+ n^{ - \gamma} E$, where $A$ is a banded Toeplitz matrix (as in Definition \ref{def:toep}) and $E$ is a random or deterministic perturbation.  Here, $\gamma > 0$ is chosen so that 
\begin{equation} \label{eq:normE2}
	n^{-\gamma} \| E \| = o(1).
\end{equation}  

While our strongest results are in Sections~\ref{sec:ratec}, \ref{sec:deterministic}, and \ref{sec:locallaw}, we have included this section since it illustrates some of our proof techniques in a much simpler setting compared to those presented in subsequent sections.   

We begin with our most general result for the limiting empirical spectral measure.  

\begin{theorem}\label{thm:pert-toep}
Let $\{a_j \}_{j \in \mathbb{Z}}$ be a sequence of complex numbers, indexed by the integers, so that
\begin{equation} \label{eq:aj}
	\sum_{j \in \mathbb{Z}} |j a_j| < \infty. 
\end{equation} 
Let $k_n$ be a sequence of non-negative integers that converges to $k \in [0 , \infty]$ as $n \to \infty$ and satisfies 
\begin{equation} \label{eq:kno}
	k_n = o \left( \frac{n}{\log n} \right). 
\end{equation} 
Let $A$ be an $n \times n$ Toeplitz matrix with symbol $\{a_j \}_{j \in \mathbb{Z}}$ truncated at $k_n$, and take $\gamma > 0$. Let $E$ be an $n \times n$ random matrix which satisfies:
\begin{enumerate}
\item there exists $\alpha \geq 0$ so that 
\begin{equation} \label{eq:normE}
	\|E\| = O(n^\alpha)
\end{equation}
with probability $1 - o(1)$; and
\item for almost every $z \in \mathbb{C}$, there exists $\kappa_z > 0$ so that 
\begin{equation} \label{eq:lsvE}
	\Prob (\sigma_{\min}(A + n^{-\alpha - \gamma}E - zI) \leq n^{-\kappa_z} ) = o(1). 
\end{equation} 
\end{enumerate} 
Then there exists a (deterministic) probability measure $\mu$ on $\mathbb{C}$ so that the empirical spectral measure $\mu_{A + n^{-\alpha - \gamma} E}$ of $A + n^{-\alpha - \gamma} E$ converges weakly in probability to $\mu$.  Moreover, $\mu$ is the distribution of 
\begin{equation} \label{eq:Uj}
	\sum_{|j| \leq k} a_j U^j, 
\end{equation} 
where $U$ is a random variable uniformly distributed on $S^1$.  Here, we use the convention that if $k = \infty$, 
\[ \sum_{|j| \leq k} a_j U^j = \sum_{j \in \mathbb{Z}} a_j U^j. \]
\end{theorem}

\begin{remark} \label{rem:aj}
Condition \eqref{eq:aj} is analogous to Eq. (1.2) in \cite{MR4201590}.  We also note that condition \eqref{eq:aj} implies 
\[ \sum_{j \in \mathbb{Z}} |a_j|^2 < \infty, \]
and we will use this fact in the forthcoming proofs.  
\end{remark}

\begin{figure}
\includegraphics[scale=0.4]{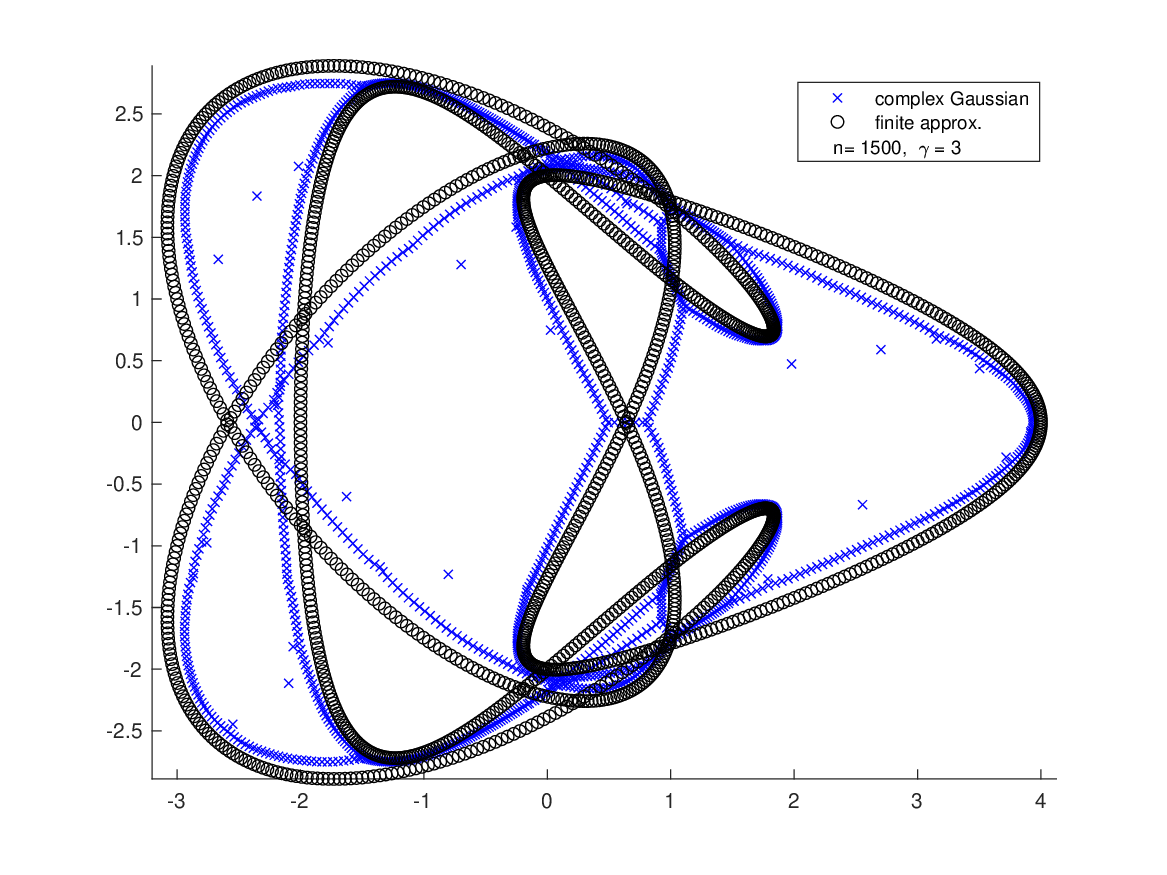}
\includegraphics[scale=0.4]{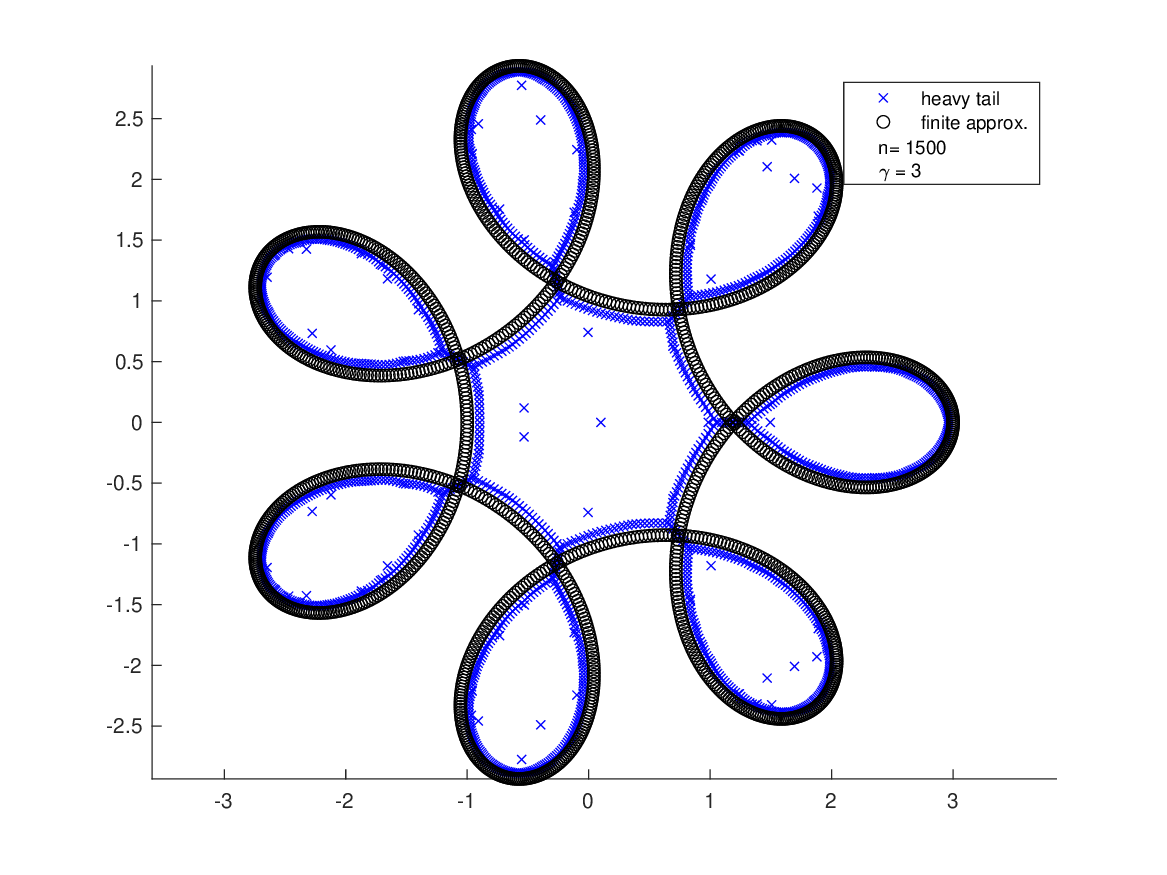}
(a) \hspace{3in} (b)
\medskip

\includegraphics[scale=0.4]{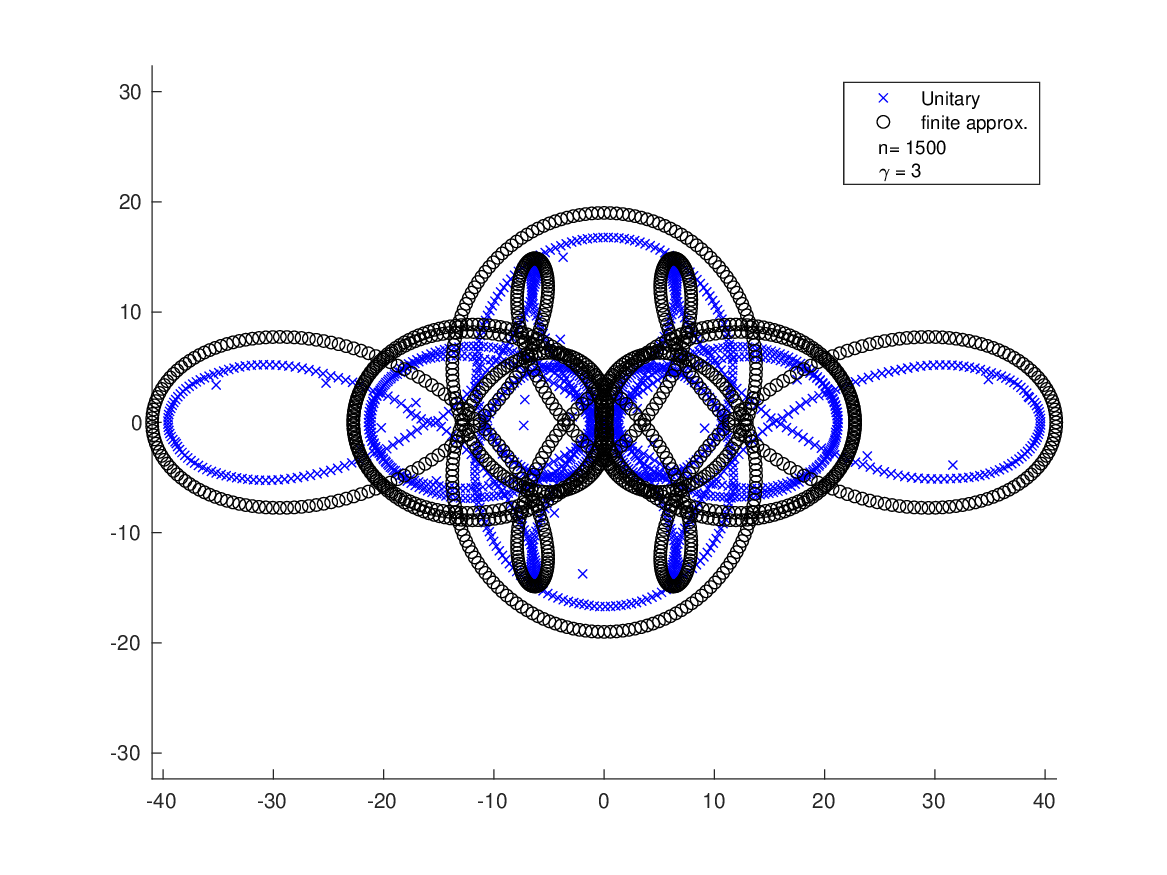}
\includegraphics[scale=0.4]{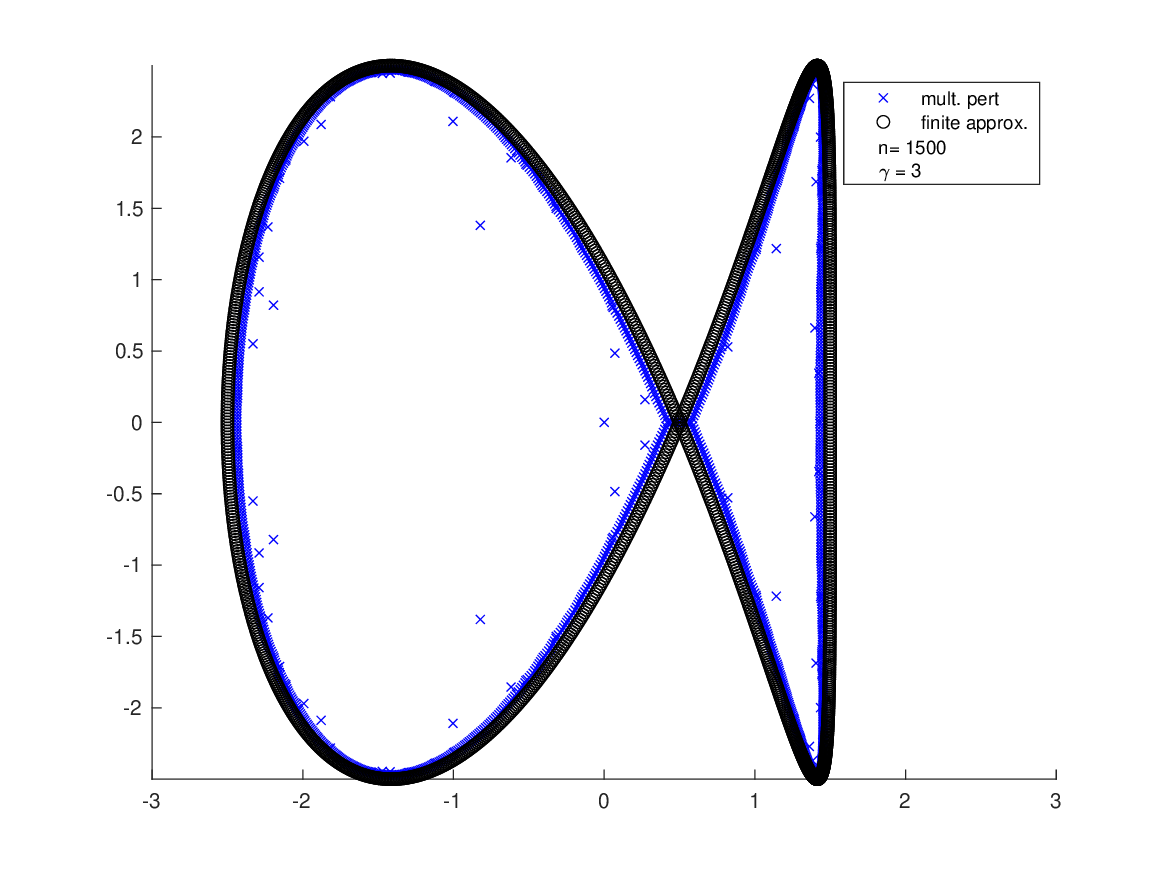}
(d) \hspace{3in} (c)
\caption{Above are plots of different Toeplitz matrices perturbed by different types of random noise.  In each case, $n=1500$ and the noise is scaled by $n^{-\gamma}$ where $\gamma = 3$.  Clockwise from the upper left, we have 
(a) $M_1+n^{-\gamma}G$ where $G$ has iid~complex standard Gaussian entries (see Corollary~\ref{cor:pert-toep} (1)),
(b) $M_2+n^{-\gamma}H$ where $H$ has iid~entries with heavy-tailed distribution $U_{[0,1]}^{-1/2}$ where $U_{[0,1]}$ is uniform on the real interval $[0,1]$ (see Theorem~\ref{thm:pert-toep}), 
(c) $M_3(I+n^{-\gamma}G)$, where $G$ has iid~real standard Gaussian entries (see Corollary~\ref{cor:multi}),
and 
(d) $M_4+n^{-\gamma}U$ where $U$ is a Haar unitary matrix (see Corollary~\ref{cor:pert-toep} (3)).
The Toeplitz matrices $M_i$ are defined as follows (any values of $a_j$ not specified are assumed to be zero): the symbol for $M_1$ is defined by $(a_{-5},\ldots,
a_6)= (1,0,1,0,1,0,0,1,0,-1,0,1)$, the symbol for $M_2$ is defined by $(a_{-1},a_0,\ldots,a_5,a_6)=(2,0,\ldots,0,1)$, the symbol for 
$M_3$ is defined by $(a_{-2},a_{-1},a_0,a_1,a_2) = (1,1,0,1,-1.5)$, and
the symbol for $M_4$ is defined by $(a_{-11}, \ldots, 
a_5)=(10,0,0,0,1,0,1,0,9,0,1,0,5,0,9,0,5)$.  The finite approximation comes from a corresponding circulant matrix, which is described in Lemma~\ref{lemma:lowrank}.
}
\label{fig:zoo}
\end{figure}

While we could not locate the results in this section in the literature, it is likely that some existing results on Toeplitz matrices (such as \cite{MR4201590,
basak_regularization_2019,
basak_spectrum_2020,MR4200678}) can be modified to cover the banded Toeplitz matrices we study here. One notable feature of Theorem~\ref{thm:pert-toep} is that it covers Toeplitz matrices with a growing number of diagonals and for general random perturbations, and further makes no additional assumptions on $\gamma$ beyond $\gamma >0$.  References \cite{basak_regularization_2019, basak_spectrum_2020} study Toeplitz matrices with a fixed number of diagonals and apply when $\gamma >0$ with Gaussian perturbations or general random perturbations, respectively.
The main result of \cite{MR4201590} applies to Toeplitz matrices with any number of diagonals under Gaussian perturbations with $\gamma > 2.5$, and \cite{MR4200678} studies a slowly growing number of diagonals with general random perturbations and $\gamma>3.5$; both of these papers also produce more refined information on the locations of eigenvalues, which can be compared to Theorem~\ref{thm:rate2} below.  

The condition on $k_n$ in \eqref{eq:kno} is likely an artifact of our proof, and we believe this condition can be relaxed.  
The assumptions on $E$ in Theorem \ref{thm:pert-toep} are very general and apply to a wide range of random matrix ensembles including matrices with iid light-tailed entries, matrices with iid heavy-tailed entries, elliptic random matrices, and random unitary matrices.   In the following corollary we specialize Theorem \ref{thm:pert-toep} to a few of these cases.  See Figure~\ref{fig:zoo} for eigenvalue plots in some example cases.

\begin{corollary} \label{cor:pert-toep}
Let $\{a_j \}_{j \in \mathbb{Z}}$ be a sequence of complex numbers, indexed by the integers, so that \eqref{eq:aj} holds.  
Let $k_n$ be a sequence of non-negative integers that converges to $k \in [0 , \infty]$ as $n \to \infty$ and satisfies \eqref{eq:kno}.  
Let $A$ be an $n \times n$ Toeplitz matrix with symbol $\{a_j \}_{j \in \mathbb{Z}}$ truncated at $k_n$.  
Assume one of the following conditions on the $n \times n$ matrix $E$ and the parameter $\gamma > 0$:
\begin{enumerate}
\item $E$ is an $n \times n$ random matrix whose entries are iid copies of a random variable with mean zero, unit variance, and finite fourth moment, and $\gamma > 1/2$.
\item $E$ is an $n \times n$ random matrix whose entries are iid copies of a random variable with mean zero and unit variance, and $\gamma > 1$.
\item $E$ is an $n \times n$ random matrix uniformly distributed on the unitary group $\mathcal{U}(n)$, and $\gamma > 0$.  
\end{enumerate}
Then there exists a (deterministic) probability measure $\mu$ on $\mathbb{C}$ so that the empirical spectral measure $\mu_{A + n^{ - \gamma } E}$ of $A + n^{ - \gamma } E$ converges weakly in probability to $\mu$.  Moreover, in all three cases $\mu$ is the distribution of the random variable given in \eqref{eq:Uj}, 
where $U$ is a random variable uniformly distributed on $S^1$.  
\end{corollary}
\begin{proof}
In order to utilize Theorem \ref{thm:pert-toep}, we only need to verify that $E$ satisfies \eqref{eq:normE} and \eqref{eq:lsvE} in each of the cases above.  For the first case, these bounds (with $\alpha = 1/2$) follow from \cite[Theorem 5.8]{BSbook} and \cite[Theorem 2.1]{MR2409368}.  The second case follows (with $\alpha =1$) from \cite[Theorem 2.1]{MR2409368} and the fact that
\[ \frac{1}{n^2} \| E \|^2 \leq \frac{1}{n^2} \|E\|_2^2 \longrightarrow 1 \]
almost surely by the law of large numbers.  In the third case, $\|E\| = 1$ since $E$ is unitary.  In this last case, the least singular value bound in \eqref{eq:lsvE} follows from \cite[Theorem 1.1]{MR3164983}.  
\end{proof}

The choice of $\gamma$ in each of the three cases given in Corollary \ref{cor:pert-toep} is so that \eqref{eq:normE2} is satisfied.    
Theorem \ref{thm:pert-toep}  and its corollary are similar to several recent works concerning fixed matrices perturbed by random matrices \cite{FPZeitouni_regularization_2015,MR4168388,basak_regularization_2019,basak_spectrum_2020,MR4201590,MR3584192,MR4200678,2001.09024}.  The case when $E$ contains independent Gaussian entires was investigated in \cite{FPZeitouni_regularization_2015,MR4201590,MR3584192,MR4200678,basak_regularization_2019}.  Since Theorem \ref{thm:pert-toep} applies to a large class of perturbations $E$, it is more closely related to the results in \cite{2001.09024,basak_spectrum_2020}.  Our assumptions on $E$ differ from these previous works, allowing us to apply our results to a slightly different classes of perturbations, including deterministic matrices (see Section \ref{sec:deterministic}).  In addition, our results allow the bandwidth of $A$ to grow with $n$, which differs from the results in \cite{basak_spectrum_2020}.  See Figure \ref{fig:growing} for an example with growing bandwidth.
Our method of proof also differs from the methods used in \cite{2001.09024,basak_spectrum_2020} and suggests a concrete way to accurately approximate the eigenvalues for reasonably small $n$ by comparing them to the eigenvalues of a deterministic circulant matrix; see Conjecture~\ref{conj:A'} for further details about this approximation.  Some numerical simulations are presented in Figure~\ref{fig:star}.  We also note that Theorem \ref{thm:pert-toep} allows us to consider multiplicative perturbations (see also Figure~\ref{fig:zoo}(c)).

\begin{figure}
\includegraphics[scale=0.4]{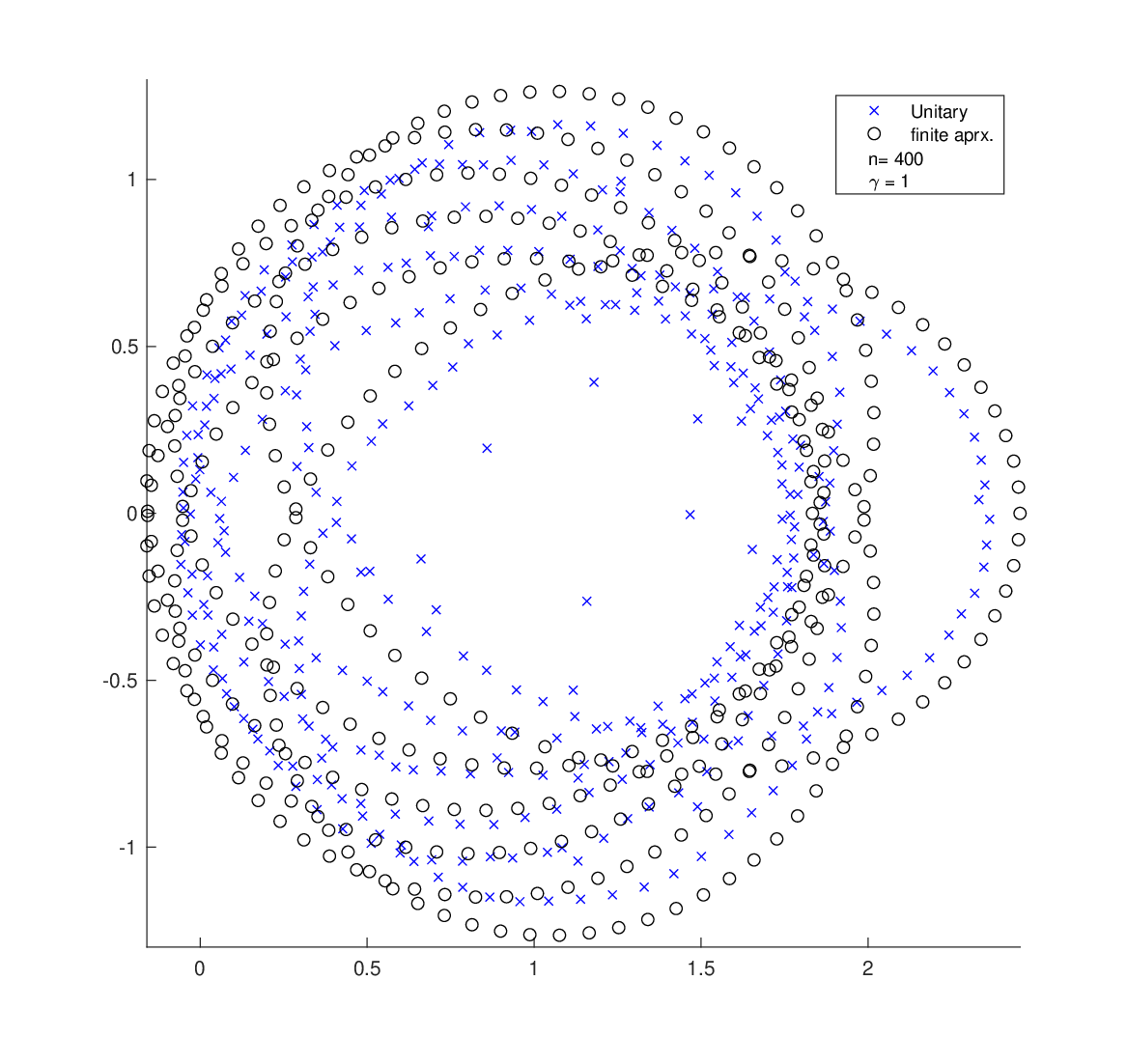}
\includegraphics[scale=0.4]{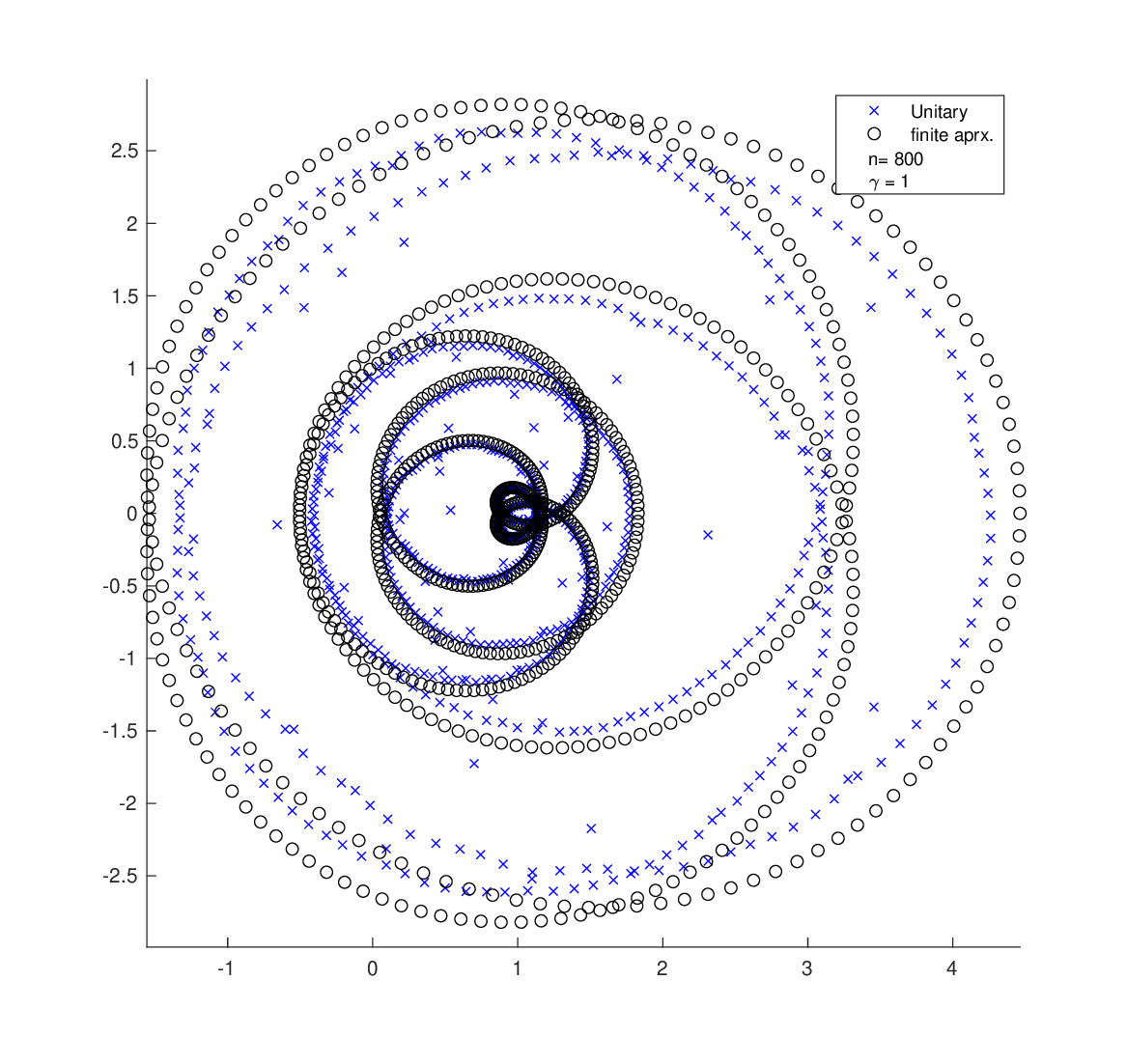}

\includegraphics[scale=0.4]{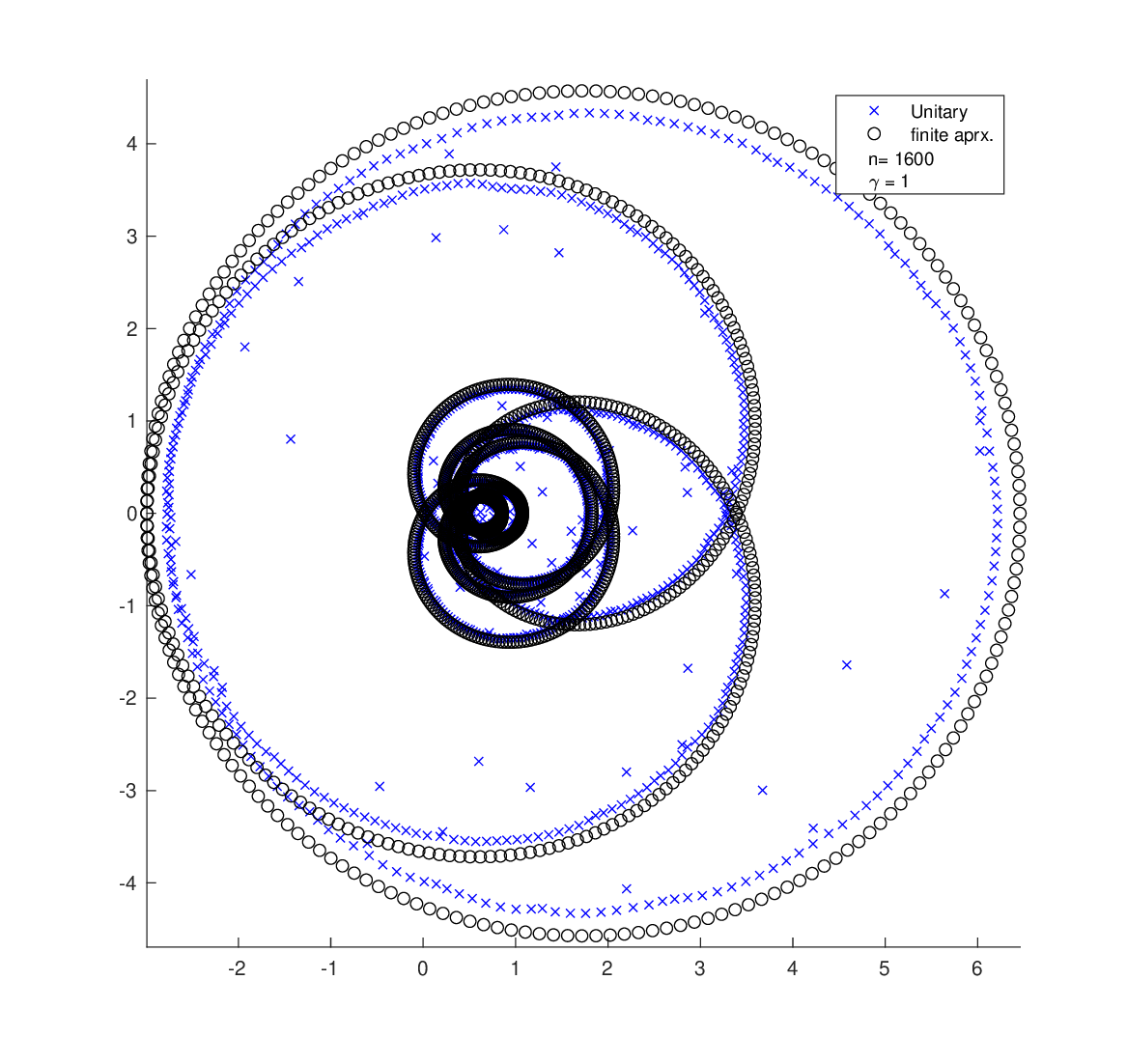}
\includegraphics[scale=0.4]{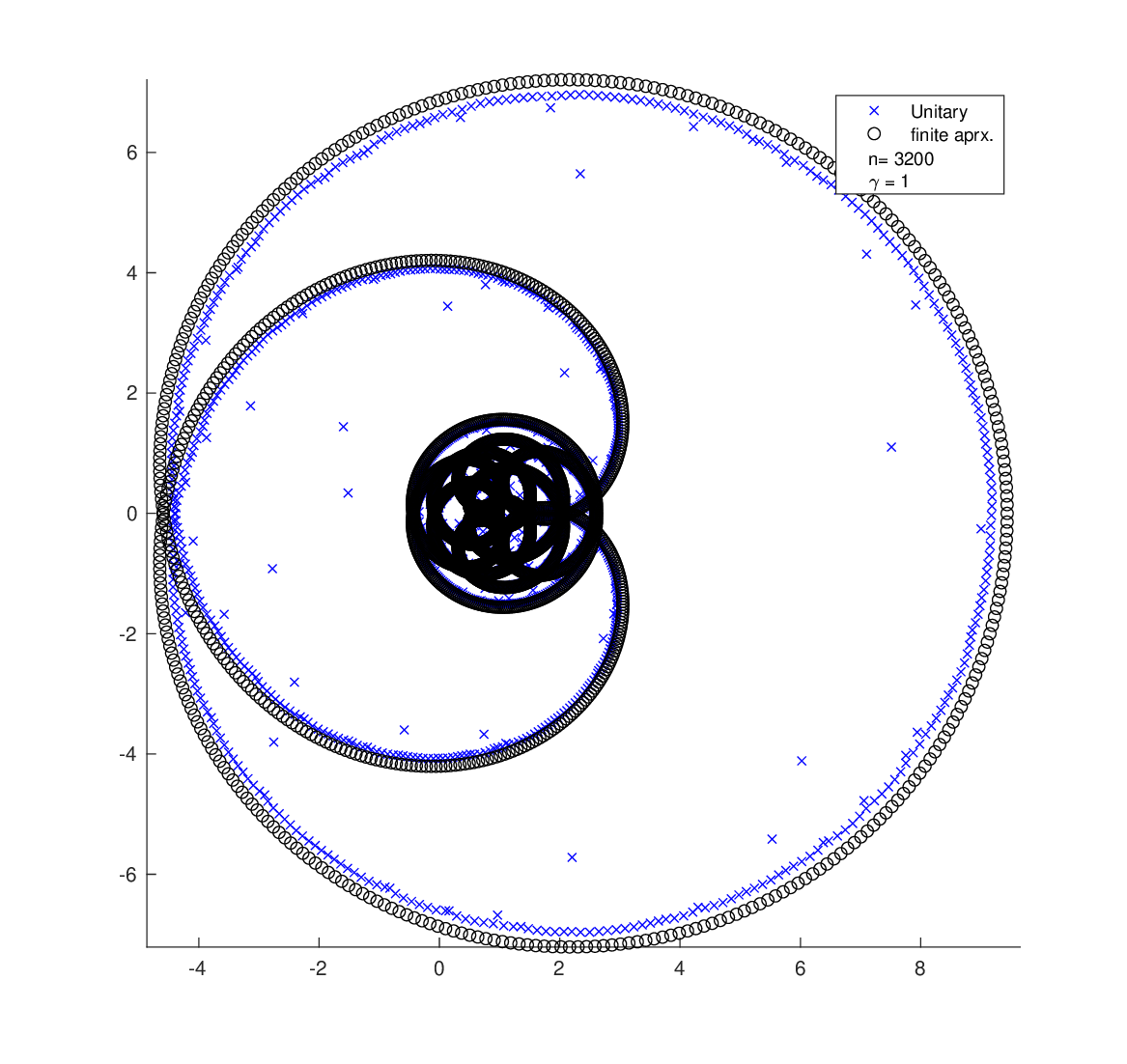}
\caption{Above are plots of the eigenvalues of a perturbed Toeplitz matrix $M$ truncated at $k_n$, where $k_n = \left\lfloor n^{1/3}\right\rfloor-1$ (which tends to infinity with $n$) and where the symbol for $M$ is defined by $a_{-j}=1/(j+1)^{2.1}$ for $j\ge 0$ and 
and $a_j = (0.999)^j$ for $j \ge 6$ and $a_j=0$ for $1\le j\le 5$. 
The eigenvalues of $M+n^{-\gamma} U$, where $U$ is a Haar uniform random unitary matrix, are plotted as blue \color{blue}$\times$\color{black}\  symbols,
and the eigenvalues of the corresponding finite approximation are plotted as a black circles ($\circ$). By Theorem~\ref{thm:pert-toep} (see also Corollary~\ref{cor:pert-toep} part (3)), the empirical spectral measure of $M+n^{-\gamma} U$ converges weakly in probability to a deterministic measure, which is also the limit of the finite approximations. 
The finite approximations are the eigenvalues of the circulant matrix $C$ described in Section~\ref{sec:circ:tools}. 
The finite approximation appears to be reasonably close to the random eigenvalues in this case, even though the finite approximation is still visibly changing for the values of $n$ shown above.
}
\label{fig:growing}
\end{figure}

\begin{figure}
\centerline{$n=50$\hspace{3in} $n=200$}
\includegraphics[scale=.4]{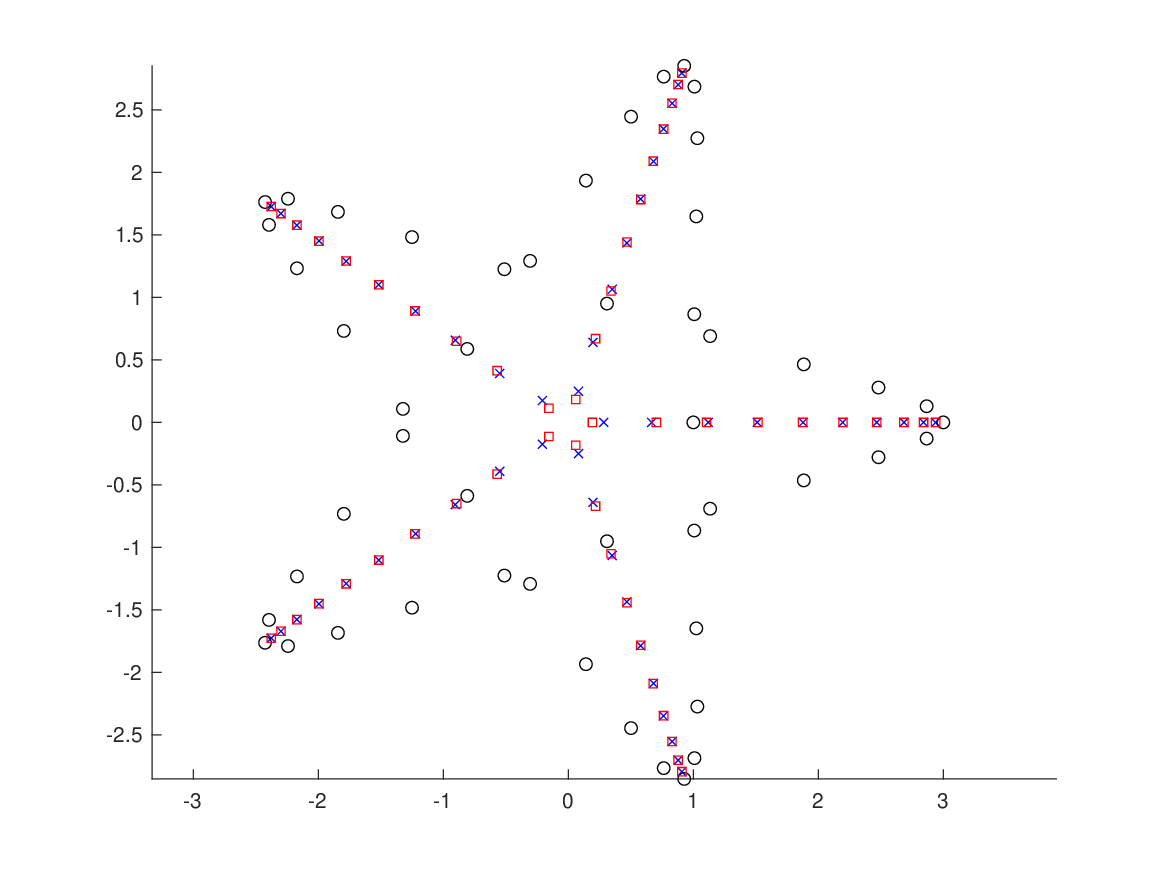}
\includegraphics[scale=.4]{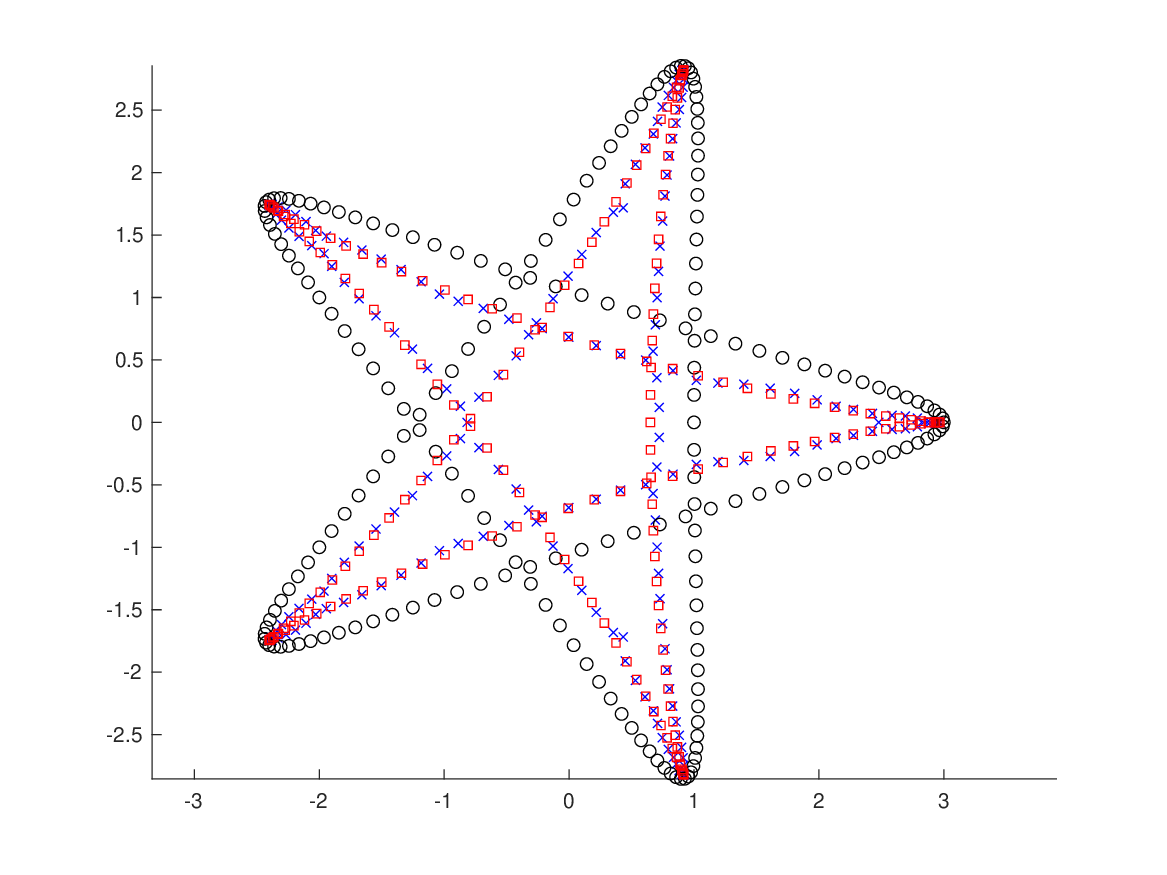}

\centerline{$n=500$\hspace{3in} $n=2000$}
\includegraphics[scale=.4]{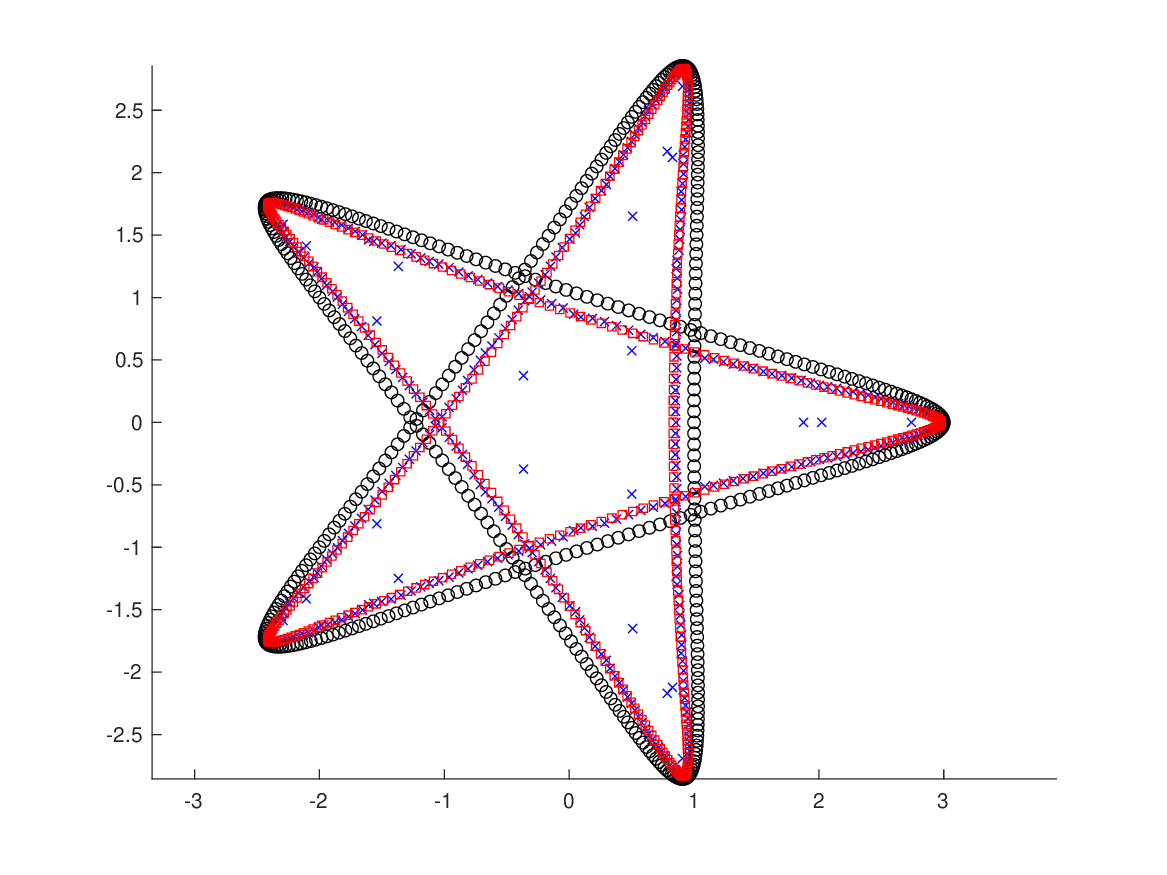}
\includegraphics[scale=.4]{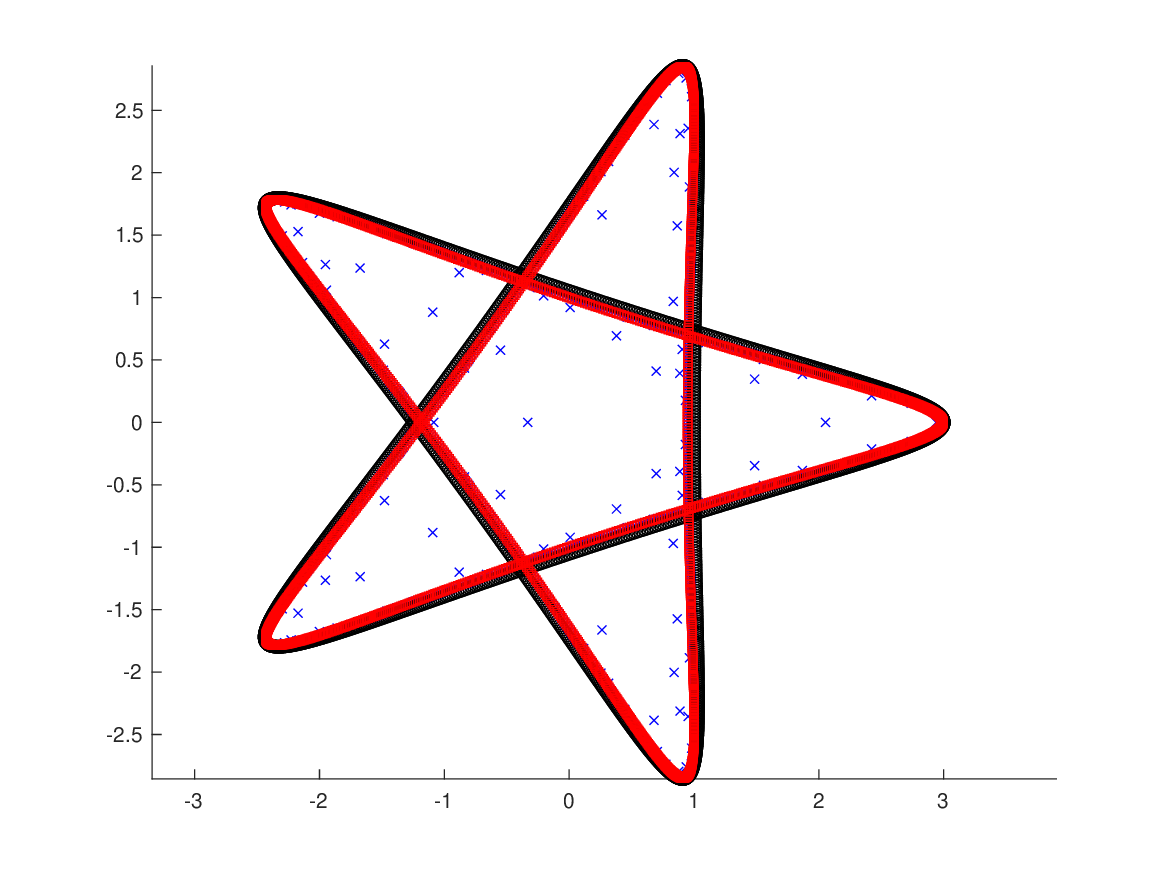}
\caption{Above is a series of plots of the eigenvalues of a perturbed Toeplitz matrix $M$ with symbol defined by $a_{-2}=2$, 
$a_3 =-1$ and all other $a_j$ equal to zero.  The plots are for values of $n$ ranging from 100 to 4000 and also show two different approximations of the eigenvalues.  
The eigenvalues of $M +n^{-2} G$, where $G$ is an iid~real Gaussian perturbation, are plotted as blue \color{blue}$\times$\color{black}\  symbols, and the eigenvalues of the corresponding finite approximation
are plotted as black circles ($\circ$). The finite approximation comes from the 
circulant matrix $C$ (with $c_{n-j}=a_{-j}$ for $1\le j\le n-1$ and $c_j=a_j$ for $0\le j \le n-1$%
\details{; so here $c_{n-2} = 2$
and $c_3=-1$ and all other $c_j$ equal zero}) 
described in Section~\ref{sec:circ:tools}.  
By Theorem~\ref{thm:pert-toep} and Lemma~\ref{lemma:circulant}, we know that the spectral measure $\mu_{M+n^{-2}G}$ is approximated by the spectral measure $\mu_C$.  
Interestingly, the scaled matrix described in Conjecture~\ref{conj:A'}, whose eigenvalues are plotted as red squares ($\color{red} \square \color{black}$), appears to be a very good match for the eigenvalues of $M+n^{-2}G$, even for small values of $n$. }
\label{fig:star}
\end{figure}

\begin{corollary}[Multiplicative perturbation] \label{cor:multi}
Let $\{a_j \}_{j \in \mathbb{Z}}$ be a sequence of complex numbers, indexed by the integers, so that \eqref{eq:aj} holds.  
Let $k_n$ be a sequence of non-negative integers that converges to $k \in [0 , \infty]$ as $n \to \infty$ and satisfies \eqref{eq:kno}.  
Let $A$ be an $n \times n$ Toeplitz matrix with symbol $\{a_j \}_{j \in \mathbb{Z}}$ truncated at $k_n$.  Let $E'$ be an $n \times n$ random matrix whose entries are iid copies of a real standard normal random variable, and take $\gamma > 1$.  
Then there exists a probability measure $\mu$ on $\mathbb{C}$ so that the empirical spectral measure $\mu_{A  (I + n^{ - \gamma -1/2} E')}$ of $A (I+ n^{ - \gamma -1/2} E')$ converges weakly in probability to $\mu$.  Moreover, $\mu$ is the distribution given in \eqref{eq:Uj}, 
where $U$ is a random variable uniformly distributed on $S^1$.  
\end{corollary}
\begin{proof}
Since $A  (I + n^{ - \gamma -1/2} E') = A + n^{-\gamma-1/2} A E'$, we will apply Theorem \ref{thm:pert-toep} with $E := A E'$.  It remains to check that $E$ satisfies \eqref{eq:normE} and \eqref{eq:lsvE}.  The bound in \eqref{eq:normE} with $\alpha = 1/2$ follows from Proposition \ref{prop:Anorm} (which provides an upper bound on $\|A\|$) and \cite[Corollary 5.35]{MR2963170}  (which provides an upper bound on $\|E'\|$).  The least singular value bound in \eqref{eq:lsvE} follows from Proposition \ref{prop:lsv}.  
\end{proof}

We now turn to the proof of Theorem \ref{thm:pert-toep}.  We begin with some preliminary details about circulant matrices.  

\subsection{Circulant matrices}

A {circulant matrix} is an $n \times n$ matrix of the form 
\begin{equation} \label{eq:circ}
	C=
\begin{bmatrix}
c_0     & c_{n-1} & \dots  & c_{2} & c_{1}  \\
c_{1} & c_0    & c_{n-1} &         & c_{2}  \\
\vdots  & c_{1}& c_0    & \ddots  & \vdots   \\
c_{n-2}  &        & \ddots & \ddots  & c_{n-1}   \\
c_{n-1}  & c_{n-2} & \dots  & c_{1} & c_0 \\
\end{bmatrix},
\end{equation} 
where $c_0, c_1, \ldots, c_{n-1} \in \mathbb{C}$.  In this case, we say $c_0, c_1, \ldots, c_n$ {generate} the matrix $C$.  

\begin{lemma}[Properties of circulant matrices] \label{lemma:circulant}
Let $C$ be the $n \times n$ circulant matrix in \eqref{eq:circ} generated by $c_0, c_1, \ldots, c_{n-1} \in \mathbb{C}$.  
\begin{enumerate}[(i)]
\item The eigenvalues of $C$ are given by 
\begin{equation} \label{eq:circeig}
	\lambda_j = c_0 + c_{n-1} \omega_n^j + c_{n-2} \omega_n^{2j} + \cdots c_1 \omega_n^{(n-1)j}, \qquad j = 0, 1, \ldots, n-1, 
\end{equation} 
where 
\begin{equation} \label{eq:omegan}
	\omega_n := \exp \left( \frac{2 \pi \sqrt{-1}}{n} \right)
\end{equation} 
is a primitive $n$-th root of unity.
The corresponding normalized eigenvectors are given by
\[ v_j := \frac{1}{\sqrt{n}} (1, \omega_n^j, \omega_n^{2j}, \ldots, \omega_n^{(n-1)j})^{\mathrm{T}}, \qquad j = 0, 1, \ldots, n-1. \]
\item $C$ is a normal matrix.
\item \label{item:circsing} The (unordered) singular values of $C$ are given by 
\[ \sigma_j := |\lambda_j|, \qquad j = 0, 1, \ldots, n-1, \]
where $\lambda_0, \lambda_1, \ldots, \lambda_{n-1}$ are the eigenvalues of $C$ given in \eqref{eq:circeig}.  
\end{enumerate}
\end{lemma}
\begin{proof}
A simple computation shows that $C$ has eigenvalue $\lambda_j$ with corresponding eigenvector $v_j$ for $j=0, 1, \ldots, n-1$.  In addition, since the eigenvectors $v_0, \ldots, v_{n-1}$ are orthonormal, it follows from Theorem 2.5.3 in \cite{HJ} that $C$ is normal.  Part \eqref{item:circsing} follows since the singular values of any normal matrix are given by the absolute values of its eigenvalues; see 2.6.P15 in \cite{HJ}.  
\end{proof}

\subsection{Required tools} \label{sec:circ:tools}

We begin the proof of Theorem \ref{thm:pert-toep} with a number of lemmata.  

\begin{lemma} \label{lemma:measurezero}
Let $\{a_j \}_{j \in \mathbb{Z}}$ be a sequence of complex numbers, indexed by the integers, which satisfy \eqref{eq:aj}.  
For each $n \geq 1$, let $k_n$ be a non-negative integer, and define a sequence of functions $f_n: S^1 \to \mathbb{C}$ by 
\[ f_n(\omega) = \sum_{|j| \leq k_n} a_j w^j. \]
In addition, let the function $f: S^1 \to \mathbb{C}$ be given by
\begin{equation} \label{def:fS1}
	f(\omega) = \sum_{j \in \mathbb{Z}} a_j w^j. 
\end{equation} 
Then the following properties hold: 
\begin{enumerate}[(i)]
\item \label{item:measurezero} the image $f(S^1)$ has Lebesgue measure zero in $\mathbb{C}$; 
\item \label{item:unionzero} the set of images $\cup_{n=1}^\infty f_n(S^1)$ has Lebesgue measure zero in $\mathbb{C}$; and
\item \label{item:distance} if $k_n \to \infty$ as $n \to \infty$, then, for each $z \in \mathbb{C} \setminus f(S^1)$, there exists $\eps > 0$ and $N \in \mathbb{N}$ so that $z$ is at least distance $\eps$ from $\cup_{n=N}^\infty f_n(S^1)$.
\end{enumerate}
\end{lemma}
\begin{proof}
Under condition \eqref{eq:aj}, it follows from Theorem 6.28 in \cite{ProbTheory} that the function $t \mapsto f( e^{\sqrt{-1} t})$ is differentiable and 
\[ \frac{d}{dt} f( e^{\sqrt{-1} t}) = \sum_{j \in \mathbb{Z}} j \sqrt{-1} a_j e^{\sqrt{-1} tj }. \]
In particular, the derivative satisfies the bound
\[ \left| \frac{d}{dt} f( e^{\sqrt{-1} t})  \right| \leq \sum_{j \in \mathbb{Z}} |j a_j| < \infty \]
for all $t \in \mathbb{R}$.  By Lemma 7.25 in \cite{RCRudin}, it follows that $f(S^1)$ has Lebesgue measure zero in $\mathbb{C}$.  

Conclusion \eqref{item:unionzero} now follows from \eqref{item:measurezero}.  Indeed, by taking $a_j = 0$ for all $|j| > k$,  conclusion \eqref{item:measurezero} also applies to functions of the form
\[ f(\omega) = \sum_{|j| \leq k} a_j \omega^j \]
for any integer $k \geq 0$.  Since the countable union of sets of Lebesgue measure zero has Lebesgue measure zero, the proof of \eqref{item:unionzero} is complete.  

In order to prove \eqref{item:distance}, let $z \in \mathbb{C} \setminus f(S^1)$. Since $f$ is differentiable, $f$ is continuous, and hence $f(S^1)$ is compact.  Thus, there exists $\eps > 0$ with the property that $z$ is distance $2\eps$ from $f(S^1)$.  Since $k_n \to \infty$, it follows that there exists $N \in \mathbb{N}$ so that 
\[ \sup_{\omega \in S^1} |f_n(\omega) - f(\omega)| \leq \sum_{|j| > k_n} |a_j| < \eps \]
for all $n \geq N$.  This implies that  $z$ is at most distance $\eps$ from $f_n(S^1)$ for each $n \geq N$. Indeed, if there exists $n \geq N$ and $\omega \in S^1$ so that $|z - f_n(\omega)| < \eps$, then $|z - f(\omega)| < 2 \eps$, a contradiction.   
\end{proof}

Recall the definition of the $L^1$-Wasserstein distance $W_1(\mu, \nu)$ between $\mu$ and $\nu$ given in \eqref{eq:defW1}.  

\begin{lemma} \label{lemma:wassbnd}
Let $k$ be a non-negative integer, and define $f: S^1 \to \mathbb{C}$ by
\[ f(\omega) = \sum_{|j| \leq k} a_j \omega^j \]
for some complex numbers $a_j$ with $-k \leq j \leq k$.  Let $\mu$ be distribution of $f(U)$, where $U$ is a random variable uniformly distributed on $S^1$, and let $\mu_n$ be the probability measure given by 
\[ \mu_n = \frac{1}{n} \sum_{i=0}^{n-1} \delta_{f(\omega_n^i)}, \]
where $\omega_n$ is a primitive $n$-th root of unity as in \eqref{eq:omegan}.  Then, there exists a constant $C >0$ (depending only on $k$, $f$, and $\{a_j\}_{|j| \leq k}$) so that 
\begin{equation} \label{eq:wassbnd}
	W_1(\mu_n, \mu) \leq \frac{C}{n}. 
\end{equation} 
\end{lemma}
\begin{proof}
Let $\theta$ be a random variable, uniform on $[0, 2\pi)$.  Define another random variable $\phi_n$ as follows.  Take $\phi_n = 2 \pi \frac{i-1}{n}$ whenever $2 \pi \frac{i-1}{n} \leq \theta < 2 \pi \frac{i}{n}$ for some integer $1 \leq i \leq n$.  Then $U := e^{\sqrt{-1} \theta}$ is uniformly distributed on $S^1$ and, taking $\psi_n := e^{\sqrt{-1} \phi_n}$, we see that $f(\psi_n)$ has distribution $\mu_n$.  By construction, it follows that
\begin{equation} \label{eq:Upsin}
	| U - \psi_n | \leq \frac{2 \pi }{n} 
\end{equation} 
almost surely.  We then see that
\[ |f(U) - f(\psi_n)| \leq \sum_{|j| \leq k} |a_j| |U^j - \psi_n^j| \leq C | U - \psi_n| \]
for some constant $C > 0$ depending only on $k$, $f$, and $\{a_j\}_{|j| \leq k}$.  In view of \eqref{eq:Upsin}, we conclude that
\[ \E |f(U) - f(\psi_n)| \leq \frac{2 \pi C}{n}, \]
which in the language of measures (and after adjusting the constant $C$) gives \eqref{eq:wassbnd}.  
\end{proof}

\begin{lemma} \label{lemma:distconvergence}
Let $\{a_j \}_{j \in \mathbb{Z}}$ be a sequence of complex numbers, indexed by the integers, satisfying 
\begin{equation} \label{eq:jaj2}
	\sum_{j \in \mathbb{Z}} |a_j| < \infty,
\end{equation} 
and let $k_n$ be a non-negative integer sequence tending to infinity.  Define $f, f_n: S^1 \to \mathbb{C}$ by
\[ f_n(\omega) = \sum_{|j| \leq k_n} a_j w^j \qquad \text{ and } \qquad f(\omega) = \sum_{j \in \mathbb{Z}} a_j w^j. \]
Let $\mu_n$ to be the probability measure on $\mathbb{C}$ given by 
\[ \mu_n = \frac{1}{n} \sum_{i=0}^{n-1} \delta_{f_n(\omega_n^i)}, \]
where $\omega_n$ is a primitive $n$-th root of unity as in \eqref{eq:omegan}.  Let $\mu$ be the distribution of $f(U)$, where $U$ is a random variable, uniform on $S^1$.  Then $\mu_n \rightarrow \mu$ weakly as $n \to \infty$.  
\end{lemma}
\begin{proof}
Let $\psi_n$ be a random variable uniformly distributed on $1, \omega_n, \omega_n^2, \ldots, \omega_n^{n-1}$.  Then $f_n(\psi_n)$ has distribution $\mu_n$.  Since $k_n$ tends to infinity, it follows from \eqref{eq:jaj2} that
\[ \sup_{\omega \in S^1} |f_n(w) - f(w)| \leq \sum_{|j| > k_n} |a_j| = o(1). \]
This implies that 
\[ |f_n(\psi_n) - f(\psi_n)| = o(1) \]
almost surely.  

It remains to show that $f(\psi_n)$ converges in distribution to $f(U)$.  As a consequence of Lemma \ref{lemma:wassbnd} (by taking the identity function for $f$), it follows that $\psi_n \to U$ in distribution as $n \to \infty$ (since convergence in Wasserstein distance implies convergence in distribution; see \cite{Dudley}).  In addition, under condition \eqref{eq:jaj2}, it follows from Theorem 6.27 in \cite{ProbTheory} that $f$ is continuous.  Thus, by the continuous mapping theorem it follows that $f(\psi_n) \to f(U)$ in distribution as $n \to \infty$.  Combining the above, we see that $f_n(\psi_n) \to f(U)$ in distribution, which in the language of measures means that $\mu_n \to \mu$ weakly as $n \to \infty$.  
\end{proof}

\begin{lemma} \label{lemma:lowrank}
Let $\{a_j \}_{j \in \mathbb{Z}}$ be a sequence of complex numbers, indexed by the integers, and let $k < \frac{n}{2}$ be a non-negative integer.  Define $f: S^1 \to \mathbb{C}$ by
\[ f(\omega) = \sum_{|j| \leq k} a_j \omega^j \qquad. \]
Let $A$ be the $n \times n$ Toeplitz matrix with symbol $\{a_j \}_{j \in \mathbb{Z}}$ truncated at $k$.  Then there exists an $n \times n$ matrix $A'$ with rank at most $2k$ so that $A + A'$ is circulant with eigenvalues given by $f(1), f(\omega_n), f(\omega_n^2), \ldots, f(\omega_n^{n-1})$, where $\omega_n$ is a primitive $n$-th root of unity as in \eqref{eq:omegan}.  Moreover, $A'$ can be chosen so that the Frobenius norm $\|A'\|_2$ satisfies 
\begin{equation} \label{eq:A'F}
	\|A'\|_2^2 \leq n \sum_{|j| \leq k } |a_j|^2. 
\end{equation} 
\end{lemma}
\begin{proof}
We define $A'$ so that $C = A + A'$, where $C$ is the circulant matrix given in \eqref{eq:circ} with $c_{j} = a_j$ for $0 \leq j \leq k$, $c_{n-j} = a_{-j}$ for $1 \leq j \leq k$, and $c_j = 0$ otherwise.  Then $A'$ has at most $2k$ rows that are non-zero, and hence $\rank(A') \leq 2k$.  In view of Lemma \ref{lemma:circulant}, the eigenvalues of $C^{\mathrm{T}}$ (and hence of $C$) are given by $f(1), f(\omega_n), f(\omega_n^2), \ldots, f(\omega_n^{n-1})$.  The bound in \eqref{eq:A'F} follows from the construction above.  
\end{proof}

\subsection{Proof of Theorem \ref{thm:pert-toep}}
With the above results in hand, we can now complete the proof of Theorem \ref{thm:pert-toep}.

\begin{proof}[Proof of Theorem \ref{thm:pert-toep}]
If $k_n$ converges to an integer $k$, then it must be the case that $k_n = k$ for all sufficiently large $n$.  In this case, we may assume that $a_j = 0$ for all $|j| > k$.  Therefore, without loss of generality, it suffices to assume that $k_n$ tends to infinity and satisfies \eqref{eq:kno}.  

We begin with some notation.  Define the functions $f_n, f: S^1 \to \mathbb{C}$ by 
\[ f_n(\omega) = \sum_{|j| \leq k_n} a_j w^j \qquad \text{ and } \qquad f(\omega) = \sum_{j \in \mathbb{Z}} a_j \omega^j. \]
We recall that $\|A\|_2$ denotes the Frobenius norm of $A$.  By definition, it follows that
\begin{equation} \label{eq:AF}
	\frac{1}{n} \|A\|_2^2 \leq \sum_{j \in \mathbb{Z}} |a_j|^2, 
\end{equation} 
and so by \eqref{eq:aj}
\begin{equation} \label{eq:normA}
	\|A \|^2 \leq n \sum_{j \in \mathbb{Z}} |a_j|^2 = O(n). 
\end{equation} 
In addition, by Weyl's perturbation theorem (see Theorem~\ref{thm:weyl}) and \eqref{eq:normE}
\begin{equation} \label{eq:AEF}
	\frac{1}{n} \|A + n^{-\alpha - \gamma} E \|_2^2  \leq 2\sum_{j \in \mathbb{Z}} |a_j|^2 + 1 
\end{equation} 
with probability $1 - o(1)$.  Let $A'$ be the matrix from Lemma \ref{lemma:lowrank} so that $C := A+ A'$ is circulant, $\rank(A') \leq 2k_n$, and 
\begin{equation} \label{eq:A'F2}
	\frac{1}{n} \|A'\|_2^2 \leq \sum_{j \in \mathbb{Z}} |a_j|^2.
\end{equation} 
In addition, it follows from Lemma \ref{lemma:lowrank} that $f_n(1), f_n(\omega_n), f_n(\omega_n^2), \ldots, f_n(\omega_n^{n-1})$ are the eigenvalues of $C$, where $\omega_n$ is a primitive $n$-th root of unity as in \eqref{eq:omegan}.  In view of \eqref{eq:AF} and \eqref{eq:A'F2}, we have
\begin{equation} \label{eq:CF}
	\frac{1}{n} \| C \|_2^2 \leq 4  \sum_{j \in \mathbb{Z}} |a_j|^2,
\end{equation} 
and hence 
\begin{equation} \label{eq:normC}
	\|C \|^2 \leq 4 n \sum_{j \in \mathbb{Z}} |a_j|^2 = O(n). 
\end{equation} 

Our goal is to apply the replacement principle, Theorem \ref{thm:TVrepl}, to show that 
\[ \mu_{A + n^{-\alpha - \gamma} E} - \mu_{C} \longrightarrow 0 \]
weakly in probability as $n \to \infty$.  This would complete the proof since Lemma \ref{lemma:distconvergence} implies that $\mu_{C} \to \mu$ weakly as $n \to \infty$.  The Frobenius norm condition of Theorem \ref{thm:TVrepl} follows immediately from \eqref{eq:aj} (see Remark \ref{rem:aj}), \eqref{eq:AEF}, and \eqref{eq:CF}, so it remains to compare the logarithmic determinants of $A + n^{-\alpha - \gamma}E$ and $C$.  

Fix $z \in \mathbb{C}$ with $z \notin f(S^1)$ and such that \eqref{eq:lsvE} holds.  The set of $z \in \mathbb{C}$ which fail to satisfy these properties has Lebesgue measure zero by Lemma \ref{lemma:measurezero} and the assumptions on $E$.  Lemma \ref{lemma:measurezero} and Lemma \ref{lemma:circulant} imply that there exists $\eps' > 0$ so 
\begin{equation} \label{eq:lsvC}
	\sigma_{\min}(C - zI) \geq \eps'
\end{equation} 
for all sufficiently large $n$.  Thus, by Weyl's perturbation theorem (see Theorem~\ref{thm:weyl}) and \eqref{eq:normE}, 
\begin{equation} \label{eq:lsvCE}
	\sigma_{\min}(C + n^{-\alpha-\gamma}E - zI) \geq \frac{\eps'}{2}
\end{equation} 
with probability $1 - o(1)$.  
Applying Theorem \ref{thm:replrank} (using \eqref{eq:normE}, \eqref{eq:normA} and \eqref{eq:normC} to bound the norms and \eqref{eq:lsvE} and \eqref{eq:lsvCE} to bound the smallest singular values), we see that
\[ \left| \mathcal{L}_{A + n^{-\alpha-\gamma}E}(z) - \mathcal{L}_{C + n^{-\alpha-\gamma}E}(z) \right| = O \left( \frac{ \rank(A') \log n}{n} \right) \]
with probability $1 - o(1)$.  In view of \eqref{eq:kno} and the fact that $\rank(A') \leq 2k_n$ we obtain
\[ \left| \mathcal{L}_{A + n^{-\alpha-\gamma}E}(z) - \mathcal{L}_{C + n^{-\alpha-\gamma}E}(z) \right| \longrightarrow 0 \]
in probability.  

Applying Theorem \ref{thm:replacement} (using \eqref{eq:lsvC} and \eqref{eq:lsvCE} to bound the smallest singular values, \eqref{eq:normE} to bound the norm, and taking $\eps$ in Theorem \ref{thm:replacement} to be $\eps'/4$), we obtain
\[ \left| \mathcal{L}_{C + n^{-\gamma-\alpha}E}(z) - \mathcal{L}_{C}(z) \right| = O(n^{-\gamma}) \]
with probability $1 - o(1)$. Combining the bounds above, we conclude that
\[ \left| \mathcal{L}_{A + n^{-\alpha-\gamma}E}(z) - \mathcal{L}_{C}(z) \right| \longrightarrow 0 \]
in probability as $n \to \infty$.  This confirms the last condition in Theorem \ref{thm:TVrepl}, and hence the proof of the theorem is complete.  
\end{proof}

We end this section with a conjecture that is suggested by the proof of Theorem~\ref{thm:pert-toep}.  

\begin{conjecture} \label{conj:A'}
Let $A$ satisfy the conditions of Theorem~\ref{thm:pert-toep}, 
and let $A'$ be the matrix from Lemma~\ref{lemma:lowrank}.  Then, we conjecture that the empirical spectral measure $\mu_{A+n^{-\gamma}A'}$ converges weakly in probability to $\mu$ where $\mu$ is the distribution of $\sum_{|j|\le k } a_j U^j$ where $U$ is a random variable uniformly distributed on $S^1$.  
\end{conjecture}

Numerical evidence supports Conjecture \ref{conj:A'}.  In fact, numerical evidence suggests that the eigenvalues of $A + n^{-\gamma} A'$ better approximate the eigenvalues of $A+n^{-\alpha-\gamma}E$ than $A + A'$ (which we use in the proof Theorem~\ref{thm:pert-toep}).

\section{Pairing between eigenvalues and a rate of convergence} \label{sec:ratec}
Our next result focuses on the pairing between eigenvalues seen in part \eqref{item:pairing} of Theorem \ref{thm:sample}.  As a consequence of this pairing, we also obtain a rate of convergence, in Wasserstein distance, for the empirical spectral distribution of $A + n^{-\gamma }E$ to its limiting distribution.  
Recall that for two probability measures $\mu$ and $\nu$ on $\mathbb{C}$, the $L^1$-Wasserstein distance between $\mu$ and $\nu$ is given by \eqref{eq:defW1}.  

For $z \in \mathbb{C}$ and $r_1, r_2 > 0$, we define the closed rectangular box  
\begin{equation} \label{eq:rectR}
	\mathfrak{R}(z, r_1, r_2) := \{ w \in \mathbb{C} : | \Re(w) - \Re(z)| \leq r_1, | \Im(w) - \Im(z) | \leq r_2 \} 
\end{equation} 
in the complex plane.  
For a finite set $S$, recall that $|S|$ denotes the cardinality of $S$.  

\begin{theorem} \label{thm:rate2}
Let $k \geq 0$ be a fixed integer, let $\{a_j\}_{j \in \mathbb{Z}}$ be a sequence of complex numbers indexed by the integers, and let $A$ be the $n \times n$ Toeplitz matrix with symbol $\{a_j\}_{j \in \mathbb{Z}}$ truncated at $k$.  Define the function $f: S^1 \to \mathbb{C}$ as
\[ f(\omega) = \sum_{|j| \leq k} a_j \omega^j, \]
and let $\omega_n$ be a primitive $n$-th root of unity as in \eqref{eq:omegan}.  
Assume the classical locations $f(1), f(\omega_n), \ldots, f(\omega_n^{n-1})$ do not concentrate in any one rectangle, that is, assume there 
exists $\eps_0 > 0$ and $c_0 \geq 1$ so that, for any $0 < \eps' \leq \eps_0$, 
\begin{equation} \label{eq:numeigdisk2}
	\sup_{z \in \mathbb{C}}| \{ 0 \leq j \leq n-1 : f(\omega_n^j) \in \mathfrak{R}(z, n^{-\eps'}, n^{-c_0 \eps'})\cup \mathfrak{R}(z, n^{-c_0 \eps'}, n^{-\eps'}) \} | = O_{\eps'} (n^{1 - 3\eps'}). 
\end{equation}  
Let $\gamma, \delta > 0$, and let $E$ be an $n \times n$ random matrix so that 
\begin{enumerate}[(i)]
\item \label{item:normbnd2} there exists $M > 0$ and $\alpha \geq 0$ so that
\[ \|E\| \leq M n^{\alpha} \]
with probability $1 - O(n^{-\delta})$.  
\item \label{item:lsvbound2} there exists $\kappa > 0$ so that 
\begin{equation} \label{eq:lsvbnd22}
	\sup_{z \in \mathbb{C}, |z| \leq \beta} \Prob( \sigma_{\min}(A + n^{-\alpha - \gamma} E  - z I) \leq n^{-\kappa}) = O \left( n^{-\delta} \right) 
\end{equation} 
for $\beta := \sqrt{2} \left( \sum_{|j| \leq k} |a_j| + M + 1 \right)$.  
\end{enumerate}
Then, for any $p \geq 1$, there exists $C, \eps > 0$ (depending on $p$, $k$, $f$, $\{a_j\}_{|j| \leq k}$, $\eps_0$, $c_0$, $\gamma$, $\delta$, and the constants from assumptions \eqref{item:normbnd2}, \eqref{item:lsvbound2}, and \eqref{eq:numeigdisk2}) so that
\begin{equation} \label{eq:pairing:main}
	\min_{\sigma: [n] \to [n]} \left( \frac{1}{n} \sum_{j=1}^n \left| \lambda_j(A + n^{-\gamma - \alpha} E) - f \left(\omega_n^{\sigma(j)} \right) \right|^p \right)^{1/p} \leq Cn^{-\eps} 
\end{equation} 
with probability $1 - o(1)$.  
In addition, there exists $C, \eps > 0$ (depending on $k$, $f$, $\{a_j\}_{|j| \leq k}$, $\eps_0$, $c_0$, $\gamma$, $\delta$, and the constants from assumptions \eqref{item:normbnd2}, \eqref{item:lsvbound2}, and \eqref{eq:numeigdisk2}) so that 
\begin{equation} \label{eq:wass:main}
	W_1( \mu_{A + n^{-\gamma - \alpha}E}, \mu) \leq C n^{-\eps} 
\end{equation} 
with probability $1-o(1)$, where $\mu$ is the distribution of $f(U)$ and $U$ is a random variable uniformly distributed on $S^1$. 
\end{theorem}

Condition \eqref{eq:numeigdisk2} is technical and requires that the points $f(\omega_n^j)$ (and hence the curve $f(S^1)$ itself) not concentrate in any small region in the plane.  We have chosen to use rectangles as this matches the geometric construction given in the proof, but other shapes could also be used with appropriate modifications to the proof.  If $A$ is Hermitian then $f$ is real-valued and will fail to satisfy \eqref{eq:numeigdisk2}.  However, the eigenvalues can be rotated by a phase (i.e., by considering $e^{\sqrt{-1} \theta} (A + n^{-\gamma - \alpha}E)$ for an appropriate choice of $\theta \in [0, 2 \pi)$), so that Theorem \ref{thm:rate2} is applicable.  The assumptions on $E$ in Theorem \ref{thm:rate2} are general and apply to a variety of random matrix ensembles.  We give a few examples of Theorem \ref{thm:rate2} below.

\begin{example} \label{exam:iid}
Consider the $n \times n$ matrix $\Jmat$ given in \eqref{eq:defT}.  $A := \Jmat$ is a Toeplitz matrix with $a_{-1} = 1$ and $a_j = 0$ for all $j \neq -1$.  Thus, 
\[ f(\omega) = \frac{1}{\omega} = \overline{\omega}. \]
It is easy to check that $f(\omega_n^j)$, $0 \leq j \leq n-1$ are uniformly spaced on $S^1$, and it follows that condition \eqref{eq:numeigdisk2} is satisfied with $c_0 = 6$ and any $\eps_0 > 0$.  
Let $E$ be an $n \times n$ random matrix whose entries are iid copies of a random variable with mean zero, unit variance, and finite fourth moment.  It follows from Proposition \ref{prop:norm} that $\|E\| \leq n^{\alpha}$ with probability at least $1 - O(n^{1/2-\alpha})$ for any $\alpha > 1/2$.  
The least singular value bound in \eqref{eq:lsvbnd22} follows from \cite[Theorem 2.1]{MR2409368}.  Therefore, for any $\gamma > 0$, Theorem \ref{thm:rate2} can be applied (where we take $\alpha = 1/2 + \gamma/2$ and $\gamma$ in Theorem \ref{thm:rate2} is taken to be $\gamma/2$) to obtain 
\begin{equation} \label{eq:pairing:Jmat}
	\min_{\sigma: [n] \to [n]} \frac{1}{n} \sum_{j=1}^n \left| \lambda_j(\Jmat + n^{-1/2 - \gamma} E) - \omega_n^{\sigma(j)} \right| \leq C n^{-\eps} 
\end{equation} 
and 
\begin{equation} \label{eq:wass:Jmat}
	W_1(\mu_{\Jmat + n^{-1/2 - \gamma} E}, \mu) \leq C n^{-\eps}
\end{equation} 
with probability $1 - o(1)$ for constants $C, \eps > 0$, where $\mu$ is the uniform probability measure on the unit circle $S^1$.  In particular, \eqref{eq:pairing:Jmat} implies that, with probability $1 - o(1)$, there is a one-to-one pairing between the eigenvalues and roots of unity so that the average distance between an eigenvalue and its paired root of unity is $O(n^{-\eps})$.  
\end{example}

\begin{example}
Theorem \ref{thm:rate2} also applies to matrices with heavy-tailed entries.  For example, let $A := \Jmat$ be the same matrix as in Example \ref{exam:iid}. Let $E$ be an $n \times n$ random matrix whose entries are iid copies of a non-constant random variable $\xi$ with $\E|\xi|^\eta < \infty$ for some $\eta > 0$ (and no other moment assumptions).  
A bound on the norm of $E$ follows from the arguments in \cite[Lemma 56]{2010.01261}.  
The least singular value bound in \eqref{eq:lsvbnd22} follows from \cite[Theorem 32]{2010.01261} (alternatively, see \cite[Lemma A.1]{MR2908617}).  
Therefore, we conclude from Theorem \ref{thm:rate2} that there exists $\gamma_0 > 0$ (depending only on $\xi$) so that for any $\gamma > \gamma_0$, \eqref{eq:pairing:Jmat} and \eqref{eq:wass:Jmat} hold with probability $1 - o(1)$ for some constants $C, \eps > 0$, even though $E$ contains heavy-tailed entries.  
\end{example}

\begin{example} 
Let $A$ be the Toeplitz matrix in \eqref{eq:ToeplitzA} with $a_{-1} = 4$, $a_1 = 1$, and $a_j =0$ for all other $j \in \mathbb{Z}$.  Then
\[ f(\omega) = \frac{4}{\omega} + \omega \]
so that $f(S^1)$ is an ellipse.  It can be checked that $f$ satisfies \eqref{eq:numeigdisk2} for $\eps_0$ sufficiently small and $c_0 = 6$.  
Let $E$ be an $n \times n$ random matrix uniformly distributed on the unitary group $\mathcal{U}(n)$.  Then $\|E\| = 1$ (with probability $1$), and for any $\gamma > 0$, the least singular value bound in \eqref{eq:lsvbnd22} follows from \cite[Theorem 1.1]{MR3164983}.  Therefore, Theorem \ref{thm:rate2} implies that there exists $C, \eps > 0$ so that
\[ \min_{\sigma: [n] \to [n]} \frac{1}{n} \sum_{j=1}^n \left| \lambda_j(A + n^{- \gamma} E) - f\left( \omega_n^{\sigma(j)}\right) \right| \leq C n^{-\eps} \]
and
\[ W_1(\mu_{A + n^{ - \gamma} E}, \mu) \leq C n^{-\eps} \]
with probability $1 - o(1)$, where $\mu$ is the distribution of $f(U)$ and $U$ is a random variable uniformly distributed on $S^1$.  
\end{example}

We have stated Theorem \ref{thm:rate2} for the case when $k$ is fixed.  However, our method allows $k$ to slowly grow with the dimension $n$.  Since the growth rate of $k$ is technical to state (and depends on many of the other parameters in Theorem \ref{thm:rate2}), we have chosen to state the theorem only for $k$ fixed.  

We conjecture that, under certain conditions such as when the entries of $E$ are iid standard normal random variables, the optimal rate of convergence for the Wasserstein distance in Theorem \ref{thm:rate2} is $O(\log n/n)$.  Our method only allows us to conclude a bound of the from $O(n^{-\eps})$ for small values of $\eps > 0$, and it is unclear how the optimal rate of convergence may depend on $\gamma$ and the other parameters in Theorem~\ref{thm:rate2}.  

Theorem \ref{thm:rate2} is related to the results of Sj\"{o}strand and Vogel \cite{MR4200678}, which show precise asymptotic bounds for the number of eigenvalues in smooth domains for a banded Toeplitz matrix perturbed by a random matrix with iid standard complex Gaussian entries.  Our results differ from those in \cite{MR4200678} in a number of ways.  Our results apply to much more general choices for the random matrix model, even including matrices with dependent entries.  While Theorem \ref{thm:rate2} compares the eigenvalues to the classical locations, the results in \cite{MR4200678} focus on comparisons with the limiting spectral distribution.  This distinction can make a difference for certain applications; for example, the rate of convergence bound in \eqref{eq:wass:main} follows trivially from \eqref{eq:pairing:main} (see the proof below).  However, we do not know a trivial way to establish the same rate of convergence from the results in \cite{MR4200678}.  Theorem \ref{thm:rate2} also holds for different values of $\alpha$ and $\gamma$ compared to the results in \cite{MR4200678}.  Lastly we mention that our method of proof uses a comparison method, which is quite different than the Grushin reduction method used in \cite{MR4200678}.   

Theorem \ref{thm:rate2} is also related to the results of Basak and Zeitouni \cite{MR4168388} concerning the outliers of $A + n^{-\alpha-\gamma}E$.  Since outlying eigenvalues are a positive, $n$-independent distance away from the classical locations, the more outliers there are, the worse the pairing in \eqref{eq:pairing:main} will be.   In fact, Theorem \ref{thm:rate2} can be used to give an upper bound on the number of outlier eigenvalues.  However, our results in Section \ref{sec:locallaw} tend to provide better bounds for the outliers in many cases.  

Theorem \ref{thm:rate2} is similar to other rate of convergence results in the random matrix theory literature, see for example \cite{MR4254801,MR2013983,MR4302281,MR3820329,MR3837270,MR1217559,MR1217560,BSbook,MR2171668,MR2105745,MR2766657} and references therein.  We draw special attention to the works \cite{MR3325952,2111.03595,1912.08856}, which directly influenced Theorem \ref{thm:rate2}.

\subsection{Proof of Theorem \ref{thm:rate2}}

The rest of this section is devoted to the proof of Theorem \ref{thm:rate2}.  Assume the setup and notation of Theorem \ref{thm:rate2}.  We allow the implicit constants in our asymptotic notation (such as $O(\cdot)$ and $\ll$) to depend on the parameters and constants of Theorem \ref{thm:rate2} (such as $p$, $k$, $f$, $\{a_j\}_{|j| \leq k}$, $\eps_0$, $c_0$, $\gamma$, $\delta$, and the constants from assumptions \eqref{item:normbnd2}, \eqref{item:lsvbound2}, and \eqref{eq:numeigdisk2}) without denoting this dependence.  

Define the event
\[ \Omega := \{ \|E\| \leq M n^{\alpha} \}, \]
which by supposition holds with probability $1 - O(n^{-\delta})$.  On $\Omega$, for $n$ sufficiently large, 
\begin{equation} \label{eq:nbAE}
	\|A + n^{-\alpha - \gamma} E \| < \|A\| + M \leq \sum_{|j| \leq k} |a_j| + M,  
\end{equation} 
where the bound for $\|A\|$ follows from Proposition \ref{prop:Anorm} in Appendix \ref{sec:multi}.  

By Lemma \ref{lemma:lowrank}, there exists an $n \times n$ deterministic matrix $A'$ with rank at most $2k$ so that $A + A'$ is circulant with eigenvalues given by $f(1), f(\omega_n), f(\omega_n^2), \ldots, f(\omega_n^{n-1})$.  
From Lemma \ref{lemma:circulant}, we see that the singular values of $A + A'$ are then $|f(1)|, |f(\omega_n)|, |f(\omega_n^2)|, \ldots, |f(\omega_n^{n-1})|$.  This implies that 
\begin{equation} \label{eq:nbAA'}
	\|A + A'\| \leq \sup_{\omega \in S^1} |f(\omega)| \leq \sum_{|j| \leq k} |a_j|, 
\end{equation} 
and hence 
\begin{equation} \label{eq:nbAA'E}
	\|A + A' + n^{-\alpha - \gamma} E \| < \sum_{|j| \leq k} |a_j| + M 
\end{equation} 
on the event $\Omega$.  

Define the box $R$ in the complex plane by  
\[ R = \left\{ z \in \mathbb{C} : -\sum_{|j| \leq k} |a_j| - M - 1/2 \leq \Re(z), \Im(z) < \sum_{|j| \leq k} |a_j| + M + 1/2 \right\}. \] 
Notice that 
\begin{equation} \label{eq:BRbnd}
	B \left (0, \sum_{|j| \leq k} |a_j| + M + 1/4 \right) \subset R \subset B \left( 0, \beta \right). 
\end{equation} 
A set $S$ in the complex plane of the form
\[ S = \{ z \in \mathbb{C} : a \leq \Re(z) < b, c \leq \Im(z) < d \} \]
with $b - a = d - c > 0$ is called a \emph{square}, and we say $b-a$ is the \emph{side length} and $(a+b)/2 + \sqrt{-1} (c+d)/2$ is the \emph{center} of $S$.  For example, $R$ is a square with center $0$ and side length $2\left( \sum_{|j| \leq k} |a_j|\right) + 2M + 1$.   Let $R_1, \ldots, R_L$ be a partition of $R$ into disjoint squares, all with equal side length of $\Theta(n^{-a})$ for some $a > 0$ to be chosen later.  In particular, we note that 
\[ \bigcup_{i=1}^L R_i = R. \]
A volume argument shows that 
\begin{equation} \label{eq:Lbnd}
	L = O(n^{2a}). 
\end{equation} 

For $1 \leq i \leq L$, we define $X_i$ to be the number of eigenvalues of $A + n^{-\alpha - \gamma} E$ in $R_i$, and set $Y_i$ to be the number of eigenvalues of $A + A'$ in $R_i$.  The following lemma represents the key technical result we need for the proof.  
\begin{lemma} \label{lemma:compareXiYi}
Under the assumptions of Theorem \ref{thm:rate2}, for any sufficiently small $a, \eps > 0$, there exists a constant $C > 0$ so that 
\begin{equation} \label{eq:XiYibound}
	\sup_{1 \leq i \leq L } |X_i - Y_i| \leq C \left( n^{1 - \eps + 6c_0a} + n^{1 - 3a} \right) 
\end{equation} 
with probability $1 - o(1)$.  (Recall that the squares $R_i$ have side length $\Theta(n^{-a})$ and $c_0$ is the constant from \eqref{eq:numeigdisk2}.) Here the sufficient smallness of $a$ and $\eps$ depends on $k$, $\{a_j\}_{|j| \leq k}$, $f$, $\eps_0$, $c_0$, $\gamma$, $\delta$, and the constants from assumptions \eqref{item:normbnd2}, \eqref{item:lsvbound2}, and \eqref{eq:numeigdisk2} in Theorem \ref{thm:rate2}; the constant $C$ depends on $a$ and $\eps$ as well as these other parameters.  
\end{lemma}

Before proving Lemma \ref{lemma:compareXiYi}, we first complete the proof of Theorem \ref{thm:rate2}.  
Recall that $\lambda_1(B), \ldots, \lambda_n(B)$ denote the eigenvalues of the $n \times n$ matrix $B$, and as noted in Section \ref{sec:illustrating_example}, we order the eigenvalues by lexicographic order.  

Observe that, for any permutation $\sigma:[n] \to [n]$, we have
\[ W_1(\mu_{A + n^{-\alpha-\gamma}E}, \mu_{A + A'}) \leq \frac{1}{n} \sum_{j=1}^n \left| \lambda_j(A + n^{-\alpha-\gamma}E) - \lambda_{\sigma(j)}(A + A') \right|. \]
Thus, \eqref{eq:wass:main} follows by applying \eqref{eq:pairing:main} (with $p=1$), Lemma \ref{lemma:wassbnd}, and the triangle inequality for the Wasserstein metric. 

It remains to prove \eqref{eq:pairing:main}.  Fix $p \geq 1$.  We will establish \eqref{eq:pairing:main} by constructing a permutation $\sigma: [n] \to [n]$ so that 
\[ \frac{1}{n} \sum_{j=1}^n \left| \lambda_j(A + n^{-\gamma - \alpha} E) - \lambda_{\sigma(j)}(A + A') \right|^p = O(n^{-\eps}) \]
with probability $1-o(1)$.  
To do so, we work on the event 
\[ \mathcal{F} := \left\{ \sup_{1 \leq i \leq L } |X_i - Y_i| \leq C \left( n^{1 - \eps + 6c_0a} + n^{1 - 3a} \right)  \right\} \bigcap \Omega. \]
By Lemma \ref{lemma:compareXiYi} and assumption \eqref{item:normbnd2} from Theorem \ref{thm:rate2}, $\mathcal{F}$ holds with probability $1-o(1)$ for $\eps, a > 0$ sufficiently small (in particular we will take $8c_0a < \eps$) and $C > 0$ sufficiently large.  

Notice that the permutation $\sigma$ defines a paring between the eigenvalues of $A + n^{-\alpha-\gamma}E$ and the eigenvalues of $A + A'$.  Thus, in order to define $\sigma$, we may equivalently construct such a pairing, i.e., we say $\lambda_i(A + n^{-\alpha-\gamma}E)$ and $\lambda_j(A + A')$ are paired if and only if $\sigma(i) = j$.  

We now construct $\sigma$ on the event $\mathcal{F}$.  Indeed, $\sigma$ itself will be random, so for each outcome in $\mathcal{F}$, we construct a possibly different permutation $\sigma$. Fix an outcome in $\mathcal{F}$, and observe from \eqref{eq:nbAE} and \eqref{eq:nbAA'} that all the eigenvalues of both matrices are contained in $R$ (by \eqref{eq:BRbnd} and the fact that $\mathcal{F} \subset \Omega$).  This means all the eigenvalues are contained in the squares $R_1, \ldots, R_L$.  As we construct the permutation $\sigma$, we will say the index $i$ (or the pairing of $i$) is \emph{good} if both $\lambda_i(A + n^{-\alpha-\gamma}E)$ and $\lambda_{\sigma(i)}(A + A')$ are in the same square $R_k$, $1 \leq k \leq L$; otherwise we call the index $i$ (or pairing of $i$) \emph{bad}.  To start, arbitrarily choose $\min\{X_1, Y_1\}$ eigenvalues of $A + n^{-\alpha-\gamma}E$ in $R_1$ and pair them arbitrarily with $\min\{X_1, Y_1\}$ eigenvalues of $A + A'$ in $R_1$.  After this first step, there may remain some eigenvalues in $R_1$ that are unpaired; we will leave them unpaired until the last step.  Next, repeat the procedure for $R_2$:  arbitrarily choose $\min\{X_2, Y_2\}$ eigenvalues of $A + n^{-\alpha-\gamma}E$ in $R_2$ and pair them arbitrarily with $\min\{X_2, Y_2\}$ eigenvalues of $A + A'$ in $R_2$.  Continue in this way, choosing $\min\{X_k, Y_k\}$ eigenvalues of $A + n^{-\alpha-\gamma}E$ in $R_k$ and pairing them arbitrarily with $\min\{X_k, Y_k\}$ eigenvalues of $A + A'$ in $R_k$ for all $1 \leq k \leq L$.  So far, all the pairings we have made are good pairings.  To complete the construction of $\sigma$, now pair all the remaining unpaired eigenvalues of $A + n^{-\alpha-\gamma}E$ arbitrarily with the remaining unpaired eigenvalues of $A + A'$; all the pairings in this last step are bad pairings.  This procedure constructs the random permutation $\sigma$ on the event $\mathcal{F}$; we will only work with $\sigma$ on $\mathcal{F}$, but $\sigma$ can easily be extended to the entire probability space by taking $\sigma$ to be the identity permutation on $\mathcal{F}^c$.  

We then have
\begin{align*}
	\frac{1}{n} \sum_{i=1}^n \left| \lambda_i(A + n^{-\alpha-\gamma}E) - \lambda_\sigma(i)(A + A') \right|^p &\leq \frac{1}{n} \sum_{i \text{ is good}} \left| \lambda_i(A + n^{-\alpha-\gamma}E) - \lambda_{\sigma(i)}(A + A') \right|^p \\
	 &\qquad + \frac{1}{n} \sum_{i \text{ is bad}} \left| \lambda_i(A + n^{-\alpha-\gamma}E) - \lambda_{\sigma(i)}(A + A') \right|^p.  
\end{align*}
On the one hand, if $i$ is good, then both $\lambda_i(A + n^{-\alpha-\gamma}E)$ and $\lambda_{\sigma(i)}(A + A')$ lie in the same square, so the distance between them is at most the diameter of the square, which by construction is $\Theta(n^{-a})$.   On the other hand, if $i$ is bad, then both $\lambda_i(A + n^{-\alpha-\gamma}E)$ and $\lambda_{\sigma(i)}(A + A')$ lie in $R$ which has diameter less than $2\beta = O(1)$ (see \eqref{eq:BRbnd}).  However, we note that there cannot be too many bad indices.  Indeed, after we make the good pairings, each square $R_k$ has at most $C(n^{1 - \eps + 6c_0a} + n^{1 - 3a})$ unpaired eigenvalues remaining on the event $\mathcal{F}$ for some constant $C > 0$.  In view of \eqref{eq:Lbnd} then, there are at most $O(n^{1 - \eps + 8c_0a} + n^{1 - a})$ total bad indices.  Therefore, using that there are at most $n$ good indices, we conclude that
\begin{align*}
	\frac{1}{n} \sum_{i=1}^n \left| \lambda_i(A + n^{-\alpha-\gamma}E) - \lambda_\sigma(i)(A + A') \right|^p &\ll \frac{1}{n} \sum_{i \text{ is good}} n^{-ap} + \frac{1}{n} \sum_{i \text{ is bad}} (2\beta)^p \\
	&\ll n^{-ap} +   n^{- \eps + 8c_0a}
\end{align*}
on the event $\mathcal{F}$.  Choosing  $a, \eps$ sufficiently small with $8c_0a < \eps$ completes the proof of Theorem \ref{thm:rate2}, and it only remains to prove Lemma \ref{lemma:compareXiYi}.  

\subsection{Proof of Lemma \ref{lemma:compareXiYi}}
We conclude this section with the proof of Lemma \ref{lemma:compareXiYi}.  

Let $a, \eps > 0$ be sufficiently small to be chosen later, and recall that the squares $R_i$ all have the same have side length $s = \Theta(n^{-a})$.  
For $1 \leq i \leq L$, we let $\tilde{R}_i$ and $\breve{R}_i$ be the squares with the same center as $R_i$ but with side lengths of $s + n^{-c_0 a}$ and $s - n^{-c_0 a}$, respectively, where $c_0$ is the constant from \eqref{eq:numeigdisk2}.  For $1 \leq i \leq L$, define smooth functions $\tilde{\varphi}_i, \breve{\varphi}_i: \mathbb{C} \to [0,1]$ so that $\tilde{\varphi}_i$ is supported on $\tilde{R}_i$ with $\tilde{\varphi}_i(z) = 1$ for $z \in R_i$ and $\breve{\varphi}_i$ is supported on $R_i$ with $\breve{\varphi}_i(z) = 1$ for $z \in \breve{R}_i$.  We can construct $\tilde{\varphi}_i$ and $\breve{\varphi}_i$ using products of bump functions in such a way that
\begin{equation} \label{eq:lapbnd}
	\sup_{1 \leq i \leq L} \left( \| \Delta \tilde{\varphi}_i \|_{\infty} + \| \Delta \breve{\varphi}_i \|_{\infty} \right) = O(n^{6c_0a}) 
\end{equation}
by construction of the squares $R_i, \tilde{R}_i$, and $\breve{R}_i$.  

By construction of these functions we find
\begin{equation} \label{eq:Xibnd}
	\sum_{j=1}^n \breve{\varphi}_i( \lambda_j(A + n^{-\alpha - \gamma}E)) \leq X_i \leq \sum_{j=1}^n \tilde{\varphi}_i(\lambda_j(A + n^{-\alpha-\gamma}E))
\end{equation} 
and
\begin{equation} \label{eq:firstYibnd}
	\sum_{j=1}^n \breve{\varphi}_i( \lambda_j(A + A')) \leq Y_i \leq \sum_{j=1}^n \tilde{\varphi}_i(\lambda_j(A + A'))
\end{equation} 
for $1 \leq i \leq L$.  
By taking $a \leq \eps_0$, we apply \eqref{eq:numeigdisk2} to find that 
\[ \sup_{1 \leq i \leq L} \left| \sum_{j=1}^n \left( \tilde{\varphi}_i(\lambda_j(A + A'))  - \breve{\varphi}_i(\lambda_j(A + A')) \right) \right| = O(n^{1 - 3a}). \]
(No union bound is required here since the eigenvalues $f(1), f(\omega_n), f(\omega_n^2), \ldots, f(\omega_n^{n-1})$ of $A + A'$ are deterministic.)
Thus, we can write \eqref{eq:firstYibnd} as
\begin{equation} \label{eq:Yibnd}
\sum_{j=1}^n	\tilde{\varphi}_i(\lambda_j(A + A')) - O(n^{1 - 3a}) \leq Y_i \leq \sum_{j=1}^n \breve{\varphi}_i( \lambda_j(A + A')) + O(n^{1 - 3a}) 
\end{equation} 
uniformly for $1 \leq i \leq L$.  Subtracting \eqref{eq:Yibnd} from \eqref{eq:Xibnd} yields
\begin{align} \label{eq:XiYidiff}
	\sum_{j=1}^n \breve{\varphi}_i( \lambda_j(A + n^{-\alpha - \gamma}E)) &- \sum_{j=1}^n \breve{\varphi}_i( \lambda_j(A + A')) - O(n^{1 - 3a})   \\ 
	& \leq X_i  - Y_i  \leq \sum_{j=1}^n \tilde{\varphi}_i(\lambda_j(A + n^{-\alpha-\gamma}E)) - \sum_{j=1}^n \tilde{\varphi}_i(\lambda_j(A + A')) + O(n^{1 - 3a}) \nonumber
\end{align} 
uniformly for $1 \leq i \leq L$.  
The goal is to compare $X_i$ and $Y_i$ by comparing the differences of the sums above using Theorem \ref{thm:non}.   

To this end, take $b := 100c_0a$.  
For $1 \leq i \leq L$, let $\tilde Z_i$ and $\breve Z_i$ be random variables, independent of all other sources of randomness, uniform on the supports of $\tilde \varphi_i$ and $\breve \varphi_i$, respectively.  Let $Q$ be all the points in the complex plane at distance at most $n^{-b}$ from $f(S^1)$.  For $n$ sufficiently large, $Q \subset B(0, \beta)$.  In addition, since $f(S^1)$ has arc length $O(1)$, it follows that $Q$ has Lebesgue measure $O(n^{-b})$.  

Define the events
\[ \mathcal{E}_i := \{ \sigma_{\min}(A + n^{-\alpha - \gamma} E - \tilde Z_i I) > n^{-\kappa} \} \cap \{ \tilde Z_i \not \in Q \} \cap \{ \sigma_{\min}(A + n^{-\alpha - \gamma} E - \breve Z_i I) > n^{-\kappa} \} \cap \{ \breve Z_i \not \in Q \} \]
for $1 \leq i \leq L$.  It follows from \eqref{eq:lsvbnd22} (by conditioning on $\tilde Z_i, \breve Z_i$) and the bound on the Lebesgue measure of $Q$ that 
\[ \sup_{1 \leq i \leq L} \Prob( \mathcal{E}_i^c) = O(n^{-98c_0a}) \]
for $a$ sufficiently small (in particular this requires $98c_0a < \delta$).  Hence, by the union bound (see \eqref{eq:Lbnd}), the event 
\[ \mathcal{E} := \Omega \bigcap \left( \bigcap_{i=1}^L \mathcal{E}_i \right) \]
holds with probability $1 - O(n^{-96c_0a})$.  

We now work on the event $\mathcal{E}$.  Indeed, since $\mathcal{E} \subset \Omega$, it follows that the norm bounds in \eqref{eq:nbAE}, \eqref{eq:nbAA'}, and \eqref{eq:nbAA'E} hold on the event $\mathcal{E}$.  Moreover, note that for $z \not\in Q$, $\sigma_{\min}(A + A' - zI) \geq n^{-b}$ as the singular values of $A + A' - zI$ are the values $|f(1) - z|, |f(\omega_n) -z|, |f(\omega_n^2) -z|, \ldots, |f(\omega_n^{n-1}) -z|$ by Lemma \ref{lemma:circulant}.  By Weyl's perturbation theorem (Theorem \ref{thm:weyl}) taking $a$ sufficiently small so that $b < \gamma$, we see that 
\[ \sigma_{\min}(A + A' + n^{-\alpha-\gamma}E - \tilde{Z}_i I) \geq \frac{n^{-b}}{2}, \qquad\qquad 1 \leq i \leq L \]
for $n$ sufficiently large on the event $\mathcal{E}$.  

Therefore, by Theorem \ref{thm:replrank},  
\begin{equation} \label{eq:Lbnd1}
	\sup_{1 \leq i \leq n} \left| \mathcal{L}_{A + n^{-\alpha-\gamma}E}(\tilde Z_i) - \mathcal{L}_{A + A' + n^{-\alpha-\gamma}E}(\tilde Z_i) \right| = O \left( \frac{\log n}{n} \right) 
\end{equation} 
on the event $\mathcal{E}$ since $A'$ has rank at most $2k = O(1)$.  

We next apply Theorem \ref{thm:replacement}.  On the event $\mathcal{E}$, $\nu_{A + A' - \tilde Z_iI}([0, n^{-2b}]) = 0$ for all $1 \leq i \leq L$ (from the discussion above), and hence Theorem \ref{thm:replacement} implies that 
\[ \sup_{1 \leq i \leq L} \left| \mathcal{L}_{A + A' + n^{-\alpha-\gamma}E}(\tilde Z_i) - \mathcal{L}_{A + A'}(\tilde Z_i) \right| \ll n^{2b-\gamma}  \]
on the event $\mathcal{E}$.  Taking $a$ (and hence $b$) sufficiently small yields 
\begin{equation} \label{eq:Lbnd2}
	\sup_{1 \leq i \leq L} \left| \mathcal{L}_{A + A' + n^{-\alpha-\gamma}E}(\tilde Z_i) - \mathcal{L}_{A + A'}(\tilde Z_i) \right| \ll n^{-\eps} 
\end{equation}
for $\eps < \gamma/2$.  
Combining \eqref{eq:Lbnd1} and \eqref{eq:Lbnd2} shows that
\begin{equation} \label{eq:supLrandbnd}
	\sup_{1 \leq i \leq L} \left| \mathcal{L}_{A + n^{-\alpha-\gamma}E}(\tilde Z_i) - \mathcal{L}_{A+A'}(\tilde Z_i) \right| = O(n^{-\eps}) 
\end{equation}
on the event $\mathcal{E}$ for any $\eps < \min\{ \gamma/2, 1/2\}$.  By repeating the argument above with $\breve Z_i$ taking the place of $\tilde Z_i$, we similarly find that 
\begin{equation} \label{eq:supLrandbnd2}
	\sup_{1 \leq i \leq L} \left| \mathcal{L}_{A + n^{-\alpha-\gamma}E}(\breve Z_i) - \mathcal{L}_{A+A'}(\breve Z_i) \right| = O(n^{-\eps}) 
\end{equation}
on the event $\mathcal{E}$.  

Using \eqref{eq:supLrandbnd} and \eqref{eq:supLrandbnd2}, we now apply Theorem \ref{thm:non} (we continue to use the norm bounds in \eqref{eq:nbAE}, \eqref{eq:nbAA'}, and \eqref{eq:nbAA'E} which all hold with probability $1 - O(n^{-\delta})$).  Indeed, applying Theorem \ref{thm:non} (using the description of $\Csixone$ given in \eqref{e:ThmConst}) with $m = \lfloor n^{60c_0a} \rfloor$, \eqref{eq:lapbnd}, and the union bound, it follows that 
\begin{equation} \label{eq:finaltilde}
	\sup_{1 \leq i \leq L } \left| \sum_{j=1}^n \tilde \varphi_i(\lambda_j(A + n^{-\alpha-\gamma}E) - \sum_{j=1}^n \tilde \varphi_i ( \lambda_j(A + A')) \right| \ll n^{1 - \eps + 6c_0a} + n^{1 - 6c_0a} 
\end{equation}
with probability $1 - O(n^{-34c_0a})$.  Repeating the argument for $\breve \varphi_i$, $1 \leq i \leq L$ (using \eqref{eq:supLrandbnd2}), we similarly obtain
\begin{equation} \label{eq:finalbreve}
	\sup_{1 \leq i \leq L } \left| \sum_{j=1}^n \breve \varphi_i(\lambda_j(A + n^{-\alpha-\gamma}E) - \sum_{j=1}^n \breve \varphi_i ( \lambda_j(A + A')) \right| \ll n^{1 - \eps + 6c_0a} + n^{1 - 6c_0a} 
\end{equation}
with probability $1 - O(n^{-34c_0a})$. 

Combining \eqref{eq:finaltilde} and \eqref{eq:finalbreve} with \eqref{eq:XiYidiff}, we conclude that
\[ \sup_{1 \leq i \leq L} |X_i - Y_i| \ll n^{1-\eps + 6c_0a} + n^{1 - 6c_0a} + n^{1-3a} \ll n^{1-\eps + 6c_0a} + n^{1-3a} \]
with probability $1 - O(n^{-34c_0a})$.  This completes the proof of the lemma.

\section{Deterministic, non-asymptotic results} \label{sec:deterministic} \label{sec:app}

Many of our results also hold in the case when $E$ is deterministic as the following propositions show.  Below, we derive non-asymptotic bounds on the locations of the eigenvalues for deterministic perturbations of a Jordan block matrix.  There are many asymptotic results in the literature for perturbed Jordan canonical form matrices, for example \cite{MR3134007,Wood2016,FPZeitouni_regularization_2015,basak_regularization_2019,basak_spectrum_2020}, and we consider a more general matrix in Jordan canonical form in Section~\ref{sec:examples}.  Davies and Hager \cite{MR2490477} prove a non-asymptotic result bounding the magnitudes of the eigenvalues of the perturbed Jordan block matrix $\Jmat$, and an application of the Poisson-Jensen formula results in a factor of $\log n$ in the final bound.  Propositions~\ref{p:smalln} and \ref{prop:leaving-the-disk} below avoid the extra factor of $\log n$.

\begin{proposition}\label{p:smalln}
Let $n\ge 7$ be an integer and let $\gamma\ge 5$ and $1/20 > \epsilon>0$ be real numbers. Let $\Jmat$ be an $n$ by $n$ Jordan block matrix with eigenvalue zero (see \eqref{eq:defT}), and let $\pone$ be any matrix in which each entry is $n^{-\gamma}$ or $-n^{-\gamma}$.
If $n \le \frac12 \gamma \log \gamma$, then there are at most $O_\epsilon(1)$ eigenvalues of $\Jmat+\pone$ that fall outside the disk with radius $1/4-\epsilon$ centered at the origin, and if  $n \ge 2 \gamma \log \gamma$, then there are at most $O_\epsilon(1)$ eigenvalues that fall inside the disk with radius $1/4-\epsilon$.
\end{proposition}

Proposition~\ref{p:smalln} guarantees the same eigenvalue behavior for $\Jmat+\pone$, regardless of the choice of signs for the entries of $\pone$.  In this way, Proposition \ref{p:smalln} can be viewed as describing a model in which regardless of the way an adversary chooses signs for the matrix $\pone$, the eigenvalues always have the same non-asymptotic behavior.  The restriction to the disk of radius $1/4$ in Proposition \ref{p:smalln} is for simplicity and can likely be extended to any disk of radius less than one using similar methods.   In what follows, we prove Proposition~\ref{prop:leaving-the-disk}, a more general version of Proposition~\ref{p:smalln} that provides explicit, non-asymptotic bounds.
More generally, we conjecture that the eigenvalues of $\Jmat+\pone$, regardless of the choice of signs for the entries of $\pone$, will behave in a similar fashion as the eigenvalues of $\Jmat + n^{-\gamma} E$ when $E$ is a random matrix with iid standard normal entries.

\begin{proposition}\label{prop:leaving-the-disk}
Let $n$ be a positive integer and let $\Jmat$ be the $n$ by $n$ matrix given in \eqref{eq:defT} with all entries on the super diagonal equal to 1 and all other entries equal to zero.  Let $\gamma$ be a positive real number, let $\pone$ be an arbitrary deterministic $n$ by $n$ matrix where each entry is $\pm n^{-\gamma}$, and let $\ptwo$ be a matrix with entry $(n,1)$ equal to $n^{-\gamma}$ and all other entries zero. 
\begin{enumerate}
\item[(i)] If $n \le \frac{\gamma \log \gamma - \log 3}{\log 5}$ and $n \ge 7$, then for any smooth function $\varphi: \mathbb{C} \to \mathbb{C}$ with support of $\Delta \varphi$ contained in $\{z \in \mathbb{C} : 1/5 \le |z| < 1/4\}$, we have that 

$$\abs{ \sum_{i=1}^n \varphi(\lambda_i(\Jmat+\pone)) - \varphi(\lambda_i(\Jmat+\ptwo)) }  \le  3 \|\Delta \varphi\|_\infty. $$

\item[(ii)] If $n\ge 2 \gamma\log \gamma$ and $\gamma \ge 5$, then for any smooth function $\varphi: \mathbb{C} \to \mathbb{C}$ with support  of $\Delta \varphi$ contained in $\{z \in \mathbb{C} : |z| < 1/4\}$, we have that 
$$\abs{ \sum_{i=1}^n \varphi(\lambda_i(\Jmat+\pone)) }  \le 3 \|\Delta \varphi\|_\infty.$$
\end{enumerate}
Above, $\|\Delta \varphi\|_\infty$ denotes the $L^\infty$-norm of $\Delta \varphi$.  
\end{proposition}

Proposition~\ref{prop:leaving-the-disk} quantifies the fact that the eigenvalues of $\Jmat+\pone$ are very close to the origin when $n$ is small, and then rapidly shift to being close to roots of unity as $n$ increases.  For example, if $\epsilon >0$ is small and if $\varphi$ is chosen to be a smooth function that is one when $|z| \le 1/4 -\epsilon$ and zero when $|z|\ge 1/4$, then part (i) shows that when $n \le \frac{\gamma \log \gamma -\log 3}{\log 5}$ and $n \ge 7$, there are at most a constant number of eigenvalues that fall \emph{outside} the disk of radius $1/4-\epsilon$.  On the other hand, for large $n$, part (ii) shows that when $n \ge 2\gamma \log \gamma$ and $\gamma \ge 5$, there are at most a constant number of eigenvalues that fall \emph{inside} the disk of radius $1/4-\epsilon$ (note that $\varphi(\lambda_i(\Jmat+\ptwo))=0$ for all $i$ in this case).  Furthermore, these results hold for small values $n$ and $\gamma$, without any appeal to asymptotic behavior.   One can interpret Proposition~\ref{prop:leaving-the-disk} as proving a version of Conjecture~\ref{conj:A'} for the special case of the matrix $\Jmat$.  An example of Proposition~\ref{prop:leaving-the-disk} is shown in Figure~\ref{fig:prop-small-n}.

\begin{figure}
\includegraphics[scale=0.8]{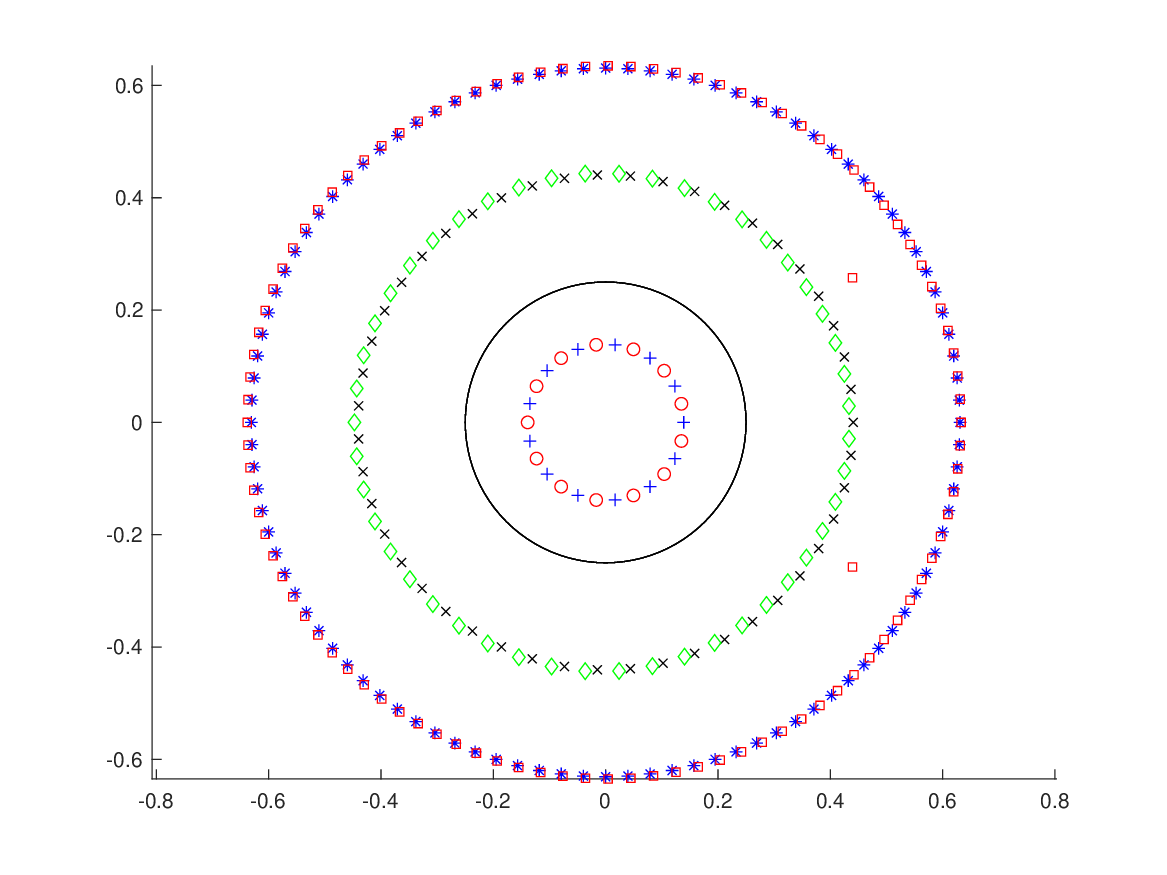}
\caption{The plot above shows the eigenvalues of $\Jmat+\ptwo$ and $\Jmat+ \pone$, where $\Jmat$ is as defined in \eqref{eq:defT}, the matrix $\ptwo$ has $(n,1)$ entry equal to $n^{-\gamma}$ and all other entries equal to zero, and $\pone$ is a matrix with iid $\pm n^{-\gamma}$ entries, where $\gamma =10$.  The eigenvalues for  $\Jmat+\ptwo$ and $\Jmat+ \pone$ are plotted for $n=13$ (blue plus signs $\color{blue}+\color{black}$ and red circles $\color{red}\circ\color{black}$, respectively), $n=47$ (black times signs $\times$ and green diamonds $\color{green}\diamond\color{black}$, respectively), and $n=100$ (blue asterisks $\color{blue}\ast\color{black}$ and red squares $\color{red}\square\color{black}$, respectively).  A circle with radius $1/4$ centered at the origin is also drawn in black.  For small $\epsilon >0$, Proposition~\ref{prop:leaving-the-disk}(ii) shows that there are only constantly many eigenvalues inside the $1/4-\epsilon$ disk when $n=47$ or $100$, and Proposition~\ref{prop:leaving-the-disk}(i) shows that there are only constantly many eigenvalues \emph{outside} the $1/4-\epsilon$ disk when $n=13$. Note that the spectrum appears to be extremely regular for $n=13$ and $n=47$, with all eigenvalues having nearly the same magnitude; however, in this instance when $n=100$, one can see that there are a few eigenvalues with different magnitude (see the red squares $\color{red}\square\color{black}$ at approximately $(0.42,\pm 0.24)$), which is typical behavior.} 
\label{fig:prop-small-n}
\end{figure}

To prove Proposition~\ref{prop:leaving-the-disk}, we will combine Theorem~\ref{thm:non} with Proposition~\ref{prop:small-cancel} (a variant Theorem~\ref{thm:replacement}) by comparing the perturbation $\pone$ with a perturbation by an $n$ by $n$ matrix $\ptwo$ that has its $(n,1)$ entry equal to $n^{-\gamma}$ and all other entries equal to zero.  
In order to use Theorem~\ref{thm:non}, we will need to understand the smallest singular values of $\Jmat+\pone-zI$ and $\Jmat+\ptwo-zI$ for any $z$ satisfying $|z|\le 1/4$, and we will use the following two lemmas to prove sharp lower bounds on the smallest singular value.

\begin{lemma}\label{lem:smallsing}
Let $n$ be a positive integer, let $\gamma \ge 5$, let $|z| \leq 1/4$, let $n \ge 2 \gamma \log \gamma$, and let $\Jmat$ and $\pone$ be defined as in Proposition~\ref{prop:leaving-the-disk}. Then, the smallest singular value of $\Jmat+\pone-zI$ is at least $0.15 n^{-\gamma}$.
\end{lemma}

\begin{lemma}\label{lem:bdd-very-small-noise}
Let $n \ge 7$, let $z$ be a complex number with $|z|\le 1/4$, let $\Jmat$ be the $n$ by $n$ matrix defined in \eqref{eq:defT}, and let $\pone$ be an $n$ by $n$ matrix with each entry uniformly bounded in absolute value by $|z|^n/3$.  Then the smallest singular value of $\Jmat+\pone-zI$ is at least $0.19|z|^n$.
\end{lemma}

To prove these lower bounds, we follow an approach similar to the one outlined by Rudelson and Vershynin in \cite{RVershynin_nonasym_2010}, showing that for any unit vector $x=(x_1,\dots,x_n)^{\mathrm{T}}$, we must have that $\|(\Jmat+\pone-zI)x\| \ge c n^{-\gamma}$ where $c$ is an absolute constant (in particular, we will take $c=0.15$).  We will consider two cases: first where $x$ does not have entries approximating a geometric progression, and second where $x$ does have entries that approximate a geometric progression, which is formalized in Lemma~\ref{lem:bigcoord}.  The intuition for using these two cases is that the smallest singular value for $\Jmat-zI$ is exponentially small in $|z|$ when $|z|$ is small, and the unit vector 
\begin{equation}\label{e:vzero}
v_0(z)=(1,z, z^2, z^3,\dots, z^{n-1})\pfrac{1-|z|^2}{1-|z|^{2n}}^{1/2} 
\end{equation}
produces the exponentially small vector $(\Jmat-zI)v_0(z)= (0,\dots, 0, z^n)\pfrac{1-|z|^2}{1-|z|^{2n}}^{1/2}$; thus, one might expect that the unit singular vector for the smallest singular value of a perturbation of $\Jmat-zI$ would have a similar structure. 

We also need a correspondingly sharp upper bound on the smallest singular value, and for that we use the following lemma in both cases of Proposition~\ref{prop:leaving-the-disk}.

\begin{lemma}\label{lem:upper-bd-small-sing}
Let $n\ge 7$ be a positive integer and let $D$ be a real number depending on $n$ that satisfies $Dn^{3/2} \le 0.001$.  Let $z$ be a complex number with $|z|\le 1/4$, let $\Jmat$ be the $n$ by $n$ matrix defined in \eqref{eq:defT}, and let $\pone$ be an $n$ by $n$ matrix with each entry bounded in absolute value by $D$. Then, the smallest singular value of $\Jmat+\pone-zI$ is at most $1.0011\left(|z|^n + D\pfrac{1.001}{(1-|z|)^2}\right)$.
\end{lemma}

Here, the proof is constructive, showing that, for a given perturbation $\pone$, one can construct a perturbation of the vector $v_0(z)$ from Equation~\eqref{e:vzero} that exhibits enough cancellation to produce a sufficiently small singular value.  Full proofs of Lemmas~\ref{lem:smallsing}, \ref{lem:bdd-very-small-noise}, and \ref{lem:upper-bd-small-sing} appear in Section~\ref{subsec:PropProof}.

\begin{proof}[Proof of Proposition~\ref{prop:leaving-the-disk}(ii)]
We will use bounds on the smallest singular value in combination with Theorem~\ref{thm:non} and Proposition~\ref{prop:small-cancel}.
We will compare the two matrices $\Jmat+\pone$ and $\Jmat+\ptwo$, where we recall that $\ptwo$ is the $n$ by $n$ matrix with entry $(n,1)$ equal to $n^{-\gamma}$ and all other entries equal to zero, and that $\pone$ is a deterministic matrix in which each entry is $-n^{-\gamma}$ or $n^{-\gamma}$.  

From \details{the first lower bound in~}Lemma~\ref{lem:determinstic-singbd-plus-epsilon}, we know that the smallest singular value of $\Jmat+\ptwo-zI$ is at least $ \frac{n^{-\gamma}\sqrt{(1-(1/4))^3}}{\sqrt{2n^{-2\gamma+1} +8}} \ge 0.229n^{-\gamma}$, where we used the assumptions from Proposition~\ref{prop:leaving-the-disk} that $\gamma \ge 5$ and $n \ge 2 \gamma\log\gamma$.  Also, by Lemma~\ref{lem:upper-bd-small-sing} with $D=n^{-\gamma}$, we know that the smallest singular value of $\Jmat+\ptwo-zI$ is at most $(1.0011)\left(\pfrac{1}{4}^n + \frac{n^{-\gamma}16.016}{9}\right)\le 1.0011\left(0.0003 n^{-\gamma} +\frac{n^{-\gamma}16.016}{9}\right) \le 1.782n^{-\gamma}$ when $n\ge 2\gamma \log \gamma$ and $\gamma\ge 5$.

From Lemma~\ref{lem:smallsing}, we know that the smallest singular value of $\Jmat+\pone-zI$ is at least $ 0.15 n^{-\gamma}$, and by Lemma~\ref{lem:upper-bd-small-sing}, we know that the smallest singular value of $\Jmat+\pone- zI$ is at most $1.782 n^{-\gamma}$\details{~(by the same reasoning is in the previous paragraph)}.

Weyl's perturbation theorem (Theorem~\ref{thm:weyl}) along with Lemma~\ref{lem:det-sing-bd} shows that $\sigma_{n-1} (\Jmat+\pone-zI), \sigma_{n-1}(\Jmat+\ptwo-zI) \ge 3/4-n^{-\gamma+1} > 0.74 $.
We may now apply Proposition~\ref{prop:small-cancel} with $\epsilon =0.74$, using the facts that $0.15n^{-\gamma} \le \sigma_n(\Jmat+\pone-zI), \sigma_n(\Jmat+\ptwo -zI) \le 1.782 n^{-\gamma}$, to get that
$$\abs{ \mathcal L_{\Jmat+\pone}(z)- \mathcal L_{\Jmat+\ptwo}(z) } \le \frac{\log(1.782/0.15)}{n}+\frac1{0.74}2n^{-\gamma+1} 
\details{\le \frac1n\left( \log(1.782/0.15)+\frac{17^{-3}}{0.37}\right)  }
< \frac{2.476}{n},
$$
when $n \ge 2\gamma\log \gamma$ and $\gamma\ge 5$\details{, which implies $n \ge 17$}.

We can now use the inequality above to apply Theorem~\ref{thm:non}.  Note that $\|\Jmat+\pone-zI\|+\|\Jmat+\ptwo-zI\| \le 2(|z|+1) +2n^{-\gamma+1}\le 2.5 +2n^{-\gamma+1} \le e$ whenever $\gamma\ge 5$ and $n\ge \gamma\log\gamma$.  Thus, we may take $T=e$ to satisfy the first assumption of Theorem~\ref{thm:non} with any positive $\eps$. Also, by the application of Proposition~\ref{prop:small-cancel}, in the previous paragraph, we may satisfy the second condition of Theorem~\ref{thm:non} by taking $\eta =\frac{2.476}n$, again using any positive value for $\eps$.  We also note that from the assumptions on $\varphi$, we have that  
$\Csixone \le 2\frac{\sqrt{\pi}}{4}  \|\Delta \varphi\|_\infty \pfrac32 \log_2(2.75)/\sqrt \pi <  (1.095)\|\Delta\varphi\|_\infty$.

Because the bound in Theorem~\ref{thm:non} works for any integer $m$, and because we are able to satisfy the two assumptions with any positive $\varepsilon$, we may follow Remark~\ref{rem:mlim} and set $\eps = 1/m^{3/2}$ and take the limit as $m\to\infty$, in which case $\frac{\log T}{m\sqrt\eps}\to 0$ and $2(m+1)\eps \to 0$, thus proving that
$$\frac1n\abs{ \sum_{i=1}^n \varphi(\lambda_i(\Jmat+\pone)) } \le \Csixone \eta  < 
1.095\|\Delta \varphi\|\frac{2.476}{n} <  3 \frac{\|\Delta \varphi\|_\infty}{n}, $$
given the assumptions that $\gamma \ge 5$ and $n \ge 2\gamma\log \gamma$. Here, we used that, under the assumptions on $\gamma$ and $n$, none of the eigenvalues of $\Jmat+\ptwo$ lie in $\{z \in \mathbb{C} : |z| \leq 1/4\}$ and so $\sum_{i=1}^n \varphi(\lambda_i(\Jmat + \ptwo)) = 0$.    
\end{proof}

The proof of Proposition~\ref{prop:leaving-the-disk}(i) is similar to the proof of Proposition~\ref{prop:leaving-the-disk}(ii).  We present the proof of Proposition~\ref{prop:leaving-the-disk}(i) in Appendix~\ref{sec:deter-lemmas}, along with the proofs of Lemmas~\ref{lem:smallsing}, \ref{lem:bdd-very-small-noise}, and \ref{lem:upper-bd-small-sing}.

\section{Local law} \label{sec:locallaw}

This section is devoted to a local law for perturbed banded Toeplitz matrices, which compares the eigenvalues to the classical locations at small scales.  The following main result is a generalized version of conclusion \eqref{item:locallaw} from Theorem \ref{thm:sample}.  

\begin{theorem}[Local law] \label{thm:local}
Fix $C_0 > 0$, and let $\{a_j \}_{j \in \mathbb{Z}}$ be a sequence of complex numbers, indexed by the integers, so that $|a_j| \leq C_0$ for all $j \in \mathbb{Z}$. Let $k < n/2$ be a non-negative integer, and let $A$ be an $n \times n$ Toeplitz matrix with symbol $\{a_j \}_{j \in \mathbb{Z}}$ truncated at $k$. 
Let $\gamma > 0$, and let $E$ be an $n \times n$ random matrix which satisfies one of the following:
\begin{enumerate}[(i)]
\item \label{item:iidentries} the entries of $E$ are iid copies of a random variable $\xi$ with mean zero and unit variance.  
\item $E$ is a Haar distributed unitary random matrix. 
\end{enumerate} 
In addition, let the function $f: S^1 \to \mathbb{C}$ be given by
\[ f(\omega) = \sum_{|j| \leq k} a_j \omega^j, \]
and define the classical locations $\lambda_i' = f(\omega_n^i)$ for $1 \leq i \leq n$, where $\omega_n$ is the primitive $n$-th root of unity defined in \eqref{eq:omegan}.  
Let $\varphi: \mathbb{C} \to \mathbb{C}$ be a smooth function supported in the ball $B(0, C_0)$, and let $z_0 \in \mathbb{C}$ with $|z_0| \leq C_0$.  For any $a \geq 0$, let $\varphi_{z_0}(z) = \varphi(n^a (z - z_0))$ be the rescaling of $\varphi$ to size order $n^{-a}$ around $z_0$.  Then, for any $\kappa > 0$, there exists a constant $C > 0$ (depending only on $\varphi$, $\gamma$, $a$, $C_0$, $\kappa$, and the distribution of $\xi$ in case \eqref{item:iidentries}) so that
\begin{equation} \label{eq:llconc}
	\left| \sum_{i=1}^n \varphi_{z_0}( \lambda_i(A + n^{-\gamma} E)) - \sum_{i=1}^n \varphi_{z_0}(\lambda_i') \right| \leq C \left( k \log n + n^{a+1 - \gamma} \|E\|_2 \right) 
\end{equation} 
with probability at least $1 - C n^{-\kappa}$.  
Moreover, the bound in \eqref{eq:llconc} also holds for any $a, a' \geq 0$ when 
\begin{equation} \label{eq:lllogvarphi}
	\varphi_{z_0}(z) = \varphi \left( \frac{n^{a}}{(\log n)^{a'}} \left( z - z_0 \right) \right). 
\end{equation} 
\end{theorem}

A few remarks concerning Theorem \ref{thm:local} are in order.  We have stated the local law here in terms of a test function $\varphi$, similar to other versions of the local law for non-Hermitian matrices in the random matrix theory literature, including \cite{MR3683369,MR3770875,MR3622892,MR3230004,MR3278919}.  Interestingly, the error term in Theorem \ref{thm:local} is quite different than the error term for the local circular law \cite{MR3230002,MR3683369,MR3770875,MR3230004,MR3278919} and holds at much smaller scales.  This is likely due to the fact that the limiting spectral distribution for perturbed Toeplitz matrices is supported on a one-dimensional curve (see Theorem \ref{thm:pert-toep}), while the limiting distribution for iid matrices is supported on the two-dimensional unit disk.   We also note that when $\gamma$ is sufficiently large and $k = O(1)$, the error term in \eqref{eq:llconc} is $O(\log n)$, even for large values of $a$, which is significantly different than the polynomial bounds found in \cite{MR3230002,MR3683369,MR3770875,MR3230004,MR3278919} which grow larger as $a$ increases.  

While $a$ can be taken arbitrary large in Theorem \ref{thm:local}, the error term in \eqref{eq:llconc} can become larger than $\sum_{i=1}^n \varphi_{z_0}(\lambda_i')$ when $a \geq 1$.  We have stated Theorem \ref{thm:local} for only two ensembles of random matrices, but the method holds more generally provided the random matrix $E$ satisfies conditions similar to those in Theorem \ref{thm:rate2}.  For simplicity, we have not tried to state a more general version of Theorem \ref{thm:local}.  

Part \eqref{item:locallaw} from Theorem \ref{thm:sample} is just a simplified version of of Theorem \ref{thm:local} with $A = \Jmat$.  In addition, conclusion \eqref{item:dist} and the $O_r(\log n)$ bound in part \eqref{item:inliers} of Theorem \ref{thm:sample} also follow as direct corollaries of Theorem \ref{thm:local}.  

Theorem \ref{thm:local} can be used to compare two randomly perturbed Toeplitz matrices. For example, if $A$ is a banded Toeplitz matrix satisfying the conditions of Theorem \ref{thm:local} and both $E$ and $E'$ are $n \times n$ random matrices satisfying the assumptions of the theorem, then we can conclude that
\[ \left| \sum_{i=1}^n \varphi_{z_0}( \lambda_i(A + n^{-\gamma} E)) - \sum_{i=1}^n \varphi_{z_0}(\lambda_i(A + n^{-\gamma} E')) \right| \leq C \left( k \log n + n^{a+1 - \gamma} \left( \|E\|_2 + \|E'\|_2 \right) \right) \]
with probability at least $1 - 2C n^{-\kappa}$.  This shows a strong type of universality, where even the local behaviors of the eigenvalues are universal, irregardless of the ensemble for the random perturbation.  
  
Perhaps the closest results to Theorem \ref{thm:local} in the literature come from the work of Sj\"{o}strand and Vogel \cite{MR4200678}.  They show precise asymptotic bounds for the number of eigenvalues in smooth domains of the matrix $A + n^{-\gamma} E$, where $A$ is a banded Toeplitz matrix and $E$ is a random matrix with iid standard complex Gaussian entries.  There are a few differences between their work and Theorem \ref{thm:local}.  For example, Theorem \ref{thm:local} applies to a larger class of random matrices compared to the results in \cite{MR4200678}.  In addition, Theorem \ref{thm:local} applies to the linear statistics of the eigenvalues, rather than the counting function of the eigenvalues.  There are also some similarities between Theorem \ref{thm:local} and the work of Davies and Hager \cite{MR2490477}, which provides some bounds on the radial locations of the eigenvalues of $\Jmat + n^{-\gamma}E$, where $\Jmat$ is defined in \eqref{eq:defT} and $E$ is a random matrix.  We note that Theorem \ref{thm:local} locates the eigenvalues at a finer scale, does not involve only the radial components, and applies to a larger class of banded Toeplitz matrices.  

It remains an open question whether the bound in \eqref{eq:llconc} is sharp (for any choice of parameters).  In particular, it is an interesting question to understand the optimal dependence on $\gamma$ and $a$ in the error term.  

\subsection{Proof of Theorem \ref{thm:local}}
In order to prove Theorem \ref{thm:local}, we will need the following sampling lemma from \cite{TV-univ_zeros}.  
\begin{lemma}[Monte Carlo sampling lemma; Lemma 6.1 from \cite{TV-univ_zeros}] \label{lemma:monte}
Let $(X, \rho)$ be a probability space, and let $F: X \to \mathbb{C}$ be a square integrable function.  Let $m\geq 1$, let $Z_1, \ldots, Z_m$ be drawn independently at random from $X$ with distribution $\rho$, and let $S$ be the empirical average
\[ S := \frac{1}{m} ( F(Z_1) + \cdots + F(Z_m) ). \]
Then, for any $\delta > 0$, one has the bound
\[ \left| S - \int_X F \d \rho \right| \leq \frac{1}{\sqrt{m \delta} } \left(  \int_X \left| F - \int_X F \d \rho \right|^2 \d \rho \right)^{1/2} \]
with probability at least $1 - \delta$.   
\end{lemma}

We now turn to the proof of Theorem \ref{thm:local}.  Assume $\varphi$ is nonzero as the result is trivial otherwise.  We will only consider the case when $\varphi_{z_0}(z) = \varphi(n^a (z - z_0))$ as the case when $\varphi$ is given by \eqref{eq:lllogvarphi} is nearly identical.  

Throughout the proof, in order to ease notation, we will not always indicate when an implicit constant in the asymptotic notation depends on $\varphi$, $\gamma$, $\kappa$, $a$, $C_0$, or the distribution of $\xi$.  

Let $A'$ be the matrix from Lemma \ref{lemma:lowrank} with rank at most $2k$ so that $A + A'$ has eigenvalues $\lambda_1', \ldots, \lambda_n'$.  By Lemma \ref{lemma:circulant}, $A + A'$ is normal.  It follows from the deterministic bound for normal matrices due to Sun \cite{MR1407668} that there is a permutation $\sigma: [n] \to [n]$ so that
\[ \sqrt{\sum_{i=1}^n \left| \lambda_i( A + A' + n^{-\gamma} E) - \lambda_{\sigma(i)}' \right|^2} \leq  n^{1/2-\gamma} \|E\|_2. \]
Since $\varphi$ is smooth and compactly supported, it follows that $\varphi_{z_0}$ is Lipschitz continuous with Lipschitz constant $O_\varphi(n^a)$.  Thus, by the Cauchy--Schwarz inequality, we have
\begin{align*}
	\left| \sum_{i=1}^n \varphi_{z_0} (\lambda_i(A + A' + n^{-\gamma}E)) - \sum_{i=1}^n \varphi_{z_0}(\lambda_i') \right| &\ll_\varphi n^a \sum_{i=1}^n \left| \lambda_i(A + A' + n^{-\gamma}E) - \lambda_{\sigma(i)}' \right| \\
	&\leq n^{1/2 + a} \sqrt{\sum_{i=1}^n \left| \lambda_i( A + A' + n^{-\gamma} E) - \lambda_{\sigma(i)}' \right|^2}  \\
	&\leq n^{1 + a - \gamma} \|E\|_2.
\end{align*}

Therefore, it suffices to show that there exists a constant $C > 0$ so that
\begin{equation} \label{eq:llsumsbndstep2}
	\left| \sum_{i=1}^n \varphi_{z_0}( \lambda_i(A + n^{-\gamma} E)) - \sum_{i=1}^n \varphi_{z_0}(\lambda_i(A + A' + n^{-\gamma}E)) \right| \leq C k \log n  
\end{equation} 
with probability at least $1 - C n^{-\kappa}$.  
We begin with the case when the entries of $E$ are iid copies of a random variable $\xi$ with mean zero and unit variance.  In this case, it follows that $\E \|E\|^2 \leq \E \|E\|^2_2 = n^2$, and so by Markov's inequality we have 
\begin{equation} \label{eq:llEnorm}
	\|E \| \leq n^{\kappa + 1} 
\end{equation} 
with probability at least $1 - n^{-2\kappa}$.  
In addition, we note that
\begin{equation} \label{eq:llAnorm}
	\|A\| \leq \|A\|_2 \leq n \max_{|j| \leq k} |a_j| = O(n)
\end{equation} 
and similarly, by \eqref{eq:A'F}, 
\begin{equation} \label{eq:llA'norm}
	\|A + A' \| = O(n). 
\end{equation} 

By Green's formula (see, for instance, Section 2.4.1 in \cite{HKPV}) and an integral substitution, we can write 
\[ \sum_{i=1}^n \varphi_{z_0}(\lambda_i(A + n^{-\gamma} E)) = \frac{1}{2 \pi} \int_{\mathbb{C}} \Delta \varphi(z) \log | \det( A + n^{-\gamma}E - z_0I - n^{-a} zI) | \d^2 z \]
and
\[ \sum_{i=1}^n \varphi_{z_0}(\lambda_i(A + A' + n^{-\gamma} E)) = \frac{1}{2 \pi} \int_{\mathbb{C}} \Delta \varphi(z) \log | \det( A + A' + n^{-\gamma}E - z_0I - n^{-a} zI) | \d^2 z. \]
Thus, we obtain
\begin{equation} \label{eq:llsumbnd}
	\sum_{i=1}^n \varphi_{z_0}(\lambda_i(A + n^{-\gamma} E)) - \sum_{i=1}^n \varphi_{z_0}(\lambda_i(A + A' + n^{-\gamma} E)) = \int_{K} f_n(z) \d^2 z, 
\end{equation} 
where 
 $K := \supp (\Delta \varphi)$ and 
\[ f_n(z) := \frac{1}{2 \pi} \Delta \varphi(z) \left( \log | \det( A + n^{-\gamma}E - z_0I - n^{-a} zI)|  - \log | \det( A + A' + n^{-\gamma}E - z_0I - n^{-a} zI) | \right). \]

Our goal is to apply Lemma \ref{lemma:monte} to the function $f_n$.  To this end, we have
\begin{align*} 
	\int_{K} &|f_n(z)|^2 \d^2 z \\
	&\ll_{\varphi} \int_{K} \log^2 | \det (A + n^{-\gamma}E - z_0I - n^{-a} zI)| \d^2 z + \int_K \log^2 |\det( A + A' + n^{-\gamma}E - z_0I - n^{-a} zI) | \d^2 z, 
\end{align*} 
where $\Delta \varphi$ was bounded by its supremum norm and absorbed into the implicit constant.  Since, by \eqref{eq:llEnorm} and \eqref{eq:llAnorm}, $\|A + n^{-\gamma} E \| = O(n^{1 + \kappa})$ with probability $1 - O(n^{-\kappa}$), we have
\begin{align*}
	\int_K \log^2 | \det (A + n^{-\gamma}&E - z_0I - n^{-a} zI)| \d^2 z \\
	&= \int_K \left( \sum_{i=1}^n \log | \lambda_i(A + n^{-\gamma}E - z_0I - n^{-a} zI)| \right)^2 \d^2 z \\
	&\ll n^2 \sup_{\lambda \in \mathbb{C} : |\lambda| \leq O(n^{1 + \kappa})} \int_{K} \log^2 | \lambda - z_0 - n^{-a} z | \d^2 z \\
	&\ll n^2 \left( \log^2 n + \sup_{\lambda \in \mathbb{C} : |\lambda| \leq O(n^{1 + \kappa})} \int_K \log^2 | n^{a} (\lambda - z_0) - z| \d^2 z \right)
\end{align*} 
on the same event.  Here, we used the Cauchy--Schwarz inequality in the first inequality and the fact that $K \subset B(0, C_0)$ has Lebesgue measure at most $\pi C_0^2 = O(1)$ in the second bound.  Since $\log | \cdot |$ is locally square integrable, we conclude that
\[ \int_K \log^2 | \det (A + n^{-\gamma}E - z_0I - n^{-a} zI)| \d^2 z \ll n^2 \log^2 n \]
with probability $1 - O(n^{-\kappa}$).  An analogous argument (using \eqref{eq:llA'norm} instead of \eqref{eq:llAnorm}) shows that
\[ \int_K \log^2 |\det( A + A' + n^{-\gamma}E - z_0I - n^{-a} zI) | \d^2 z \ll  n^2 \log^2 n \]
with probability $1 - O(n^{-\kappa}$).  We conclude that
\begin{equation} \label{eq:llintKfn}
	\int_{K} |f_n(z)|^2 \d^2 z \ll n^2 \log^2 n 
\end{equation} 
with probability $1 - O(n^{-\kappa}$).

Let $|K|$ be the Lebesgue measure of $K$; as noted above $K \subset B(0, C_0)$ and so $|K| = O(1)$.    Set $m := \lceil n^{4 + 2 \kappa} \rceil$, and let $Z_1, \ldots Z_m$ be iid random variables which are uniform on $K$. Then, by Lemma \ref{lemma:monte} and \eqref{eq:llintKfn}, we have
\[ \int_K f_n(z) = \frac{|K|}{m} \sum_{i=1}^m f_n(Z_i) + O \left( \frac{ n \log n}{ m^{1/4}} \right) \]
with probability $1 - O(m^{-1/2}) - O(n^{-\kappa})$.  By our choice of $m$, this implies that
\begin{equation} \label{eq:llintKfnz}
	\int_K f_n(z) = \frac{|K|}{m} \sum_{i=1}^m f_n(Z_i) + o(1) 
\end{equation} 
with probability $1 - O(n^{-\kappa})$.  Since $|K| = O(1)$, it remains to show that
\begin{equation} \label{eq:llfnsup}
	\sup_{i \in [m]} |f_n(Z_i)| \ll k \log n 
\end{equation} 
with probability $1 - O(n^{-\kappa})$.  

From \eqref{eq:llEnorm} and the least singular value bound given in Theorem 2.1 of \cite{MR2409368} (by conditioning on $Z_i, 1 \leq i \leq m$), there exists $b > 0$ (depending on $\kappa, \gamma$, $C_0$ and the distribution of $\xi$) so that the event 
\begin{align*}
	\Omega &:= \left \{ \|E\| \leq n^{1 + \kappa} \right\} \bigcap \\
	&\qquad \bigcap_{i=1}^m \left\{ \sigma_{\min}(A + n^{-\gamma} E- z_0I - n^{-a} Z_iI) \geq n^{-b}, \sigma_{\min}(A + A' + n^{-\gamma} E- z_0I - n^{-a} Z_iI) \geq n^{-b} \right\} 
\end{align*} 
holds with probability $1 - O(n^{-\kappa})$.  Therefore, by Theorem \ref{thm:replrank}, it follows that
\[ \sup_{i \in [m]} |f_n(Z_i)| \ll \|\Delta \varphi\|_{\infty} k \log n \]
on $\Omega$ since $\rank(A') \leq 2k$.  Here, $\|\Delta \varphi\|_{\infty} $ is the $L^\infty$-norm of $\Delta \varphi$.  This establishes \eqref{eq:llfnsup}.  Combining \eqref{eq:llfnsup} with \eqref{eq:llsumbnd} and \eqref{eq:llintKfnz}, we conclude that
\[ \left| \sum_{i=1}^n \varphi_{z_0}(\lambda_i(A + n^{-\gamma} E)) - \sum_{i=1}^n \varphi_{z_0}(\lambda_i(A + A' + n^{-\gamma} E)) \right| \ll k \log n \]
with probability $1 - O(n^{-\kappa})$, which yields \eqref{eq:llsumsbndstep2} and completes the proof in the case when $E$ has iid entries with mean zero and unit variance.    

The case when $E$ is a Haar distributed unitary matrix is similar.  In this case, one trivially has $\|E\| = 1$.  In addition, from Theorem 1.1 of \cite{MR3164983}, there exists $b > 0$ so that the event $\Omega$ holds with probability $1 - O(n^{-\kappa})$.  With these changes, the proof in this case is nearly identical to the proof given above; we omit the details.

\appendix

\section{Proofs of Theorems \ref{thm:non}, \ref{thm:replacement}, \ref{thm:replrank}, and Proposition~\ref{prop:small-cancel}} \label{sec:prooftheory} 

In this section, we give proofs of Theorems \ref{thm:non}, \ref{thm:replacement},  \ref{thm:replrank}, and Proposition~\ref{prop:small-cancel}, and provide further details for Example~\ref{eg:smallblocks}.   

\subsection{Proof of Theorem \ref{thm:replacement}}

We begin with the proof of Theorem \ref{thm:replacement}.  We will need the following lemma.  

\begin{lemma}[Lemma 4.3 from \cite{2008.03850}] \label{lemma:ibp}
For any probability measures $\mu$ and $\nu$ on $\mathbb{R}$ and any $0 < a < b$,
\[ \left| \int_{a}^b \log(x) d \mu(x) - \int_a^b \log(x) d \nu(x) \right| \leq 2 [ | \log b | + |\log a| ] \| \mu - \nu \|_{[a,b]}, \]
where
	\[ \| \mu - \nu \|_{[a,b]} := \sup_{x \in [a,b]} | \mu( [a,x]) - \nu([a,x]) |. \]
\end{lemma}

\details{
\begin{proof}
We rewrite
\begin{align*}
	\int_a^b \log(x) d \mu(x) = \log(b) \mu([a,b]) - \int_{a}^b \int_x ^b \frac{1}{t} dt d\mu(x). 
\end{align*}
Applying Fubini's theorem, we deduce that
\[ \int_{a}^b \int_x ^b \frac{1}{t} dt d\mu(x) = \int_a^b \frac{ \mu([a,t]) }{t} dt. \]
Similarly, the same equalities apply to $\nu$.  Thus, we obtain that
\begin{align*}
	\left| \int_{a}^b \log(x) d \mu(x) - \int_a^b \log(x) d \nu(x) \right| &\leq |\log (b)| | \mu([a,b]) - \nu([a,b]) | + \left| \int_a^b \frac{ \mu([a,t]) - \nu([a,t]) }{t} dt \right| \\
	&\leq | \log b| \| \mu - \nu \|_{[a,b]} + \| \mu - \nu \|_{[a,b]} \int_{a}^b \frac{1}{t} dt,
\end{align*}
from which the conclusion follows.  
\end{proof}
}

We now turn to the proof of Theorem \ref{thm:replacement}.  For simplicity we define 
\[ D:= \max\{1, \| M_1 - zI\|, \| M_2 - zI \| \}. \]
By writing the logarithmic potential in terms of the logarithm of the determinant, we find
\begin{align*}
	\mathcal{L}_{ M_i} (z) &= \int_0^\infty \log(x) d \nu_{ M_i - zI} (x) \\
	&= \int_{\sigma_{\min}}^D \log(x) d \nu_{M_i - zI}(x) \\
	&= \int_{[\sigma_{\min}, \eps/2]} \log(x) d \nu_{M_i - zI}(x) + \int_{(\eps/2, D]} \log(x) d \nu_{M_i - zI}(x). 
\end{align*} 
Using Weyl's inequality for singular values (Theorem~\ref{thm:weyl}) and the assumption that $\| M_1 - M_2 \| < \eps/2$, we have
\begin{equation} \label{eq:smallcomp}
	\nu_{M_1 - zI}([0, \eps/2]) \leq \nu_{M_2 - zI}([0, \eps]), 
\end{equation}
and so by the triangle inequality
\[ \| \nu_{M_1 - zI} - \nu_{M_2 - zI} \|_{[0, \eps/2]} := \sup_{x \in [0, \eps/2]} | \nu_{M_1 - zI}([0,x]) - \nu_{M_2 - zI}([0, x]) | \leq 2 \nu_{M_2 - zI}([0, \eps]). \]
Thus, by Lemma \ref{lemma:ibp}, we conclude that
\begin{align} 
	&\left| \int_{[\sigma_{\min}, \eps/2]} \log(x) d \nu_{M_1 - zI}(x) - \int_{[\sigma_{\min}, \eps/2]} \log(x) d \nu_{M_2 - zI}(x) \right| \nonumber\\
	&\qquad\qquad\leq 4 \left( | \log(\eps/2)| + |\log \sigma_{\min} | \right) \nu_{M_2 - zI}([0, \eps]).  \label{eq:conclusion1}
\end{align}

It remains to bound
\[ \left| \int_{(\eps/2, D]} \log(x) d \nu_{M_1 - zI}(x) - \int_{(\eps/2, D]} \log(x) d \nu_{M_2 - zI}(x) \right|. \]

Recall that $\sigma_1(M_i - zI) \geq \cdots \geq \sigma_n(M_i - zI)$ denote the ordered singular values of $M_i - zI$ (which are all contained in the interval $[0,D]$).  We will show that most singular values can be paired so that the difference of the integrals above is not too large, and the singular values that cannot be paired will also contribute a well-controlled error.  Let $j_1$ be the number of singular values of $M_1 - zI$ in $(\eps/2,D]$, and let $j_2$ be the number of singular values of $M_2 - zI$ in $(\eps/2,D]$. Then 
\begin{align}
\nonumber	&\left| \int_{\eps/2}^D \log(x) d \nu_{M_1 - zI}(x) - \int_{\eps/2}^D \log(x) d \nu_{M_2 - zI}(x) \right| \\
		&\qquad\qquad\leq \frac{1}{n}\abs{ \sum_{j=1}^{j_1} \log (\sigma_j(M_1 - zI)) - \sum_{j=1}^{j_2} \log(\sigma_j(M_2 - zI)),} \label{eq:firstsumj} 
\end{align}
where we use the convention that if the lower index on the sum exceeds the upper\details{~(so 
$j_1 = 0$ or $j_2 = 0$ in this case)}, then the respective sum in \eqref{eq:firstsumj} is empty and hence equal to zero.  Thus, we obtain 
\begin{align} 
\nonumber	&\left| \int_{\eps/2}^D \log(x) d \nu_{M_1 - zI}(x) - \int_{\eps/2}^D \log(x) d \nu_{M_2 - zI}(x) \right| \\
\label{eq:mainsum}	
&\qquad\qquad\leq \frac{1}{n} \sum_{j=1}^{\min\{j_1,j_2\}} \abs{ \log (\sigma_j(M_1 - zI)) - \log(\sigma_j(M_2 - zI))}\\
\label{eq:secondsum}	&\qquad\qquad\qquad\qquad+\frac{1}{n}\sum_{j=\min\{j_1,j_2\} + 1}^{\max\{j_1,j_2\}}\abs{ \log (\sigma_j(M_K - zI))} 
\end{align}
where $K=1$ if $j_1 > j_2$, $K=2$ if $j_2 > j_1$ and the value of $K$ is irrelevant if $j_1=j_2$, since then the sum \eqref{eq:secondsum} equals zero using the same convention as before. 
When $j_1\ne j_2$, we will show that \eqref{eq:secondsum} is still small.  Indeed, note that each singular value that appears in the sum \eqref{eq:secondsum} must have a corresponding singular value from the other matrix that lies in the interval $[0,\eps/2]$. Thus, we know the number of terms in the sum \eqref{eq:secondsum} is at most 
$$\max\{n\nu_{M_1-zI}([0,\eps/2]),n\nu_{M_2-zI}([0,\eps/2])\} \le 2n\nu_{M_2-zI}([0,\eps])$$ 
(using \eqref{eq:smallcomp} for the inequality).  Furthermore, by Weyl's perturbation theorem for singular values (Theorem~\ref{thm:weyl}), we know each singular value that appears as a summand in \eqref{eq:secondsum} is bounded above by $\eps<1$, which means each summand is at most $\abs{\log(\eps/2)}$.  Thus, we have that 
$$\frac{1}{n}\sum_{j=\min\{j_1,j_2\}+1}^{\max\{j_1,j_2\}}\abs{ \log (\sigma_j(M_K - zI))} \le 2\abs{\log(\eps/2)}\nu_{M_2-zI}([0,\eps]). $$

Note that the function $\log(\cdot)$ is $\pfrac{2}{\eps}$-Lipschitz continuous on $(\eps/2,D]$, and so bounding the sums in \eqref{eq:mainsum} and \eqref{eq:secondsum}, we see that
\begin{align*}
	&\left| \int_{\eps/2}^D \log(x) d \nu_{M_1 - zI}(x) - \int_{\eps/2}^D \log(x) d \nu_{M_2 - zI}(x) \right| \\
	&\qquad\qquad\le  2\abs{\log(\eps/2)}\nu_{M_2-zI}([0,\eps]) + \frac{2}{n\eps} \sum_{j=1}^n |\sigma_j(M_1 - zI) - \sigma_j(M_2 - zI)| .
\end{align*}
Applying Weyl's perturbation theorem for singular values (Theorem \ref{thm:weyl}), we see that
\[ \frac{1}{n} \sum_{j=1}^n  |\sigma_j(M_1 - zI) - \sigma_j(M_2 - zI)| \leq \| M_1 - M_2 \|. \]
Combining the last two inequalities above with \eqref{eq:conclusion1} gives us a final bound of 
\begin{align} \label{e:betterbd}
\abs{\mathcal L_{M_1}(z) -  \mathcal L_{M_2}(z)} &\le 
 \left( 6 | \log(\eps/2)| + 4|\log \sigma_{\min} | \right) \nu_{M_2 - zI}([0, \eps]) +  \frac{2}{\eps}\| M_1 - M_2 \| \\
\nonumber &\le  6\left(| \log(\eps/2)| + |\log \sigma_{\min} | \right) \nu_{M_2 - zI}([0, \eps]) +  \frac{2}{\eps}\| M_1 - M_2 \|,
\end{align}
which completes the proof.

\subsection{Proof of Theorem \ref{thm:replrank}}
We now turn to the proof of Theorem \ref{thm:replrank}.  
Since 
\[ \mathcal{L}_{M_i}(z) = \int_0^\infty \log(x) d \nu_{M_i - zI}(x) = \int_{\sigma_{\min}}^{\sigma_{\max}} \log(x) d \nu_{M_i - zI}(x) \]
for $i = 1,2$, it follows from Lemma \ref{lemma:ibp} that
\[ \left| \mathcal{L}_{M_1}(z) - \mathcal{L}_{M_2}(z) \right| \leq 2 \left( | \log \sigma_{\min} | + | \log \sigma_{\max} | \right) \| \nu_{M_1 - z I} - \nu_{M_2 - z I} \|_{[\sigma_{\min}, \sigma_{\max}]}, \]
where 
\[ \| \mu - \nu \|_{ [a,b]} = \sup_{x \in [a,b]} \left| \mu([a,x]) - \nu([a,x]) \right| \]
for any real numbers $a \leq b$ and any two probability measures $\mu$ and $\nu$.  

It remains to show  
\begin{equation} \label{eq:showrank}
	\| \nu_{M_1 - z I} - \nu_{M_2 - z I} \|_{[\sigma_{\min}, \sigma_{\max}]} \leq \frac{1}{n} \rank (M_1 - M_2). 
\end{equation} 
To this end, let $\nu'_{M_i - zI}$ be the empirical measure constructed from the squared singular values of $M_i - zI$ for $i=1,2$, i.e., 
\[ \nu'_{M_i - zI} = \frac{1}{n} \sum_{j=1}^n \delta_{\sigma^2_j(M_i -z I)}, \qquad i = 1,2. \]
Then for any $x \geq 0$, 
\[ \nu'_{M_i - zI}((-\infty, x^2]) = \nu_{M_i - zI}((-\infty, x]) \]
for $i=1,2$.  Thus, by definition of $\sigma_{\min}$ and $\sigma_{\max}$, we have 
\begin{align*}
	\| \nu_{M_1 - z I} - \nu_{M_2 - z I} \|_{[\sigma_{\min}, \sigma_{\max}]} &= \sup_{x \geq 0} \left| \nu_{M_1 - z I}((-\infty, x]) - \nu_{M_2 - z I}((-\infty, x]) \right| \\
		&= \sup_{x \geq 0} \left| \nu'_{M_1 - z I}((-\infty, x^2]) - \nu'_{M_2 - z I}((-\infty, x^2]) \right| \\
		&\leq \sup_{x \in \mathbb{R}} \left| \nu'_{M_1 - z I}((-\infty, x]) - \nu'_{M_2 - z I}((-\infty, x]) \right|.
\end{align*}
It follows from Theorem A.44 in \cite{BSbook} that
\[ \sup_{x \in \mathbb{R}} \left| \nu'_{M_1 - z I}((-\infty, x]) - \nu'_{M_2 - z I}((-\infty, x]) \right| \leq \frac{1}{n} \rank (M_1 - M_2), \]
which when combined with the bounds above yields \eqref{eq:showrank}.  The proof of the theorem is complete.

\subsection{Proof of Proposition~\ref{prop:small-cancel}}

The idea is to match the smallest singular values and use cancellation, and otherwise proceed as in Theorem~\ref{thm:replacement}.  We may use the hypotheses and the fact that $\log(\cdot)$ is $\pfrac 1\epsilon$-Lipschitz on $[\epsilon,\infty)$ to compute
\begin{align*}
\abs{\mathcal L_{M_1}(z) - \mathcal L_{M_2}(z)} 
&\le \frac1n \abs{ \sum_{j=1}^n \log(\sigma_j(M_1-zI))-\log(\sigma_j(M_2-zI))  }\\
\details{&\le} \details{ \frac 1n \abs{\log(\sigma_n(M_1-zI))-\log(\sigma_n(M_2-zI))} + \frac1n \abs{ \sum_{j=1}^{n-1} \log(\sigma_j(M_1-zI))-\log(\sigma_j(M_2-zI))  }\\}
\details{&\le}\details{ \frac1n\abs{\log(a)-\log(b)} + \frac1n \abs{ \sum_{j=1}^{n-1} \log(\sigma_j(M_1-zI))-\log(\sigma_j(M_2-zI))  }\\}
&\le \frac{\log(b/a)}{n} + \frac1{\epsilon n} \sum_{j=1}^{n-1} \abs{ \sigma_j(M_1-zI)-\sigma_j(M_2-zI)  }\\
&\le \frac{\log(b/a)}{n} + \frac1{\epsilon}\|M_1- M_2\|,\\
\end{align*}
where the last inequality follows from Weyl’s perturbation theorem for singular values (Theorem~\ref{thm:weyl}).

\subsection{Proof of Theorem \ref{thm:non}}
We next give a proof of Theorem \ref{thm:non}.  
Define
\begin{equation}\label{e:fidef}
 f_i(z) := \frac{1}{n} \log | \det (zI - M_i) | 
 \end{equation}
for $i = 1,2$.  
For $r > 0$, we define
\[ B(r) := \{ z \in \mathbb{C} : |z| < r \} \]
to be the ball of radius $r$ centered at the origin.  
We will let $d^2 z$ denote integration with respect to the Lebesgue measure on $\mathbb{C}$, i.e., $\int g(z) \d^2 z$. 
We will need the following results.  

\begin{lemma} \label{lemma:intbnd}
Under the assumptions of Theorem \ref{thm:non}, there exists a constant $\sixthree > 0$ (depending only on $\varphi$) so that 
\[ \max_{i=1,2} \int_{\mathbb{C}} |\Delta \varphi (z)|^2 |f_i(z)|^2 \d^2 z \leq \sixthree \log^2 T \]
with probability at least $1 - \eps$, where we may take the constant $\sixthree$ to be $\|\Delta \varphi\|_\infty^2 \sixthreed$ where $\sixthreed$ is a constant depending only on the diameter and distance from the origin of $\supp(\Delta \varphi)$ (see \eqref{e:sixthreed}).
\end{lemma}
\begin{proof}
Since the result is trivially true when $\varphi \equiv 0$, we assume $\varphi$ is nonzero.  
By bounding $ |\Delta \varphi (z)|^2 $ using the $L^\infty$-norm $\|\Delta \varphi\|^2_{\infty}$, it suffices to show that
\[ \max_{i=1,2} \int_{\supp(\Delta \varphi)} |f_i(z)|^2 \d^2 z \leq \sixthreed  \log^2 T \]
with probability at least $1 - \eps$.  
We will prove this bound on the event where $\|M_1\| + \|M_2\| \leq T$ (which by supposition holds with probability at least $1 - \eps$).    

By the Cauchy--Schwarz inequality, 
\[ |f_i(z) |^2 = \frac{1}{n^2} \left( \sum_{j=1}^n \log |z - \lambda_j(M_i)| \right)^2 \leq \frac{1}{n} \sum_{j=1}^n \log^2 |z - \lambda_j(M_i)|, \]
and hence
\begin{align*}
	\max_{i=1,2} \int_{\supp(\Delta \varphi)} |f_i(z)|^2 \d^2 z &\leq \max_{i=1,2} \max_{1 \leq j \leq n} \int_{\supp(\Delta \varphi)} \log^2 |z - \lambda_j(M_i)| \d^2 z \\
	&\leq \sup_{|\lambda| \leq T} \int_{\supp(\Delta \varphi)} \log^2 |z - \lambda| \d^2 z
\end{align*}
since $\|M_1\| + \|M_2\| \leq T$. 

Let $\diam$ be the diameter of $\supp (\Delta \varphi)$.  We will bound the integral $\int_{\supp(\Delta \varphi)} \log^2 |z - \lambda| \d^2 z$ for an arbitrary $\lambda$ satisfying $|\lambda| \le T$ by considering two cases.  

For the first case, assume that there is some element of the set $\supp(\Delta \varphi) - \lambda$ that is within distance $1$ of the origin.  Then we have that the set $\supp(\Delta \varphi) - \lambda$ is contained in $B(\diam+1)$, so that $\int_{\supp(\Delta \varphi)} \log^2 |z - \lambda| \d^2 z
=\int_{\supp(\Delta \varphi)-\lambda} \log^2 |z | \d^2 z
 \le \int_{B(\diam+1)} \log^2 |z | \d^2 z$.  We can now convert to polar coordinates and use integration by parts twice to get that
\begin{align*}
 \int_{B(\diam+1)} \log^2 |z | \d^2 z &= \int_0^{2\pi} \int_0^{\diam+1} \log^2(r) r \, dr d\theta = 2\pi\int_0^{\diam+1} \log^2(r) r \, dr \\
 &=  2\pi\left(\frac{(\diam+1)^2}{2}\log^2(\diam+1)-\int_0^{\diam+1} r\log (r)\, dr \right)\\ 
\details{ &= }\details{ 2\pi \left(\frac{(\diam+1)^2}{2}\log^2(\diam+1) - \frac{(\diam+1)^2}{2}\log(\diam+1) + \frac{(D+1)^2}{4}\right)\\}
 &=\pi(\diam+1)^2\left( \log^2(\diam+1) - \log(\diam+1) + \frac{1}{2}\right).  
\end{align*}

For the second case, assume conversely that all elements of the set $\supp(\Delta \varphi) - \lambda$ are distance greater than $1$ from the origin. Then, letting $\sdist$ be the distance from $\supp(\Delta \varphi)$ to the origin, we have 
\begin{align*}
\int_{\supp(\Delta \varphi)} \log^2 |z - \lambda| \d^2 z
&=\int_{\supp(\Delta \varphi)-\lambda} \log^2 |z | \d^2 z \\
&\le \int_{\supp(\Delta \varphi)-\lambda} \log^2 |\sdist +\diam +T | \d^2 z \\
\details{&\le}\details{ \log^2(T+\sdist +\diam) \int_{B(z_0, \diam)} \d^2 z  \\}
&\leq \pi \diam^2 \log^2(T+\sdist +\diam).
\end{align*}

\details{In the above, $B(z_0, \diam)$ is a ball of radius $\diam$ centered at $z_0$, where $z_0$ is an arbitrary element of $\supp (\Delta \varphi)-\lambda$.  }
Combining the bounds from the two cases, we have shown that for any $\lambda$ satisfying $|\lambda| \le T$ we have 
$$\int_{\supp(\Delta \varphi)} \log^2 |z - \lambda| \d^2 z \le \log^2(T+\sdist + \diam) \pi (\diam+1)^2 \max\left\{\log^2(\diam+1)-\log(\diam+1)+\frac12, 1 \right\}.$$

To complete the proof of the lemma, we note that $\log^2(T+\sdist+\diam) \le \log^2(T)\frac{\log^2(2+\sdist+\diam)}{\log^2 2}$ (since $\log(T+\sdist+\diam)/\log(T)$ is decreasing in $T$ on the interval $T\in [2,\infty)$ when $\sdist+\diam\ge 0$).  Thus, for the statement of the lemma,
we can set the constant $\sixthree = \|\Delta \varphi(z)\|_{\infty}^2 \sixthreed$, where
we define the constant $\sixthreed$ by
\begin{equation}\label{e:sixthreed}
 \sixthreed= \pi (\diam+1)^2\pfrac{\log^2(2+\sdist+\diam)}{\log^2 2} \max\left\{\log^2(\diam+1)-\log(\diam+1)+\frac12, 1\right\},\end{equation}
which we note depends only on the diameter of $\supp(\Delta \varphi)$ and its distance from the origin.
\end{proof}

\begin{lemma} \label{lemma:applymonte}
Assume the conditions of Theorem \ref{thm:non} and define $f_i$ as in \eqref{e:fidef}.  Define
\begin{equation} \label{def:F}
	F(z) := \Delta \varphi(z) (f_1(z) - f_2(z)). 
\end{equation}
Let $Z_1, \ldots, Z_m$ be uniformly distributed on $K := \supp(\Delta \varphi)$, independent of $M_1$ and $M_2$.  Then 
\[ \left| \int_{\mathbb{C}} F(z) \d^2 z - \frac{|K|}{m} \sum_{j=1}^m F(Z_j) \right| \leq  4 \sqrt{|K| \sixthree} \frac{ \log T}{ m \sqrt{\eps}} \]
with probability at least $1 - (m+1) \eps$, where $\sixthree$ is the constant from Lemma \ref{lemma:intbnd}.  Here, $|K|$ denotes the Lebesgue measure of $K$.  
\end{lemma}
\begin{proof}
Assume $\varphi$ is nonzero as the result is trivial when $\varphi \equiv 0$.  Let $\rho$ denote the uniform probability distribution on $K$.  By Lemma \ref{lemma:intbnd}, it follows that 
\begin{align} 
\nonumber	\int_{\mathbb{C}} |F|^2 \d \rho &= \frac1{|K|} \int_K |\Delta \varphi|^2|(f_1(z)-f_2(z))|^2\, d^2z\\
\nonumber	&\le \frac2{|K|} \int_{\mathbb C} |\Delta \varphi|^2(|f_1(z)|^2+|f_2(z)|^2)\, d^2z    & \mbox{ (Cauchy--Schwarz)}\\
	&\leq \frac{4\sixthree}{|K|} \log^2 T  \label{eq:F2norm}
\end{align} 
with probability at least $1 - \eps$.  Here the constant $\sixthree > 0$ is from Lemma~\ref{lemma:intbnd}.  
Take
\[ S := \frac{1}{m}  \sum_{j=1}^m F(Z_j). \]
By Lemma \ref{lemma:monte}, with probability at least $1 - m \eps$ (with respect to $Z_1, \ldots, Z_m$), 
\begin{equation} \label{eq:Sfd}
	\left| S - \int F \d \rho \right| \leq \frac{2}{m \sqrt{\eps}} \left( \int |F|^2 \d \rho \right)^{1/2}.
\end{equation} 
Combining \eqref{eq:Sfd} with \eqref{eq:F2norm}, we find that, with probability at least $1 - (m+1) \eps$, 
\begin{align} \label{eq:SFd2}
\left| S - \int F \d \rho \right|  &	\leq 	\frac{ 4 \sqrt{\sixthree}}{m \sqrt{\eps|K|} } \log T.   
\end{align} 
Since 
\[ \int F \d \rho = \frac{1}{|K|} \int_K F(z) \d^2 z = \frac{1}{|K|} \int_{\mathbb{C}} F(z) \d^2 z, \]
the conclusion follows by scaling \eqref{eq:SFd2} by $|K|$.  
\end{proof}

We are now in a position to prove Theorem \ref{thm:non}.  

\begin{proof}[Proof of Theorem \ref{thm:non}]
By Green's formula (see, for instance, Section 2.4.1 in \cite{HKPV}), we can write
\[ \frac{1}{n} \sum_{j=1}^n \varphi(\lambda_j(M_i)) = \frac{1}{2 \pi} \int_{\mathbb{C}} \Delta \varphi(z) f_i(z) \d^2 z \]
for $i = 1,2$.  Thus, 
\[ \frac{1}{n} \sum_{j=1}^n \varphi(\lambda_j(M_1)) - \frac{1}{n} \sum_{j=1}^n \varphi(\lambda_j(M_2)) = \frac{1}{2 \pi} \int_{\mathbb{C}} F(z) \d^2 z, \]
where $F$ is defined in \eqref{def:F}.   By Lemma \ref{lemma:applymonte}, it follows that
\begin{equation} \label{eq:sumdiff}
	\left|  \frac{1}{n} \sum_{j=1}^n \varphi(\lambda_j(M_1)) - \frac{1}{n} \sum_{j=1}^n \varphi(\lambda_j(M_2))  \right| \leq \frac{|K| |S|}{2\pi} + \frac{ 4 \sqrt{|K|\sixthree}}{2\pi m \sqrt{\eps}} \log T 
\end{equation} 
with probability at least $1 - (m+1) \eps$, where
\[ S := \frac{1}{m} \sum_{j=1}^m F(Z_j) \]
and $Z_1, \ldots, Z_m$ are iid random variables, independent of $M_1$ and $M_2$, uniformly distributed on $K := \supp(\Delta \varphi)$.  Here, $|K|$ denotes the Lebesgue measure of $K$.  

By \eqref{eq:concentration} and the union bound, 
\begin{align*}
 \sup_{1 \leq j \leq m} |F(Z_j)| &=  \sup_{1 \leq j \leq m} |\Delta \varphi(Z_j)| \abs{f_1(Z_j)-f_2(Z_j) } \\
 &=   \sup_{1 \leq j \leq m} |\Delta \varphi(Z_j)| \abs{\frac{1}{n} \log | \det (Z_jI - M_1) |-\frac{1}{n} \log | \det (Z_jI - M_2) | } \\
 \details{&\le }\details{  \sup_{1 \leq j \leq m} |\Delta \varphi(Z_j)| \eta \\}
 &\leq \| \Delta \varphi \|_{\infty} \eta
\end{align*}
with probability at least $1 - m \eps$, where $\| \Delta \varphi \|_{\infty}$ is the $L^\infty$-norm of $\Delta \varphi$.  On this same event, $|S| \leq \| \Delta \varphi \|_{\infty}  \eta$.  Therefore, combining the bounds above with \eqref{eq:sumdiff}, we find

\begin{align*}
 \left|  \frac{1}{n} \sum_{j=1}^n \varphi(\lambda_j(M_1)) - \frac{1}{n} \sum_{j=1}^n \varphi(\lambda_j(M_2))  \right| 
 &\le  \frac{|K| \|\Delta \varphi\|_\infty \eta}{2\pi } +  \frac{4 \sqrt{|K|  \sixthree}}{2\pi m \sqrt{\eps}} \log T \\
 &\le  \frac{|K| \|\Delta \varphi\|_\infty \eta}{2\pi } +  \frac{4 \sqrt{|K|  \|\Delta\varphi\|_\infty^2\sixthreed}}{2\pi m \sqrt{\eps}} \log T \\
 &\le  \frac{4\max\{|K|,\sqrt{|K|}\} \|\Delta \varphi\|_\infty\sqrt{\sixthreed}}{2\pi }\left( \eta +  \frac{  \log T}{m \sqrt{\eps}}\right) 
&\mbox{(since $\sixthreed> 1$)}\\
 &\leq \Csixone \left( \eta + \frac{\log T}{m \sqrt{\eps}} \right) 
\end{align*}
with probability at least $1 - 2(m+1) \eps$, where, using the definition of $\sixthreed$ in \eqref{e:sixthreed}, we may set 
\begin{align}\label{e:ThmConst}
\Csixone
&:=\frac{2\max\left\{|K|,\sqrt{|K|}\right\} \|\Delta \varphi\|_\infty (D+1)\log(2+\sdist+\diam)}{\sqrt\pi\log(2)}
\sqrt{\max\left\{\log^2(D+1)-\log(D+1)+\frac12,1\right\}},
\end{align} 
which is a positive constant depending only on $\varphi$.  This completes the proof of Theorem \ref{thm:non}.
\end{proof}

\subsection{Details for Example~\ref{eg:smallblocks}}\label{a:s2egs}

We conclude this section with a proof of the claims made in Example~\ref{eg:smallblocks}. 
Let $\gamma \ge 4$ and $n \ge 11$, and suppose that $A$
is an $n$ by $n$ matrix in Jordan canonical form with all eigenvalues equal to zero, where each block has size at most $m:= \log(n)/t_n\ge 1$ for all $n \ge 1$ and where $t_n \to \infty$ as $n \to \infty$ and $1\le t_n \le \log n$ for all $n$.  Let $M_1$ be an $n$ by $n$ matrix with iid~random entries each having absolute value at most $n$, and let $M_2$ be the $n$ by $n$ zero matrix. Finally, let $\varphi$ be a smooth test function with compact support.

We will apply Theorem~\ref{thm:non} to show that the empirical spectral measures of $A+n^{-\gamma} M_1$ and $A+M_2$ are close to each other when integrated against a smooth test function $\varphi$.  
Because Theorem~\ref{thm:non} samples $Z$ uniformly at random from $\supp \Delta \varphi$, we can ignore cases where $|Z|$ is within $e^{-\gamma t_n/10}$ of 1 or of 0, which happens with reasonably small probability.  This lets us apply Lemma~\ref{lem:determinstic-singbd-plus-epsilon} to show that the smallest singular value of $A-ZI$ is at least $\frac12n^{-\gamma/10}$ with probability at least $1-ce^{-3\gamma t_n/10}$.
 Also, note that the difference between the two matrices is $n^{-\gamma}M_1$ which has norm at most $n^{-\gamma+2}$. Thus, we can combine the lower bound on the smallest singular value of $A-ZI$ with Weyl's inequality (see Theorem~\ref{thm:weyl}) to show that $\smin \ge 0.478 n^{-\gamma/10}$ with probability at least $1-ce^{-3\gamma t_n/10}$, where $\smin$ is the minimum of all singular values of $A+M_2-ZI$ and $A+n^{-\gamma} M_1-ZI$ and $c$ is an absolute constant. We can now apply Theorem~\ref{thm:replacement} to show that the difference in log potentials is sufficiently small, namely,

\begin{equation}\label{e:smallblocks}
 \left| \mathcal{L}_{A + M_1}(Z) - \mathcal{L}_{A + M_2}(Z) \right| \le 
 \frac{2}{\varepsilon_{\ref{thm:replacement}}}\|n^{-\gamma}M_2 \|
 \le 
 \frac{2n^{-\gamma+2}}{2n^{-\gamma/2+1}}=n^{-\gamma/2+1}
 \end{equation}
with probability 
at least $1-ce^{-3\gamma t_n/10}$ for $n$ sufficiently large.
Finally, we can apply Theorem~\ref{thm:non} (choosing parameters $T=e$, \quad $\eta = (1.01)n^{-\gamma/2+1}$, \quad $\epsilon = \frac{8\pi}{K} e^{-3\gamma t_n/10}$, and $m = \left\lceil\epsilon^{-2/3}\right\rceil$) to show that
$$\left| \int_{\mathbb C} \varphi \, d \mu_{A} - \int_{\mathbb C} \varphi \, d \mu_{A+n^{-\gamma} M_1} \right| \le C \left( n^{-\gamma/2} + e^{-\gamma t_n/20}\right)
$$
with probability at least $1-c'(e^{-\gamma t_n/10}+ e^{-3\gamma t_n/10})$, 
where $c'$  is a constant depending only on $\varphi$ and $\gamma\ge 4$ and $n\ge 11$.

\section{Proofs of Proposition~\ref{prop:leaving-the-disk} and supporting lemmas}\label{sec:deter-lemmas}

We will start with a proof of Proposition~\ref{prop:leaving-the-disk}(i), with proofs of supporting lemmas to follow.

\begin{proof}[Proof of Proposition~\ref{prop:leaving-the-disk}(i)]
We will follow the same approach as in the proof of Proposition~\ref{prop:leaving-the-disk}(ii) where we compare $\Jmat+\pone-zI$ and $\Jmat+\ptwo-zI$, but now under the constraints that $1/5 \le |z|< 1/4$ and $n \le \frac{\gamma \log \gamma - \log 3}{\log 5}$.  In this case, we are considering the situation where $\gamma$ is rather large and $n$ is rather small; thus, we can think of $n^{-\gamma}$ as being exponentially small, rather than polynomially small.

To start, we will show that the condition $n \le \frac{\gamma \log \gamma - \log 3}{\log 5}$ implies that $n^{-\gamma} < \frac{|z|^n}{3}$, an exponential upper bound since $\gamma$ is bounded below by a function of $n$.   To show that  $n^{-\gamma} < \frac13 |z|^n$ we 
consider two cases.  First, if $n \le \gamma$, then $\frac13 |z|^n \ge \frac13 |z|^\gamma > n^{-\gamma}$ so long as $n \ge 7$ (using the fact that $|z| \ge \frac15$).  Second, if 
$\gamma < n \le \frac{\gamma \log \gamma - \log 3}{\log 5}$, then $\frac13|z|^n\ge \frac13 |z|^{\frac{\gamma \log \gamma - \log 3}{\log 5}}\ge \frac13 |z|^{\frac{-\gamma \log \gamma + \log 3}{\log |z|}} = \gamma^{-\gamma}> n^{-\gamma}$.  Thus, given the conditions on $n$ and $|z|$, we have that $n^{-\gamma} < \frac13 |z|^n$.

To find a lower bound for the smallest singular value of $\Jmat+\ptwo-zI$, we will apply the second case of Lemma~\ref{lem:determinstic-singbd-plus-epsilon} with $\epsilon = n^{-\gamma}$ and $m_i=n$ and $c_i=0$, noting that $ n^{-\gamma}< \frac12|z|^n$ by the previous paragraph. Thus, we have that the smallest singular value of $\Jmat+\ptwo-zI$ is at least $\sigma_n(\Jmat+\ptwo-zI) \ge \frac{|z|^n \abs{1-|z|}^{3/2}}{\sqrt{2n |z|^{2n}+2}} > 0.45927|z|^n$, using the fact that $n \ge 7$ and $|z| < 1/4$.

To find a lower bound for the smallest singular value of $\Jmat+\pone-zI$, we may apply Lemma~\ref{lem:bdd-very-small-noise}, showing that $\sigma_n(\Jmat+\pone-zI)\ge 0.19|z|^n$.  Also, we may compute upper bounds for the smallest singular value for both $\Jmat+\pone-zI$ and $\Jmat+\ptwo-zI$ using Lemma~\ref{lem:upper-bd-small-sing} with $D=|z|^n/3$, thus showing that $\sigma_n(\Jmat+\pone-zI), \sigma_n(\Jmat+\ptwo-zI) \le 1.0011\left(|z|^n + \pfrac{|z|^n}{3}\frac{1.001}{(1-|z|)^2} \right)\le
1.6|z|^n$.
 
Weyl's perturbation theorem (Theorem~\ref{thm:weyl}) along with Lemma~\ref{lem:det-sing-bd} shows that $\sigma_{n-1} (\Jmat+\pone-zI), \sigma_{n-1}(\Jmat+\ptwo-zI) \ge 3/4-n|z|^n/3 > 0.74 $.
We may now apply Proposition~\ref{prop:small-cancel} with $\epsilon =0.74$, using the facts that $0.19|z|^n \le \sigma_n(\Jmat+\pone-zI), \sigma_n(\Jmat+\ptwo -zI) \le 1.6 |z|^n$, to get that
$$\abs{ \mathcal L_{\Jmat+\pone}(z)- \mathcal L_{\Jmat+\ptwo}(z) } \le \frac{\log\pfrac{1.9}{0.19}}{n}+\frac1{0.74}2n\pfrac{|z|^n}{3}
\details{ \le \frac1n\left( \log\pfrac{1.9}{0.19}+\frac{7(0.25)^{7}}{1.11}\right) }
< \frac{2.31}{n},
$$
using $|z|\le 1/4$ and $n \ge 7$.

We can now use the result above to apply Theorem~\ref{thm:non}.  Note that $\|\Jmat+\pone-zI\|+\|\Jmat+\ptwo-zI\| \le 2(|z|+1) +2n |z|^n/3 \le e$ whenever $|z|\le 1/4$ and $n\ge 7$.  Thus, we may take $T=e$ to satisfy the first assumption of Theorem~\ref{thm:non} with any positive $\eps$. Also, by the application of Proposition~\ref{prop:small-cancel}, in the previous paragraph, we may satisfy the second condition of Theorem~\ref{thm:non} by taking $\eta =\frac{2.31}n$, again using any positive value for $\eps$.  We also note that from the assumptions on $\varphi$, we have that  
$\Csixone \le 2\frac{\sqrt{\pi}}{4}  \|\Delta \varphi\|_\infty \pfrac32 \log_2(2.75)/\sqrt \pi <  (1.095)\|\Delta\varphi\|_\infty$.

Because the bound in Theorem~\ref{thm:non} works for any integer $m$, and because we are able to satisfy the two assumptions with any positive epsilon, we may follow Remark~\ref{rem:mlim} and set $\eps = 1/m^{3/2}$ and take the limit as $m\to\infty$, in which case $\frac{\log T}{m\sqrt\eps}\to 0$ and $2(m+1)\eps \to 0$, thus proving that
$$\frac1n\abs{ \sum_{i=1}^n \varphi(\lambda_i(\Jmat+\pone)) -\varphi(\lambda_i(\Jmat+\ptwo))} \le \Csixone \eta  < 
1.095\|\Delta \varphi\|\frac{2.31}{n} <  3 \frac{\|\Delta \varphi\|_\infty}{n}. $$ 
\end{proof}

\subsection{Supporting lemmas for singular value bounds}

In this subsection we state and prove lemmas used to prove the singular value bounds in the next subsection, where we prove Lemmas~\ref{lem:smallsing}, \ref{lem:bdd-very-small-noise}, and \ref{lem:upper-bd-small-sing}.

\begin{lemma}\label{lem:bigcoord}
Let $n$ be a positive integer, let $B$ be a positive real quantity, let $\pone=(r_{ij})$ be an $n$ by $n$ matrix, and let $\Jmat$ be an $n$ by $n$ matrix as defined in \eqref{eq:defT}.  Let $x$ be a unit vector corresponding to the smallest singular value for $\Jmat+\pone-zI$, and assume that $\abs{\sum_{j=1}^n r_{ij}x_j} \le C$ for all $1\le i \le n-1$.  Then, either the smallest singular value of $\Jmat+\pone-zI$ is at least  $B -C$, or $x$ has a the structure of an approximate geometric progression in that
 \begin{equation}\label{e:xigeoB}
-B(1+|z|+\dots+|z|^{i-2})+ \abs{x_1}|z|^{i-1} \le \abs{x_i} \le 
B(1+|z|+\dots+|z|^{i-2})+ \abs{x_1}|z|^{i-1}
\end{equation}
for all $2\le i \le n$.
\end{lemma}

\begin{proof}
Assume that there exists $1\le i \le n-1$ such a that $\abs{-zx_i + x_{i+1}} \ge B$.  Then the $i$-th coordinate of $(\Jmat+\pone-zI)$ has absolute value at least $\abs{-zx_i + x_{i+1}} -\abs{ \sum_{j=1}^n x_j r_{ij}}\ge B-C$.  This implies that the norm of $(\Jmat+\pone-zI)x$ is at least $B-C$, proving that $\Jmat+\pone-zI$ has least singular value at least $B-C$.

On the other hand, if we have $\abs{-zx_i + x_{i+1}} < B$ for all $1 \le i \le n-1$, then by the triangle inequality we have  
\begin{equation}\label{i:coordineq}
-B + |z||x_i| < |x_{i+1}| < B + |z| |x_i|,
\end{equation} for each $1 \le i \le n-1$.  We can now prove \eqref{e:xigeoB} by induction on $i$\details{.  For the base case, note that \eqref{i:coordineq} with $i=1$ shows that $-B + |z||x_1| < |x_{2}| < B + |z| |x_1|$, which is exactly \eqref{e:xigeoB} when $i=2$.  For the induction step, assume that we have shown 
$$-B(1+|z|+\dots+|z|^{i-2})+ \abs{x_1}|z|^{i-1} \le \abs{x_i} \le 
B(1+|z|+\dots+|z|^{i-2})+ \abs{x_1}|z|^{i-1}$$ 
for some $2 \le i \le n-1$.  Then, using \eqref{i:coordineq}, we have that $-B + |z||x_i| < |x_{i+1}| < B + |z| |x_i|$, which, when applying the induction hypothesis to each side, becomes
$$-B + |z|(-B(1+|z|+\dots+|z|^{i-2})+ \abs{x_1}|z|^{i-1}) \le \abs{x_{i+1}} \le+B+ 
|z|(B(1+|z|+\dots+|z|^{i-2})+ \abs{x_1}|z|^{i-1}).$$ 
This simplifies to the desired conclusion}, completing the proof.
\end{proof}

\begin{lemma}\label{lem:geom-totalnoise}
Let $B>0$ be a real number and assume that $\pone= (r_{ij})$ is an $n$ by $n$ matrix with entries uniformly bounded in absolute value by $D$.  Given a unit vector $x$ with the approximate geometric progression structure
$$-B(1+|z|+\dots+|z|^{i-2})+ \abs{x_1}|z|^{i-1} \le \abs{x_i} \le 
B(1+|z|+\dots+|z|^{i-2})+ \abs{x_1}|z|^{i-1},$$
then $\abs{ \sum_{j=1}^n r_{ij}x_j} \le \frac{D}{1-|z|}\left(1+ Bn \right)$ 
\end{lemma}

\begin{proof}
We compute, using the triangle inequality and \eqref{e:xigeoB}
\begin{align*}
\abs{\sum_{j=1}^n x_j r_{ij}} &\le D \sum_{j=1}^n\left(  |x_1| |z|^{j-1} + \frac{B}{1-|z|}\right)
\details{\\ 
&}\le \frac{D}{1-|z|}\left(|x_1|  + Bn \right) 
\le \frac{D}{1-|z|}\left(1  + Bn \right) 
\end{align*}
\end{proof}

Next, we show that $|x_1|$ is bounded below by something close to $\sqrt{1-|z|^2}$.  

\begin{lemma}\label{lem:controlx_1}
Given a unit vector $x$ that satisfies \eqref{i:coordineq} for all $1\le i \le n-1$ for some positive quantity $B$ (and thus
also \eqref{e:xigeoB}) where $nB < 1/10$ 
and for $\abs z < 1/4$, we have that $|x_1| \ge \sqrt{1-|z|^2} - B$.
\end{lemma}

\begin{proof}
Since $(x_1,\dots,x_n)$ is a unit vector, we have that
\begin{align*}
1&=\sum_{i=1}^n |x_i|^2 = |x_1|^2 + \sum_{i=1}^{n-1} |x_{i+1}|^2<|x_1|^2 + \sum_{i=1}^{n-1}\left(B+|z||x_i| \right)^2\\
&=|x_1|^2 + \sum_{i=1}^{n-1}\left(B^2+2B|z||x_i|    +|z|^2|x_i|^2 \right)\\
\details{&\le}\details{|x_1|^2 + nB^2+2B|z|\left(\sum_{i=1}^{n-1}|x_i|  \right)     +|z|^2(1-|x_n|^2)\\}
&\le|x_1|^2 + nB^2+2B|z|\left(\sum_{i=1}^{n-1}\abs{x_1}|z|^{i-1}+\frac{B}{1-|z|}  \right)  +|z|^2\\
\details{&\le}\details{|x_1|^2 + nB^2+\frac{2nB^2}{1-|z|}+ 2B|z|\left(\sum_{i=1}^{n-1}\abs{x_1}|z|^{i-1}\right)     +|z|^2\\}
&\le|x_1|^2 + nB^2+\frac{2nB^2}{1-|z|}+ \frac{2B\abs{x_1}|z|}{1-|z|}  +|z|^2\\
\details{&=}\details{|x_1|^2 +\frac{nB^2( 3-|z|)}{1-|z|}+ \frac{2B \abs{x_1}|z|}{1-|z|}     +|z|^2\\}
&\le \left(|x_1| + \frac{B|z|}{1-|z|}\right)^2 +\frac{nB^2(3-|z|)}{1-|z|}    +|z|^2.
\end{align*}
We can rearrange the last inequality to get
$$
\left(1-|z^2| - \frac{nB^2(3-|z|)}{1-|z|}  \right)^{1/2} -\frac{B|z|}{1-|z|} \le |x_1|.
$$

Using Taylor's Theorem with remainder for a function of $y$, we note that $(1-|z|^2 -y)^{1/2}>\sqrt{1-|z|^2}-y\pfrac{1}{\sqrt{1-|z|^2}}$ for $y$ satisfying $0<y< \frac{3}4(1-|z|^2)$.
Since $\frac{nB^2(3-|z|)}{1-|z|}< \frac{3}4(1-|z|^2)$ whenever $nB < 1/10$ (using the assumption that $\abs z < 1/4$), we thus have that
\begin{align*}
|x_1|&>\sqrt{1-|z|^2}-\frac{nB^2(3-|z|)}{(1-|z|)\sqrt{1-|z|^2}}  -\frac{B|z|}{1-|z|}\\
\details{&=}\details{ \sqrt{1-|z|^2}-\pfrac{B}{1-|z|}\left(  \frac{nB(3-|z|)}{\sqrt{1-|z|^2}}+|z|  \right)\\}
&\ge \sqrt{1-|z|^2}-\frac{4B}{3}\left( \frac{3nB}{\sqrt{15}/4}+\frac14  \right) &\mbox{ (since } |z|\le 1/4 \mbox{)} \\
&\ge \sqrt{1-|z|^2} - B &\mbox{ (since } nB <1/10 \mbox{)}.
\end{align*}
\end{proof}

\subsection{Proofs of lemmas bounding singular values for Proposition~\ref{prop:leaving-the-disk}}
\label{subsec:PropProof}

Lemmas~\ref{lem:smallsing}, \ref{lem:bdd-very-small-noise}, and \ref{lem:upper-bd-small-sing} are used to prove bounds on the relevant small singular values used in the proof of Proposition~\ref{prop:leaving-the-disk}.

\begin{proof}[Proof of Lemma~\ref{lem:smallsing}]
Let $(r_{ij})=\pone$ and let $x$ be a unit vector.  
 We will use Lemma~\ref{lem:bigcoord} twice is succession, each time showing that either the smallest singular value has a sufficient lower bound, or the unit vector $x$ must have a the structure of an approximate geometric progression.
Note that $\abs{\sum_{j=1}^n x_j r_{ij}} \le \left(\sum_{j=1}^nr_{ij}^2\right)^{1/2}=n^{-\gamma+1/2}$ by the Cauchy-Schwarz inequality, and so we may apply Lemma~\ref{lem:bigcoord} with $B=2n^{-\gamma+1/2}$ and $C=n^{-\gamma+1/2}$ to get that either the least singular value of $\Jmat+\pone -zI$ is at least $n^{-\gamma+1/2}$ (in which case, we are finished, since $n^{1/2}> 0.15$\details{~when $n \ge 2\gamma\log \gamma$ and $\gamma \ge 5$}), or the unit vector $x$ has the structure of an approximate geometric progression, satisfying \eqref{e:xigeoB} with $B=2n^{-\gamma + 1/2}$.  Thus, we will assume that $x$ satisfies \eqref{e:xigeoB} with $B=2n^{-\gamma+1/2}$ for the remainder of the proof.

By applying Lemma~\ref{lem:geom-totalnoise} with $B=2n^{-\gamma+1/2}$ and $D=n^{-\gamma}$, we have the improved bound of 
\begin{equation}\label{i:improved-small-noise-bd}
\abs{\sum_{j=1}^n r_{ij} x_j  }  \le \frac{n^{-\gamma}}{1-|z|} \left(1 + 2n^{-\gamma+3/2}\right) \le 
1.334 n^{-\gamma},
\end{equation}
using the assumptions that $n \ge 2\gamma \log \gamma$ and $\gamma\ge 5$.
We can now apply Lemma~\ref{lem:bigcoord} with $C=1.334n^{-\gamma}$ and $B=1.4845n^{-\gamma}$, showing that either the smallest singular value of $\Jmat+\pone-zI$ is at least $0.1505n^{-\gamma}$ (in which case we are finished), or the unit vector $x$ has the structure of an approximate geometric progression, satisfying \eqref{e:xigeoB} with $B=1.4845n^{-\gamma}$, saving a factor of $n^{1/2}$ compared to the bound in the previous paragraph. Thus, we will assume that $x$ satisfies \eqref{e:xigeoB} with $B=1.4845n^{-\gamma}$ for the remainder of the proof.

By applying Lemma~\ref{lem:geom-totalnoise} with $B=1.4845n^{-\gamma}$ and $D=n^{-\gamma}$, we have that $\abs{\sum_{j=1}^n r_{ij} x_j} \le \frac{n^{-\gamma}}{1-|z|}\left(1+2n^{-\gamma+1}\right)$.  We can now establish a useful lower bound on $|x_1|$, applying Lemma~\ref{lem:controlx_1} with $B=2n^{-\gamma}$ (note that $nB \le 2(17)^{-4} < 1/10$) to show that $|x_1| \ge \sqrt{1-|z|^2}-2n^{-\gamma}$.

In our computation below, we need to show that the contribution of a term with the form $|z|^nn^\gamma$ is very small.  This is the place in the proof where we need the assumption that $n \ge 2\gamma\log \gamma$.  In particular, we note that $|z|^n n^\gamma\le n^\gamma4^{-n}$ (since $|z|\le 1/4$) is a deceasing function of $n$, thus it is at most $(2\gamma\log\gamma)^\gamma 4^{-2\gamma\log \gamma} = (2\gamma^{1-2\log 4}\log \gamma)^\gamma$.  Because $1-2\log 4< -1.77258$, we see that this last function is decreasing in $\gamma$ when $\gamma\ge 5$, and thus we have that $|z|^nn^\gamma \le (2(5^{1-2\log 4})\log 5)^5 < 0.00023$.

We now have the ingredients for a final bound on the smallest singular value.  Note that the last coordinate of $(\Jmat+\pone-zI)x$ has absolute value at least $\abs{\sum_{j=1}^nr_{n,j}x_j} -|z x_n|$.  We may bound this quantity from below by applying \eqref{i:improved-small-noise-bd} and \eqref{e:xigeoB} (with $B= 1.4845n^{-\gamma}$) as follows:
\begin{align*}
 \abs{\sum_{j=1}^n x_jr_{nj}}  - |z||x_n| &\ge n^{-\gamma}|x_1| - n^{-\gamma} \sum_{j=2}^n |x_j|  -|z| |x_n| \\
 &\ge  n^{-\gamma}|x_1| - n^{-\gamma} \sum_{j=2}^n \left(\frac{1.4845n^{-\gamma}}{1-|z|} +|x_1||z|^{j-1}   \right) -|z| \left(\frac{1.4845n^{-\gamma}}{1-|z|} +|z|^{n-1}|x_1|\right) \\
 &\ge  n^{-\gamma}|x_1| - \frac{1.4845n^{-2\gamma+1}}{1-|z|} -\frac{n^{-\gamma}|x_1||z|}{1-|z|}  -\frac{1.4845 n^{-\gamma}|z| }{1-|z|} -|z|^{n}|x_1| \\
\details{ &\ge}\details{ n^{-\gamma}\left(\frac{|x_1|(1-2|z|)}{1-|z|} -\frac{1.4845 |z| }{1-|z|}- \frac{1.4845n^{-\gamma+1}}{1-|z|}   -|z|^{n}|x_1|n^{\gamma}  \right) \\}
\details{&\ge}\details{ n^{-\gamma}\left(\frac{2|x_1|}{3} -\frac{1.4845}{3}- \frac{1.4845n^{-\gamma+1}}{1-|z|}   -|z|^{n}|x_1|n^{\gamma}\right) \\}
 &> 0.15 n^{-\gamma} ,
\end{align*}
where the last inequality uses the facts that   $|z|\le 1/4$ and $\gamma\ge 5$ and $n \ge 2\gamma\log \gamma$ along with  $|x_1|\ge \sqrt{1-|z|^2}-2n^{-\gamma} > 0.968$ and $|z|^n n^{\gamma} < 0.00023$ from the previous paragraphs.

Thus $\|(\Jmat + \pone-zI)x\|> 0.15n^{-\gamma}$.  Because $x$ was an arbitrary unit vector, we have proven that the smallest singular value of $\Jmat +\pone-zI$ is at least $0.15n^{-\gamma}$. 
\end{proof}

The proof of Lemma~\ref{lem:bdd-very-small-noise} follows a similar approach to Lemma~\ref{lem:smallsing} above.

\begin{proof}[Proof of Lemma~\ref{lem:bdd-very-small-noise}]
Let $x$ be a unit vector corresponding to the smallest singular value for $\Jmat+\pone-zI$. Note that for a unit vector $x$, we can apply the Cauchy-Schwarz inequality to get that $\abs{\sum_{j=1}^n r_{ij}x_j } \le \sqrt{n}|z|^n/3$.  We can now apply Lemma~\ref{lem:bigcoord} with $B=\sqrt n|z|^n/2$ and $C=\sqrt n|z|^n/3$ to get that either the least singular value of $\Jmat+\pone-zI$ is at least $\sqrt n|z|^n/6$ (in which case we are finished, since $\sqrt{n}/6 > 0.19$ for $n\ge 7$), or the unit vector $x$ has the structure of an approximate geometric progression, satisfying \eqref{e:xigeoB} with $B=\sqrt n|z|^n/2$.  Thus, we will assume that $x$ satisfies \eqref{e:xigeoB} with $B=\sqrt n |z|^n/2$ for the remainder of the proof.

By applying Lemma~\ref{lem:geom-totalnoise} with $B=\sqrt n|z|^n/2$ and $D= |z|^n/3$, we have the improved bound of 
\begin{align}\label{i:improved-vsmall-noise-bd}
\abs{\sum_{j=1}^n r_{ij}x_j } &\le \frac{|z|^n}{3(1-|z|)}\left( 1+ \frac{n^{3/2} |z|^n}{2}\right) \le \frac{|z|^n}{2},
\end{align}
using the assumptions that $n \ge 7$ and $|z|\le 1/4$.  We can now apply Lemma~\ref{lem:bigcoord} with $C= |z|^n/2$ and $B= |z|^n$, showing that either the smallest singular value of $\Jmat+\pone-zI$ is at least $|z|^n/2$ (in which case we are again finished), or  the unit vector $x$ has the structure of an approximate geometric progression, satisfying \eqref{e:xigeoB} with $B=|z|^n$.  Thus, we will assume that $x$ satisfies \eqref{e:xigeoB} with $B=|z|^n$ for the remainder of the proof.

By applying Lemma~\ref{lem:geom-totalnoise} with $B= |z|^n$ and $D=|z|^n/3$, we have that $\abs{\sum_{j=1}^n r_{ij}x_j} \le \frac{|z|^n}{3(1-|z|)}\left(1+ n|z|^n\right)$.  We can now establish a useful lower bound on $|x_1|$, applying Lemma~\ref{lem:controlx_1} with $B=|z|^n$ (note $nB \le 5(1/4)^n < 1/10$) to show that $|x_1| \ge \sqrt{1-|z|^2} - |z|^n$.

We now have all the ingredients for a final bound on the smallest singular value. Note that the last coordinate of $(\Jmat+\pone-zI)x$ has absolute value at least $|zx_n| - \abs{\sum_{j=1}^n r_{nj}x_j }$.  We may bound this quantity below by applying \eqref{i:improved-vsmall-noise-bd} and \eqref{e:xigeoB}  (with $B=|z|^n$) as follows:
\begin{align*}
|z||x_n| - \abs{\sum_{j=1}^n r_{nj}x_j } &\ge
|z|\left( |x_1| |z|^{n-1} - \frac{|z|^n}{1-|z|}\right) - \frac{|z|^n}{3(1-|z|)} \left( 1 + n|z|^n\right) \\
\details{&\ge}\details{ \sqrt{1-|z|^2}|z|^n - |z|^{2n} -\frac{|z|^{n+1}}{1-|z|}
-\frac{|z|^{n}}{3(1-|z|)}-\frac{n|z|^{2n}}{3(1-|z|)} \\
&\ge }\details{|z|^n \left(\sqrt{1-|z|^2}  -\frac{|z|}{1-|z|}
-\frac{1}{3(1-|z|)}- |z|^{n}-\frac{n|z|^{n}}{3(1-|z|)}\right) \\
&\ge}\details{ |z|^n\left(
\frac{\sqrt{15}}{4} - \frac 43 \pfrac 7{12}- |z|^n\left(1+\frac{n}{3(1-|z|)}\right)\right)\\
&\ge}\details{ |z|^n \left(\frac{\sqrt{15}}{4} - \frac 7{9} - \pfrac{1}{4}^n\left(1 + \frac{35}{9}\right)\right)\\}
&\ge 0.19|z|^n,
\end{align*}
using the facts that $n \ge 7$ and $|z|\le 1/4$.

Thus, we have shown that the smallest singular value for $\Jmat+\pone-zI$ is at least $0.19 |z|^n$.
\end{proof}

\begin{proof}[Proof of Lemma~\ref{lem:upper-bd-small-sing}]
We will construct a constant-length vector $w$ to demonstrate the upper bound on the smallest singular value of the form $\|(\Jmat+\pone-zI)w\|/\|w\|$.  Let $v=(1, z, z^2, \dots, z^{n-1})$, and note that $(\Jmat-zI)v = (0, \dots, 0, z^n)$.  The vector $v$ itself does not give a sufficiently small upper bound on the smallest singular value, since $(\Jmat+\pone-zI) = (0,\dots, 0, z^{n-1}) + \pone v$, and $\|\pone v\|$ may be as large as $\Theta( D\sqrt n)$.  However, we will demonstrate that we can perturb $v$ slightly so that nearly all of the contribution of $\pone v$ is canceled out.  In particular, let $x$ be the vector
$$x= \sum_{k=2}^n e_k  \sum_{i=1}^{k-1} (\pone_i \cdot v) z^{k-1-i} ,$$
where $e_k$ is the standard basis vector and $\pone_i$ is the $i$-th row of the matrix $\pone$, and then define $w= v-x$.  We may now compute
\begin{align*}
(\Jmat-zI + \pone) w &= (0,\dots, 0, z^n) + \pone v - \sum_{k=2}^n e_{k-1} \sum_{i=1}^{k-1} (\pone_i\cdot v)z^{k-1-i}
+ \sum_{k=2}^n ze_k \sum_{i=1}^{k-1} (\pone_i\cdot v) z^{k-1-i} - \pone x\\
\details{&=}\details{ (0,\dots, 0, z^n) + \sum_{k=1}^n (\pone_k\cdot v)e_k - \sum_{k=1}^{n-1} e_k (\pone_k\cdot v)
-\sum_{k=3}^{n} e_{k-1} \sum_{i=1}^{k-2} (\pone_i\cdot v)z^{k-1-i}\\
&\qquad}\details{+ \sum_{k=2}^n e_k \sum_{i=1}^{k-1} (\pone_i\cdot v) z^{k-i} - \pone x\\}
&= (0,\dots, 0, z^n) + (\pone_n\cdot v)e_n -\sum_{k=2}^{n-1} e_{k} \sum_{i=1}^{k-1} (\pone_i\cdot v)z^{k-i}
+ \sum_{k=2}^n e_k \sum_{i=1}^{k-1} (\pone_i\cdot v) z^{k-i} - \pone x\\
\details{&=}\details{ (0,\dots, 0, z^n) + (\pone_n\cdot v)e_n + e_n \sum_{i=1}^{n-1} (\pone_i\cdot v) z^{n-i} - \pone x\\}
&= (0,\dots, 0, z^n) + e_n \sum_{i=1}^{n} (\pone_i\cdot v) z^{n-i} - \pone x.
\end{align*}
We will now use the above along with the triangle inequality to produce an upper bound on $\|(\Jmat-zI +\pone)w\|$ by bounding each of the three terms above.  Note that $\|(0,\dots, 0,z^n)\| \le |z|^n$ and that $\left\|e_n \sum_{i=1}^n (\pone_i\cdot v) z^{n-i}\right\| \le D/(1-|z|)^2$.  It remains to bound $\|\pone x\|$.  Note that each coordinate of $x$ has absolute value $\abs{ \sum_{i=1}^{k-1} (\pone_i\cdot v) z^{k-1-i}} \le D/(1-|z|)^2$.  Thus $\|\pone x\| \le \frac{D^2}{(1-|z|)^2} n^{3/2}$.  Putting the three bounds together shows that $\|(\Jmat+\pone-zI)w\| \le |z|^n+ \pfrac{1.001D}{(1-|z|)^2}$.  Finally, since $\|w\|\ge \|v\|-\|x\| \ge 1-\frac{D\sqrt n}{(1-|z|)^2}\ge 1-\frac{0.001}{n(1-|z|)^2} \ge 0.999$, we have shown that
$$\frac{\|(\Jmat+\pone-zI)w\|}{\|w\|} \le \frac{1}{0.999}\left(|z|^n + \frac{1.001D}{(1-|z|)^2}\right) =1.0011\left(|z|^n + \frac{1.001D}{(1-|z|)^2}\right),$$
which is the desired upper bound on the smallest singular value.
\end{proof}

\section{Singular value computations for a Jordan block matrix} \label{sec:jordan}

Let $M_n$ be an $n$ by $n$ matrix composed of Jordan blocks $B_i$ for $i=1,\dots, \ell_n$.  Each block $B_i$ is an $m_i$ by $m_i$ matrix with eigenvalue $c_i$ on the diagonal, ones on the first superdiagonal, and all other entries equal to zero; thus,

$$B_i= \begin{pmatrix}
       c_i &  1 & 0 & \dots  & 0 \\
        0 & c_i & 1 &   & \vdots \\
        \vdots & \ddots & \ddots& \ddots &  0 \\
         &  &  0 & c_i & 1\\
        0 & \dots & & 0 &c_i
       \end{pmatrix}
$$
Let $\ptwo_i$ be an $m_i$ by $m_i$ matrix with the $(m_i,1)$ entry equal to $\epsilon_i$ (which may be a function of $m_i$, $c_i$, and $z$) and all other entries equal to zero.

\begin{lemma}\label{lem:det-sing-bd}
Let $i$ and $m_i$ be positive integers, let $z$ and $c_i$ be complex constants, and let $B_i$ be an $m_i$ by $m_i$ Jordan block with eigenvalue $c_i$.  The matrix $A=B_i-zI$ has $m_i-1$ singular values bounded below by $||c_i-z|-1|$ and above by $|c_i-z|+1$. 
Furthermore, if $M$ is an $m_i$ by $m_i$ matrix with the first $m_i-1$ rows equal to the first $m_i-1$ rows of $A$ (including the case $M=A$), then every singular value except the smallest satisfies $\sigma_k(M) \ge \abs{|c_i-z|-1}$ (for $1\le k\le m_i-1$).  
If $\abs{c_i-z}<1$, then $A$ has smallest singular value of size at most $|c_i-z|^n$, and if $\abs{c_i-z}>1$, all singular values of $A$ have  size at least $\min\{  ||c_i-z|-1|, \sqrt{|c_i-z|^2-|c_i-z|}\}$.
\end{lemma}

\begin{proof}
First, we show that the largest $n-1$ singular values of $B_i-zI$ have size $\Theta(\abs{c_i-z})$.  If $A$ is an $m$ by $m$ matrix and $A'$ is an $m-1$ by $m$ matrix consisting of the first $m-1$ rows of $A$, then Cauchy's interlacing law (see \cite[Lemma~A.1]{tao_random_2010-1}) states that for any $1 \le k \le m-1$, we have $\sigma_k(A) \ge \sigma_k(A') \ge \sigma_{k+1}(A)$, where $\sigma_1(A) \ge \sigma_2(A) \ge \dots \ge \sigma_m(A)$ are the singular values for $A$, and similarly for $A'$.  If $M$ shares the first $m-1$ rows with $A$, then we will similarly have $\sigma_k(M) \ge \sigma_k(A') \ge \sigma_{k+1}(M)$ for any $1 \le k \le m-1$.
If we set $A=B_i-zI$,  the singular values of $A'$ are the square-roots of the eigenvalues of $A' (A')^*$, which is a tridiagonal Topelitz matrix with $\abs {c_i-z}^2+1$ on the diagonal, $\overline{c_i}-\overline{z}$ on the superdiagonal, and $c_i-z$ on the subdiagonal.  The eigenvalues of a tridiagonal Toeplitz matrix can be computed explicitly (see, for example \cite{noschese_tridiagonal_2013}), and for the matrix $A'(A')^*$ are given by $\lambda_k=\lambda_k(A'(A')^*)= \abs{c_i-z}^2+1+2\abs{c_i-z} \cos\pfrac{k \pi}{m_i}$ for $k=1,2,\dots,m_i-1$.  Using the fact that $ -1\le \cos(\theta) \le 1$ for all real $\theta$, we have that $(\abs{c_i-z}-1)^2 \le \lambda_k \le (\abs{c_i-z}+1)^2$.  Thus, by Cauchy's interlacing law, we see for $2 \le k \le m_i-1$ that 
$\abs{\abs{c_i-z}-1}\le \sigma_k(A), \sigma_k(M) \le \abs{\abs{c_i-z}+1}$, showing that for $2\le k \le m_i-1$, the singular values are bounded above and below by constants depending only on $c_i$ and $z$.  For $k=1$, the explicit formula for the eigenvalues for $A'(A')^*$ combined with Taylor's theorem shows that $\sigma_1^2(A)$ is bounded below by $(\abs{c_i-z}+1)^2 - \frac{|c_i-z|\pi^2}{m_i^2}$.  We can bound $\sigma_1(A)$ from above by noting that for a unit column vector $v=\transpose{(v_1,\dots, v_{m_i})}$, we have $\norm{Av}^2= \abs{(c_i-z)v_{m_i}}^2+ \sum_{k=1}^{m_i-1} \abs{v_k(c_i-z) + v_{k+1}}^2$; thus, using the Cauchy--Schwarz inequality and the fact that $v$ is a unit vector, we see that
\begin{align*}
\norm{Av}^2&\le \abs{c_i-z}^2\abs{v_{m_i}}^2+ \sum_{k=1}^{m_i-1} \abs{v_k}^2\abs{c_i-z}^2 
+ 2\sum_{k=1}^{m_i-1} \abs{c_i-z}\abs{v_k}\abs{v_{k+1}} +\sum_{k=1}^{m_i-1}  \abs{ v_{k+1}}^2 \\
&\le \abs{c_i-z}^2+2\abs{c_i-z}+1 = (\abs{c_i-z}+1)^2,
\end{align*}
which proves that $\sigma_1(A) \le \abs{c_i-z}+1$.
We have now shown for $k=1,2,\dots,m_i-1$ that $\sigma_k(A)$ is bounded above by $|c_i-z|+1$ and bounded below by $||c_i-z|-1|$.

It remains to show an upper bound on the smallest singular value $\sigma_n(A)$, which is easily done by noting that the column vector $u=\transpose{(1, (z-c_i), (z-c_i)^2, \dots, (z-c_i)^{m_i-1})}$ satisfies $(B_i-zI)u = (0,0,\dots,0,-(z-c_i)^{m_i})$, and so $\frac{\norm{(B_i-zI)u}^2}{\norm{u}^2} = \frac{\abs{c_i-z}^{2m_i}(1-\abs{c_i-z}^2)}{1-\abs{c_i-z}^{2m_i}} \le \abs{c_i-z}^{2m_i}$, when $\abs{c_i-z} <1$.  Thus, $\sigma_{m_i}(A) \le \abs{c_i-z}^{m_i}$.

The final assertion for the $\abs{c_i-z}>1$ case can be proven by noting that the matrix $AA^*$ has all diagonal entries equal to $|c_i-z|^2+1$ or $|c_i-z|^2$, has all super diagonal entries equal to $\bar{c_i}-\bar{z}$,  has all subdiagonal entries equal to $c_i-z$, and has all other entries equal to zero.  By the Gershgorin Circle Theorem, the eigenvalues of $AA^*$ are all contained in the disk with center $|c_i-z|^2+1$ and radius $2|c_i-z|$ or in the disk with center $|c_i-z|^2$ and radius $|c_i-z|$.  Thus, every eigenvalue of $AA^*$ has size at least $\min\{ (|c_i-z|-1)^2, |c_i-z|^2-|c_i-z|\}$, which implies that the smallest singular value of $A$ is at least $\min\{  ||c_i-z|-1|, \sqrt{|c_i-z|^2-|c_i-z|}\}$, which is positive  due to the assumption that $|c_i-z|>1$.
\end{proof}

\newcommand\singbd{\sigma_n(A)}

\begin{lemma}\label{lem:determinstic-singbd-plus-epsilon}
Let $i$ and $m_i$ be positive integers, let $z$ and $c_i$ be a complex numbers (which may depend on $m_i)$, let $\epsilon$ be a non-negative real number (possibly depending on $m_i,c_i$, and $z$), let $\ptwo_i$ be the $m_i$ by $m_i$ matrix with entry $(m_i,1)$ equal to $\epsilon$ and all other entries equal to zero, and let $B_i$ be an $m_i$ by $m_i$ Jordan block with eigenvalue $c_i$.   
Then the smallest singular value $\singbd$ for $A=B_i-zI+\ptwo$ satisfies
$$
\singbd \ge
\begin{cases}
 \frac{\epsilon(1-|z-c_i|)^{3/2}}{\sqrt{2m_i\epsilon^2 +8}}  
&\mbox{ if $\abs{z-c_i} <1$ and $2\abs{z-c_i}^{m_i} < \epsilon$,} \\[10pt]
\frac{\abs{z-c_i}^{m_i}(1-|z-c_i|)^{3/2}}{\sqrt{2m_i\abs{z-c_i}^{2m_i}+2}}
&\mbox{ if $\abs{z-c_i} <1$ and $\epsilon< \frac12\abs{z-c_i}^{m_i} $,} \\[10pt]
\frac{ | |z-c_i|-1 | \sqrt{ |z-c_i|^2-1}}{\sqrt{ 2 m_i (|z-c_i|^2-1)+8\epsilon^2}} 
&\mbox{ if $\abs{z-c_i}>1$ and $\epsilon < \frac12\abs{z-c_i}^{m_i}$.} 
\end{cases}
$$
\end{lemma}

\begin{proof}
Let $A=B_i-zI+\ptwo$, and assume that $A$ has smallest singular value $\singbd$.  Thus, there exists a unit column vector $x=\transpose{(x_1,\dots,x_n)}$ such that $Ax = y$ and $\norm{Ax}=\norm y = \singbd$. 

For each case, we will show bounds on the $\abs{x_k}$ that then translate into a bound on $\singbd$.

We will consider the first two cases together, assuming that 
that $\abs{c_i-z} < 1$ and highlighting in the proof where we use the assumption $2\abs{z-c_i}^{m_i}<\epsilon$ for the first case, or, respectively, $\epsilon< \frac12 \abs{z-c_i}^{m_i}$ for the second case.  The equation $Ax=y$ may be written as the system of equations
\begin{align*}
x_2 &= y_1 + (z-c_i)x_1 \\
x_3 &= y_2 + (z-c_i)x_2 \\
&\ \vdots \\
x_{m_i} &= y_{m_i-1} + (z-c_i)x_{m_i-1} \\
\epsilon x_1 &= y_{m_i} + (z-c_i)x_{m_i}.
\end{align*}	
Successively plugging in the equations for $x_2,x_3, x_4,\dots, x_n$, we may write each $x_k$ in terms of the $y_k$ and $x_1$ as follows:
\begin{align*}
x_2 &= y_1 + (z-c_i)x_1 \\
x_3 &= y_2 + (z-c_i)y_1 + (z-c_i)^2x_1 \\
\details{x_4 &=}\details{ y_3 + (z-c_i)y_2 + (z-c_i)^2 y_1+ (z-c_i)^3x_1 \\}
&\ \vdots \\
x_{m_i} &= \left(\sum_{\ell=0}^{m_i-2} y_{m_i-1-\ell}(z-c_i)^\ell\right)  + (z-c_i)^{m_i-1}x_{1} \\
\epsilon x_1 &= \left(\sum_{\ell=0}^{m_i-1} y_{m_i-\ell}(z-c_i)^\ell\right) + (z-c_i)^{m_i} x_{1}.
\end{align*}

Using the fact that $\abs{y_k} \le \singbd$ for all $1\le k\le m_i$ and the assumptions that $\abs{z-c_i} <1$ and that $\epsilon \ne \abs{z-c_i}^{m_i}$ (the latter of which follows from either the first or the second case assumption), we may solve the last equation for $x_1$ and take absolute values to arrive at

\begin{equation}\label{e:xbound-case<1}
\abs{x_1} \le \frac{\singbd}{(1-\abs{z-c_i})\abs{\epsilon-(z-c_i)^{m_i}}}.
\end{equation}

By assumption $\norm x = 1$, and so we have
\begin{align*}
 1 &= \norm {x}^2 = \sum_{k=1}^{m_i} \abs{x_k}^2 \details{\\
 &=}\details{ \abs{x_1}^2  + \sum_{k=2}^{m_i} \abs{\left(\sum_{\ell=0}^{k-2} y_{k-1-\ell}(z-c_i)^\ell\right)  + (z-c_i)^{k-1}x_{1} }^2\\
 &}
 \le \abs{x_1}^2  + \sum_{k=2}^{m_i} \left(\left(\singbd\sum_{\ell=0}^{k-2}\abs{z-c_i}^\ell\right)  + \abs{z-c_i}^{k-1}\abs{x_{1}} \right)^2\\
 &\le \abs{x_1}^2  + \sum_{k=2}^{m_i} \left(\frac{2 \singbd^2}{(1-\abs{z-c_i})^2}  + 2\abs{z-c_i}^{2(k-1)}\abs{x_{1}}^2 \right)\\
 & \le  \frac{2m_i \singbd^2}{(1-\abs{z-c_i})^2}  + \frac{2\abs{x_{1}}^2}{1-\abs{z-c_i}^2} \\
 \details{& \le}\details{  \frac{2m_i \singbd^2}{(1-\abs{z-c_i})^2}  + \frac{2s^2 }{\abs{\epsilon-(z-c_i)^{m_i}}^2(1-\abs{z-c_i})^2(1-\abs{z-c_i}^2)} \\
 & \le}\details{  \frac{2m_i \singbd^2}{(1-\abs{z-c_i})^3}  + \frac{2s^2 }{\abs{\epsilon-(z-c_i)^{m_i}}^2(1-\abs{z-c_i})^3} \\}
 & \le  \frac{2 \singbd^2}{(1-\abs{z-c_i})^3}\left(m_i + \frac{1 }{\abs{\epsilon-(z-c_i)^{m_i}}^2} \right),
\end{align*}
where we plugged in \eqref{e:xbound-case<1} in the last inequality.  
Thus, we have shown that
\begin{equation}\label{e:case1and2bound}
1 \le \frac{2 \singbd^2}{(1-\abs{z-c_i})^3}\left(m_i + \frac{1 }{\abs{\epsilon-(z-c_i)^{m_i}}^2} \right)
\end{equation}
Under the assumption that $2\abs{z-c_i}^{m_i} < \epsilon$ from the first case, \eqref{e:case1and2bound} implies 

\begin{equation}\label{e:singbd-case<1}
 \singbd  \ge \frac{\epsilon (1-|z-c_i|)^{3/2}}{\sqrt{2m_i\epsilon^2 +8}}
\end{equation}
which proves the first case of the lemma.

Under the assumption that $\epsilon < \frac12\abs{z-c_i}^{m_i}$ from the second case, \eqref{e:case1and2bound} 
can be rearranged to show 

\begin{equation}\label{e:singbd-case<1_small_eps}
 \singbd \ge \frac{\abs{z-c_i}^{m_i}( 1-|z-c_i|)^{3/2}}{\sqrt{2m_i\abs{z-c_i}^{2m_i}+8}}
\end{equation}
which proves the second case of the lemma.

For the case where $\abs{z-c_i} >1$ and $\epsilon < \frac12\abs{z-c_i}^{m_i}$, we proceed similarly.  From $Ax=y$, we have 
\begin{align*}
x_1 & = \frac{-y_1}{z-c_i}+ \frac{x_2}{z-c_i}\\		
x_2 &= \frac{-y_2}{z-c_i}  + \frac{x_3}{z-c_i}\\
&\ \vdots \\
x_{m_i-1} &= \frac{-y_{m_i-1}}{z-c_i} +  \frac{x_{m_i}}{z-c_i}\\
x_{m_i} &= \frac{-y_{m_i}}{z-c_i} +  \frac{\epsilon x_{1}}{z-c_i}\\
\end{align*}	

Substituting in for $x_{m_i}, x_{m_i-1},\dots, x_2$, we may write the right-hand side of each equation in terms of $y_k$, $x_1$, $z$, and $c_i$ as follows:
\begin{align*}
x_{m_i} &= \frac{-y_{m_i}}{z-c_i} +  \frac{\epsilon x_{1}}{(z-c_i)^{1}}\\
\details{x_{m_i-1}&=}\details{  \frac{-y_{m_i-1}}{z-c_i} + \frac{-y_{m_i}}{(z-c_i)^2} +  \frac{\epsilon x_{1}}{(z-c_i)^{2}}  \\}
&\ \vdots \\
x_k &= \left(\sum_{\ell=k}^{m_i} \frac{-y_\ell}{(z-c_i)^{\ell+1-k}} \right) + \frac{\epsilon x_{1}}{(z-c_i)^{m_i+1-k}}  \\
&\ \vdots \\
x_1 &= \left(\sum_{\ell=1}^{m_i} \frac{-y_\ell}{(z-c_i)^{\ell}} \right) + \frac{\epsilon x_{1}}{(z-c_i)^{m_i}}.
\end{align*}
From the last equation and the fact that $\abs{y_k}\le \singbd$ for all $k$ and the assumption that $\epsilon < \frac12\abs{z-c_i}^{m_i}$, we have
\begin{equation}\label{e:xboundcase>1}
 \abs{x_1} \le \frac{\singbd}{1 - \frac{\epsilon }{\abs{z-c_i}^{m_i}}} \sum_{k=1}^{m_i}\frac1{\abs{z-c_i}^{k}} 
 \le \frac{2\singbd\pfrac{1}{\abs{z-c_i}}}{1-\frac{1}{\abs{z-c_i}}}
 \le \frac{2\singbd}{\abs{z-c_i}-1}.
\end{equation}

As before, we now compute a bound on $\singbd$ in terms of $\abs{x_1}$.  By assumption $\norm x = 1$, and so we have
\begin{align*}
 1 &= \norm {x}^2 = \sum_{k=1}^{m_i} \abs{x_k}^2 \\
\details{ &=}\details{ \sum_{k=1}^{m_i} \abs{ \left(\sum_{\ell = k}^{m_i} \frac{-y_\ell}{(z-c_i)^{\ell+1-k}} \right) + \frac{\epsilon x_{1}}{(z-c_i)^{m_i+1-k}} }^2 \\}
 &\le \sum_{k=1}^{m_i} \left(\frac{\singbd }{\abs{z-c_i}} \left(\sum_{\ell = k}^{m_i} \frac{1}{\abs{z-c_i}^{\ell-k}} \right) + \frac{\epsilon\abs{ x_{1}}}{\abs{z-c_i}^{m_i+1-k}} \right)^2 \\
\details{ &\le}\details{ \sum_{k=1}^{m_i} \left(\frac{\singbd }{\abs{z-c_i}} \left(\frac{1}{1-\frac{1}{\abs{z-c_i}}}\right) 
 + \frac{\epsilon\abs{ x_{1}}}{\abs{z-c_i}^{m_i+1-k}} \right)^2 \\}
&\le \sum_{k=1}^{m_i} \left[ \frac{2\singbd^2}{(\abs{z-c_i}-1)^2} + \frac{2\epsilon^2\abs{ x_{1}}^2}{\abs{z-c_i}^{2(m_i+1-k)}} \right] \\
\details{&\le}\details{ \frac{2\singbd^2m_i}{(\abs{z-c_i}-1)^2} + \frac{2\epsilon^2\abs{ x_{1}}^2}{\abs{z-c_i}^{2}-1} \\}
&\le \frac{2\singbd^2m_i}{(\abs{z-c_i}-1)^2} + \frac{8\epsilon^2\singbd^2}{(\abs{z-c_i}^{2}-1)(\abs{z-c_i}-1)^2}, \\
\end{align*}
where we plugged in the bound from \eqref{e:xboundcase>1} for the last inequality.  Rearranging this last inequality, we have
\begin{equation}\label{e:singbd-case>1}
 \singbd \ge \frac{ | |z-c_i|-1 | \sqrt{ |z-c_i|^2-1}}{\sqrt{ 2 m_i (|z-c_i|^2-1)+8\epsilon^2}},
\end{equation}
which proves the third case of the lemma.

Combining the inequalities in \eqref{e:singbd-case<1}, \eqref{e:singbd-case<1_small_eps}, and \eqref{e:singbd-case>1}
completes the proof.
\end{proof}

\section{Multiplicative perturbations} \label{sec:multi}

This section contains some relevant bounds required in the proof of Corollary \ref{cor:multi}.  

\begin{proposition}[Spectral norm bound] \label{prop:norm}
Let $E$ be an $n \times n$ random matrix whose entries are iid copies of a random variable $\xi$ which has mean zero, unit variance, and finite fourth moment.  Then there exists $C> 0$ (depending only on $\xi$) so that, for any $\alpha > 1/2$, 
\[ \Prob( \|E\| \geq n^{\alpha} ) \leq C n^{1/2 - \alpha}.  \]
\end{proposition}
\begin{proof}
It follows from \cite[Theorem 2]{MR2111932} that $\E \|E\| \leq C \sqrt{n}$ for a constant $C > 0$ depending only on $\xi$.  The claim now follows from Markov's inequality.
\end{proof}

Our next result bounds the spectral norm of a banded Toeplitz matrix (see Definition \ref{def:toep}).  

\begin{proposition}[Spectral norm bound] \label{prop:Anorm}
Let $k \geq 0$ be an integer, and let $\{a_j\}_{j \in \mathbb{Z}}$ be a sequence of complex numbers indexed by the integers.  Let $A$ be the $n \times n$ Toeplitz matrix with symbol $\{a_j\}_{j \in \mathbb{Z}}$ truncated at $k$ (as in Definition \ref{def:toep}).  Then
\[ \|A\| \leq \sum_{|j| \leq k} |a_j|. \]
\end{proposition}
\begin{proof}
We decompose $A$ as 
\[ A = \sum_{j = 0}^k a_j J^j + \sum_{j=1}^k a_{-j} (J^{\mathrm{T}})^j, \]
where $J$ is the $n \times n$ matrix whose entries are all zero except for ones along the sub-diagonal (i.e., $J$ is the transpose of the matrix $\Jmat$ given in \eqref{eq:defT}) and $J^0$ is the identity matrix.  The claimed bound then follows from the triangle inequality since $\|J\| = 1$.  
\end{proof}

Lastly, we will need the following least singular value bound.  

\begin{proposition}[Least singular value bound] \label{prop:lsv}
Let $A$ be a deterministic $n \times n$ matrix, and let $E$ be an $n \times n$ random matrix whose entries are iid copies of a real standard normal random variable.  Then for any $z \in \mathbb{C}$ with $z \neq 0$ and any $\gamma > 1$, there exists $\kappa > 0$ (depending on $z$ and $\gamma$) so that
\[ \Prob( \sigma_{\min}(A (I + n^{-1/2 - \gamma}E) - z I) \leq n^{-\kappa}) = o(1). \]
\end{proposition} 
\begin{proof}
Fix $z \in \mathbb{C}$ with $z \neq 0$ and $\gamma > 1$.  
It follows from \cite[Corollary 5.35]{MR2963170} that $\|E\| = O(\sqrt{n})$ with probability $1 - o(1)$.  Thus, with probability $1 - o(1)$, the matrix $I + n^{-1/2-\gamma} E$ is invertible, and by utilizing a Neumann series
\begin{equation} \label{eq:invtbnd}
	\| (I + n^{-1/2 - \gamma} E)^{-1} \| \leq 2 
\end{equation} 
and
\begin{equation} \label{eq:neumann}
	\| (I + n^{-1/2 - \gamma} E)^{-1} - (I - n^{-1/2 - \gamma} E) \| = O(n^{-2\gamma}) 
\end{equation} 
with probability $1 - o(1)$. 
Therefore---using the bound $\sigma_{\min}(M_1 M_2) \geq \sigma_{\min}(M_1) \sigma_{\min}(M_2)$ for two $n \times n$ matrices $M_1$ and $M_2$ (which follows from the characterization given in \eqref{eq:varminsing})---we obtain 
\begin{align}
	\sigma_{\min}(A (I + n^{-1/2 - \gamma}E) - z I) &\geq \sigma_{\min}(I + n^{-1/2 - \gamma}E) \sigma_{\min} ( A - z (I + n^{-1/2 - \gamma}E)^{-1} ) \nonumber \\
	&\geq \frac{1}{2} \sigma_{\min} ( A - z (I + n^{-1/2 - \gamma}E)^{-1} )  \label{eq:sigmaminlower}
\end{align}
by \eqref{eq:invtbnd}.  By Weyl's perturbation theorem (Theorem \ref{thm:weyl}) and \eqref{eq:neumann}
\begin{equation} \label{eq:weylgammaE}
	\sigma_{\min} ( A - z (I + n^{-1/2 - \gamma}E)^{-1} ) \geq \sigma_{\min}(A - zI - z n^{-1/2-\gamma} E) - O(n^{-2\gamma}) 
\end{equation} 
with probability $1 - o(1)$.  By a well-known result of Sankar, Spielman, and Teng \cite{MR2255338} for the least singular value, we obtain, for any $\eps > 0$, 
\[ \sigma_{\min}(A - zI - z n^{-1/2-\gamma} E) \geq |z| n^{-1 - \gamma - \eps} \]
with probability $1 - o(1)$.  Using the fact that $\gamma > 1$, we can take $\eps > 0$ sufficiently small so that the bound above and \eqref{eq:weylgammaE} imply 
\[ \sigma_{\min} ( A - z (I + n^{-1/2 - \gamma}E)^{-1} ) \geq \frac{|z|}{2} n^{- 1 - \gamma - \eps} \]
with probability $1 - o(1)$.  Combined with \eqref{eq:sigmaminlower}, we conclude that
\[ \sigma_{\min}(A (I + n^{-1/2 - \gamma}E) - z I) > n^{-\kappa} \]
with probability $1 - o(1)$ for a sufficiently large choice of $\kappa > 0$.  
\end{proof}


\bibliographystyle{abbrv}

\bibliography{replacement} 


\end{document}